\documentclass[12pt]{amsart}
\usepackage{amscd,amsmath,amssymb,amsfonts}
\usepackage[cmtip, all]{xy}

\newtheorem{thm}{Theorem}[section]
\newtheorem{prop}[thm]{Proposition}
\newtheorem{lem}[thm]{Lemma}
\newtheorem{cor}[thm]{Corollary}
\newtheorem{conj}[thm]{Conjecture}
\theoremstyle{definition}

\newtheorem{defn}[thm]{Definition}
\theoremstyle{remark}
\newtheorem{remk}[thm]{Remark}
\newtheorem{remks}[thm]{Remarks}

\newtheorem{exm}[thm]{Example}
\newtheorem{exms}[thm]{Examples}
\newtheorem{notat}[thm]{Notation}
\numberwithin{equation}{section}

{\hfill$\square$\end{defn}}

{\hfill$\square$\end{remk}}

{\hfill$\square$\end{remks}}

{\hfill$\square$\end{exm}}

{\hfill$\square$\end{exms}}

{\hfill$\square$\end{notat}}

\newcommand{\sC}{{\mathcal C}}
\newcommand{\sD}{{\mathcal D}}
\newcommand{\sE}{{\mathcal E}}
\newcommand{\sF}{{\mathcal F}}
\newcommand{\sG}{{\mathcal G}}
\newcommand{\sH}{{\mathcal H}}
\newcommand{\sI}{{\mathcal I}}

\newcommand{\sK}{{\mathcal K}}

\newcommand{\sO}{{\mathcal O}}

\newcommand{\A}{{\mathbb A}}

\newcommand{\C}{{\mathbb C}}

\renewcommand{\H}{{\mathbb H}}

\renewcommand{\P}{{\mathbb P}}
\newcommand{\Q}{{\mathbb Q}}
\newcommand{\R}{{\mathbb R}}

\newcommand{\Z}{{\mathbb Z}}

\newcommand{\surj}{\twoheadrightarrow}
\newcommand{\inj}{\hookrightarrow}

\newcommand{\Zar}{{\text{\rm Zar}}}

\newcommand{\Sch}{{\operatorname{\mathbf{Sch}}}}

\newcommand{\wt}{\widetilde}

\newcommand{\ds}{{/\kern-3pt/}}

\newcommand{\un}{\underline}
\newcommand{\ov}{\overline}

\begin{document}
\title{A cdh approach to zero-cycles on singular varieties}
\author{Amalendu Krishna}
\address{School of Mathematics, Tata Institute of Fundamental Research,  
1 Homi Bhabha Road, Colaba, Mumbai, India}
\email{amal@math.tifr.res.in}

\baselineskip=10pt 
  

%\subjclass[2010]{Primary 14C25; Secondary 19E15}
\maketitle

\begin{abstract}
We study the Chow group of zero-cycles on singular varieties from
the perspective of the $cdh$ topology and $KH$-theory. We define the $cdh$ 
versions of the Chow groups of zero-cycles and albanese maps and formulate
conjectures about the Roitman torsion theorem and ``finite-dimensionality''
of the Chow group in this set up. We prove these conjectures for surfaces. 
We compare the $cdh$ analogue of the Chow group of zero-cycles on
a projective variety with the Chow group of its desingularization. 
This is used to prove a finite-dimensionality result for this version of the 
Chow group for a normal projective threefold.

In order to apply the $cdh$ techniques to study the known Chow group of 
zero-cycles $CH^d(X)$ on a singular scheme $X$ of dimension $d$, we show 
under certain cohomological conditions that 
there is a natural injection $CH^d(X) \inj F^dKH_0(X)$ up to torsion, if $X$ 
has only isolated singularities. This reduces the finite-dimensionality
problem for $CH^d(X)$ to its $cdh$ analogue.  

As a byproduct of this, we prove some finite-dimensionality result
for the Chow group of zero-cycles $CH^3(X)$ on a threefold $X$ with isolated 
singularities, assuming the same for smooth threefolds.   
\end{abstract}

%\tableofcontents

\section{Introduction}
The purpose of this work is to understand the Chow groups of zero-dimensional
algebraic cycles on singular varieties from the perspective of
$cdh$ topology on schemes, which was developed by Voevodsky \cite{SuslinV}
in his study of the triangulated category of motives and motivic cohomology.
The Chow groups of algebraic cycles on smooth varieties have been studied for
a long time and they were first used in a nontrivial way by
Grothendieck, when he proved the Riemann-Roch theorem for the $K$-group
of algebraic vector bundles on such varieties. Ever since, these groups of
algebraic cycles on smooth varieties and the Abel-Jacobi maps from
them into the intermediate Jacobians have have gone through a phenomenal
developments which one hopes, would culminate in the proof of the famous
general Hodge conjecture regarding them. One also knows now that the Chow groups
of algebraic cycles and their generalization by Bloch \cite{Bloch} in terms 
of the higher Chow groups provide a theory of motivic cohomology, which
describes the algebraic $K$-theory of vector bundles on smooth varieties.  

On the other hand, almost nothing is known about the
existence of such a theory of the Chow groups of algebraic cycles on 
singular varieties. It still looks a difficult question whether there is a
motivic cohomology theory which could describe the algebraic $K$-theory
of vector bundles on such schemes. In order to describe the lowest
part of ``coniveau filtration'' on the Grothendieck group of vector bundles on
a singular scheme $X$ of dimension $d$, Levine and Weibel \cite{LevineW}
invented the Chow group of zero cycles, often denoted by $CH^d(X)$. 

This group has since been extensively studied by Levine, Srinivas and many 
others. However, except for the case of curves ({\sl cf.} \cite{LevineW})
and normal surfaces ({\sl cf.} \cite{KrishnaS}), this Chow group 
is still far from being fully understood, although many properties, 
such as the singular analogue of the Roitman's torsion theorem and
Mumford's infinite-dimensionality theorem have been now been proven
({\sl cf.} \cite{Levine2}, \cite{BiswasS}, \cite{Srinivas2},
\cite{Srinivas1}, \cite{ESV}).
In higher dimensions, a formula for Chow groups of zero cycles on
a variety $X$ with only Cohen-Macaulay isolated singularities in 
terms of the analogous group for a resolution of singularities of $X$ was
recently obtained in \cite{Krishna2} (see also \cite{Krishna3}).
The main motivation behind the lookout for such a formula is the hope that 
many questions about the Chow groups of zero cycles on such singular varieties
could be reduced to answering the similar questions about the zero cycles
on smooth varieties, which are supposedly better understood.

One such question, which is the prime focus of this paper on the application
side, is the singular analogue of the generalized Bloch conjecture. To state 
this, recall that if $X$ is a smooth projective variety of dimension $d$ over
an algebraically closed field $k$, 
then the natural map from the variety to its {\sl albanese} variety
defines a surjective group homomorphism
\begin{equation}\label{eqn:SALB}
c_{0,X}: CH^d(X)_{{\rm deg} 0} \to J^d(X),
\end{equation}
where $J^d(X)$ is the albanese variety of $X$. In fact, one knows that
$J^d(X)$ is a {\sl universal regular quotient} of $CH^d(X)_{{\rm deg} 0}$.
One says that $CH^d(X)$ is {\sl finite-dimensional}, if the map
$c_{0,X}$ is an isomorphism, which means that the ``connected component''
of $CH^d(X)$ is parameterized by a projective variety.
It was conjectured by Bloch \cite{Bloch*} that a smooth projective surface $X$ 
with $H^2_{Zar}(X, \sO_X) = 0$, is finite-dimensional. This conjecture is open
till date though it has been proven for surfaces which are not of 
general type \cite{BKL}. The following is the generalized form of Bloch's
conjecture in higher dimension.
\begin{conj}{$({\bf GBC})$}\label{conj:GBC}
Let $X$ be a smooth projective variety of dimension $d$ over $k$ such
that $H^d_{Zar}(X, \Omega^i_X) = 0$ for $0 \le i \le d-2$. Then
$CH^d(X)$ is finite-dimensional.
\end{conj}
We refer the reader to \cite{Voisin} for more 
detail on this conjecture and its relations with several outstanding 
conjectures on algebraic cycles. We shall say that $GBC(X)$ holds if 
the generalized Bloch conjecture holds for the variety $X$. For $d \ge 2$,
we shall say that $GBC(d)$ holds, if this conjecture holds for all
smooth projective varieties of dimension up to $d$.

Now suppose $X$ is a singular projective variety.
It is shown in \cite{ESV} that there is a smooth commutative algebraic group 
$J^d_*(X)$ over $k$ and a surjective group homomorphism 
\begin{equation}\label{eqn:SALBS}
c_{0,X} : {CH^d(X)}_{{\rm deg} 0} \to J^d_*(X),
\end{equation}
which is a universal regular quotient in the category of {\sl regular}
maps from $CH^d(X)_{{\rm deg} 0}$ to smooth commutative algebraic groups.
The map $c_{0,X}$ is called the {\sl albanese} map. 
As in the case of smooth varieties, one says that $CH^d(X)$ is
{\sl finite-dimensional} if the albanese map is an isomorphism.
One has then the following singular analogue of the generalized Bloch
conjecture.
\begin{conj}\label{conj:GBCS}
Let $X$ be a projective variety of dimension $d$ over $k$ such
that $H^d_{Zar}(X, \Omega^i_X) = 0$ for $0 \le i \le d-2$. Then
$CH^d(X)$ is finite-dimensional.
\end{conj}
 
In the current work, we propose a new direction for studying the Chow group
of zero-cycles on singular varieties. This involves approaching the
problem from the perspective of the $cdh$ topology on singular schemes.
This topology was invented by Voevodsky \cite{SuslinV} in his study of
motives over a field of characteristic zero. This topology although
does not do anything extraordinary on smooth schemes, it has
proven to be of immense help in solving many $K$-theoretic problems
about singular schemes which are otherwise known for smooth schemes.
This has been ubiquitously used for example, in \cite{CAnnals}, \cite{CJams} 
and \cite{C-cone}, to solve some very important problems on the algebraic
$K$-theory of singular schemes. 

This work is an attempt to study the questions about algebraic cycles on
singular schemes by relating them with the homotopy invariant $K$-theory
and the $cdh$ cohomology of the $K$-theory sheaves. One of the main
tools for making this program work is the recent result 
\cite[Theorem~1.6]{CJams}, which roughly speaking,
estimates the difference between the algebraic $K$-theory and the homotopy
$K$-theory of singular schemes in terms of their Hochschild and cyclic
homology. Needless to say, these Hochschild and cyclic homology
are much simpler objects to deal with. The hope is that this,
together with many known results about the $KH$-theory, will
give a new direction in constructing the motivic cohomology
of singular schemes.

We now describe the various results of this paper in some detail.
Our results are based on exploiting the relation between the
Zariski and the $cdh$ cohomology of quasi-coherent and $\sK$-sheaves.
Towards this, we describe the Zariski cohomology of the 
K{\"a}hler differentials on normal crossing schemes in terms of their 
$cdh$ cohomology in Section~\ref{section:SNC*}. As mentioned above, the 
difference between the algebraic and the homotopy $K$-groups can be measured 
in terms of the cyclic homology of schemes. In Section~\ref{section:ZCSS}, we 
describe the relevant part of the cyclic homology of schemes with isolated
singularities, in terms of the $cdh$ cohomology of the K{\"a}hler 
differentials.
This is done by using the resolution of singularities, the cyclic
homology of smooth schemes and the computations of these groups
for normal crossing schemes.

In Section~\ref{section:KKH}, we study various filtrations on the
algebraic $K$-theory of singular schemes. In order to understand
the behavior of these filtrations under the push-forward map
on the algebraic $K$-theory, we establish a Riemann-Roch theorem 
({\sl cf.} Theorem~\ref{thm:GRR}) for
the local complete intersection morphism of singular schemes. This is
an extension of the similar result of Grothendieck et al. for $K_0$ to
higher $K$-theory. The proof involves exploiting the axiomatic approach of
Fulton and Lang \cite{FultonL} to the Grothendieck's Riemann-Roch theorem,
and the study of filtrations on higher $K$-theory by Soul{\'e} \cite{Soule}. 
This is then
used to show in Corollary~\ref{cor:compare} that on a scheme with only
isolated singularities, the Chow group of zero-cycles coincides with
the lowest level of the gamma filtration on $K_0$.
The main result here is Theorem~\ref{thm:main1}, where we show that
under certain vanishing of the cohomology of differential forms, 
the Chow group of zero-cycles can be embedded inside the lowest
level of the Brown filtration on $KH_0$. This brings the study of the
known Chow group of zero-cycles to the $cdh$ world.
     
The Chow group of algebraic cycles on schemes are often studied
via the Abel-Jacobi maps from these groups to the various intermediate
Jacobians. In order to facilitate this technique to work through
in the $cdh$ approach, we study these objects in Section~\ref{section:HODGET}.
We interpret the Hodge filtration on the
analytic cohomology of singular complex varieties in terms of the $cdh$
cohomology of differential forms. This follows from a $cdh$-descent theorem for
the Du Bois complex on these schemes. 

In Section~\ref{section:CHERNS},
we construct the theory of Chern classes on the $KH$-theory with values in
the Deligne cohomology. These Chern classes are compatible with the
similar Chern classes on the algebraic $K$-theory. This construction is
achieved by comparing the $KH$-theory with the descent $K$-theory of
\cite{PP}. We use these Chern classes to define the $cdh$ version of the
{\sl albanese} maps of ~\eqref{eqn:SALBS}. This facilitates the formulation 
of two conjectures about the $cdh$ analogues of the Roitman torsion theorem and
the finite-dimensionality problem for the Chow group of zero-cycles.

We prove our main result about the $cdh$ version of the Chow group of 
zero-cycles in Theorem~\ref{thm:Main22}. 
As an application, we prove Conjectures~\ref{conj:FDCCR}
and ~\ref{conj:FDCC} for surfaces in Section~\ref{section:Surf}.
As further applications, we prove the finite-dimensionality results  
Theorem~\ref{thm:3fold-main} and Corollary~\ref{cor:3fold-main2}
for the Chow groups of zero-cycles on normal projective three-folds
in Section~\ref{section:3fold}. 

We end this section with few comments. As said before, this work is an 
attempt to study the Chow group of algebraic cycles on singular schemes using
the $cdh$ topology on such schemes. In most of the proofs,
we have put extra conditions on the nature of the singularity,
which are seemingly needless. The results here lead one to many
other related questions on the Chow groups of singular schemes. 
The applications in this paper are mostly focussed on the finite-dimensionality
question for the Chow group of zero-cycles and its $cdh$ analogue.
In the sequel \cite{Krishna4}, we exploit and refine the techniques of this
paper to obtain generalizations of Theorem~\ref{thm:main1} for 
varieties with arbitrary singularities. This in turn will be used to 
obtain formulas for the known Chow group of zero-cycles 
in terms of the Chow group of a resolution of singularities.

\section{Recollections and preliminary results}\label{section:prelims}
In this paper, we fix once and for all, a ground
field $k$ of characteristic zero. We shall assume $k$ to be algebraically
closed although many of the results do not need this extra assumption,
as the reader will observe. We make this assumption mainly because our
applications will be over the field $\C$ of complex numbers.
Let $\Sch/k$ denote the category of all separated schemes of finite type 
over $k$ and let $Sm/k$ denote the category of smooth schemes of finite type 
over $k$. For $X \in \Sch/k$, $X^N$ will often denote the normalization of
$X_{\rm red}$. Recall from \cite{Haes} that an {\sl abstract blow-up} is a
Cartesian square 
\begin{equation}\label{eqn:ABLP}
\xymatrix{
Y' \ar[r]^{i'} \ar[d]_{f'} & X' \ar[d]^{f} \\
Y \ar[r]_{i} & X}
\end{equation}
in $\Sch/k$ where $i$ is a closed immersion, $f$ is proper and 
$(X' - Y')_{\rm red} \to (X-Y)_{\rm red}$ is an isomorphism. An 
{\sl elementary Nisnevich square} is a Cartesian square 
~\eqref{eqn:ABLP} where $f$ is
{\'e}tale, $i$ is an open immersion and $(X'-Y') \to (X-Y)$ is an
isomorphism. The {\sl cdh} topology on $\Sch/k$ is generated by
covers $\{X' \to X\}$ given by the above abstract blow-ups and elementary 
Nisnevich squares. 

For a presheaf of spectra $\sE$ on $\Sch/k$, let $\H_{cdh}(-, \sE)$
denote its fibrant replacement in the $cdh$ topology. We shall denote the
nonconnective $K$-theory spectra and its homotopy invariant counterpart
({\sl cf.} \cite{WeibelH}) on $\Sch/k$ by $\sK$ and ${\sK}{\sH}$ respectively.
For a subfield $F \subseteq k$, let ${\sH}{\sH}(/F)$ and ${\sH}{\sC}(/F)$ 
denote the presheaves of Eilenberg-Mac Lane spectra of Hochschild and cyclic 
homology respectively on $\Sch/k$ ({\sl cf.} \cite[Section~2]{CAnnals}) over 
$F$. If the field $F$ is not indicated, we shall assume these homology
to be taken over the field $\Q$ of rational numbers. For any
$X$ in $\Sch/k$, we let $HH_{i}(X/F)$ to be ${\pi}_i {\sH}{\sH}(X/F)$. One
defines the cyclic homology of schemes in a similar way. 
We shall also write ${\pi}_i \sK(X)$ (resp ${\pi}_i {\sK}{\sH}(X)$)
as $K_i(X)$ (resp $KH_i(X)$). We recall from 
\cite[Proposition~2.1]{CJams} that the Hochschild, cyclic homology and their
$cdh$-fibrant replacements have a gamma filtrations (Hodge filtration)
which give the canonical $\lambda$-decompositions
\begin{equation}\label{eqn:lambda0}
HH_n(X/F) \cong \ \stackrel{}{\underset i{\oplus}} \ HH^{(i)}_n(X/F),
\end{equation}
\[
HC_n(X/F) \cong \ \stackrel{}{\underset i{\oplus}} \ HC^{(i)}_n(X/F),  
\]
\[
\H^n_{cdh}(X, \mathcal{HH}) \ \cong \
\stackrel{}{\underset i{\oplus}} \ \H^n_{cdh}(X, \mathcal{HH}^{(i)})
\ {\rm and} 
\]
\[
\H^n_{cdh}(X, \mathcal{HC}) \ \cong \
\stackrel{}{\underset i{\oplus}} \ \H^n_{cdh}(X, \mathcal{HC}^{(i)}).
\]
Moreover, the natural map $\mathcal{HH} \to \H_{cdh}(-, \mathcal{HH})$
preserves these decompositions. The same holds for the cyclic homology
and negative cyclic homology.

For a presheaf of spectra $\sE$, let $\widetilde{\sE}$ denote the homotopy 
fiber
of the natural map of spectra $\sE \to \H_{cdh}(-, \sE)$. 
\begin{thm}{(\cite[Theorem~1.6]{CJams})}\label{thm:fiber0}
For any scheme $X$ in $\Sch/k$, the Chern character map $\sK \to 
\mathcal{HN}$ from the $K$-theory to the negative cyclic homology spectra
induces a natural weak equivalence
\[
\widetilde{\sK}(X) \xrightarrow{\cong} \widetilde{\mathcal{HN}}(X) 
\overset{\cong}{\leftarrow}
{\Omega}^{-1} \widetilde{\mathcal{HC}}(X).
\]
\end{thm}
\begin{cor}\label{cor:fiber1}
There is a natural fibration sequence of spectra
\[
{\Omega}^{-1}\widetilde{\mathcal{HC}}(X) \to \sK(X) \to {\sK}{\sH}(X).
\]
\end{cor}
\begin{proof} This follows from Theorem~\ref{thm:fiber0} and the weak
equivalence $\H_{cdh}(-, \sK) \xrightarrow{\cong} {\sK}{\sH}$
as shown in \cite[Theorem~6.4]{Haes}.
\end{proof}

We shall write $a : {\left(\Sch/k\right)}_{cdh} \to 
{\left(\Sch/k\right)}_{Zar}$ for the natural morphism from the $cdh$ to
the Zariski site of $\Sch/k$. For a Zariski sheaf $\sF$, let ${\sF}_{cdh}$
denote its associated $cdh$-sheaf $a^*\sF$ and let $H^*_{cdh}(X, \sF)$
denote the $cdh$-cohomology of ${\sF}_{cdh}$. For any subfield $F \subseteq k$
and for $i \ge 0$, let ${\Omega}^i_{/F}$ denote the presheaf $X \mapsto
{\Omega}^i_{X/F}$ of K{\"a}hler differentials. Let ${\Omega}^{\ge i}_{X/F}$
denote the brutal truncations of the algebraic de Rham complex 
${\Omega}^{\bullet}_{X/F}$ of $X$.  It is known ({\sl cf.} \cite{SuslinV},
\cite[Exp. VI, 2.11, 5.2]{Grothendieck}) that the $cdh$-cohomology commutes 
the filtered direct limits of sheaves. In particular, for a sheaf of 
$k$-modules $\sF$
and a $k$-vector space $V$, one has $H^n_{cdh}(X, \sF {\otimes}_k V)
\cong H^n_{cdh}(X, \sF) {\otimes}_k V$. We shall often use this fact for the 
sheaves of the form ${\Omega}^i_{k/{\Q}} {\otimes}_k {\Omega}^j_{X/k}$.

We also recollect here a few preliminary results that will be used often.
The case $i = 0$ of the following result was proven in
\cite[Lemma~6.5]{CAnnals}. The general case is a simple consequence of
\cite[Theorem~2.4, Lemma~2.8]{CJams}.  
We shall use such a result for comparing the
Zariski and the $cdh$-cohomology of sheaves of differential forms.
\begin{lem}\label{lem:direct-image}
For a $k$-scheme (i.e., a scheme in $\Sch/k$), the complex of Zariski
sheaves $Ra_*a^*{\Omega}^i_{X/F}$ (resp  $Ra_*a^*{\Omega}^i_{X/k}$)
has quasi-coherent (resp coherent) cohomology sheaves
for any subfield $F \subseteq k$ and any $i \ge 0$. 
\end{lem}

\begin{lem}\label{lem:folklore1}
Let $\sF \xrightarrow{\phi} \sG$ be a morphism of Zariski sheaves on a 
$k$-scheme $X$ such that $\phi$ is an isomorphism outside a closed subscheme
$Z \overset{i}{\inj} X$ of dimension $d$. Then $H^{d+1}\left(X, \sF\right) 
\surj H^{d+1}\left(X, \sG\right)$ and $H^i(X, \sF) \xrightarrow{\cong}
H^i(X. \sG)$ for $i \ge d+2$.
\end{lem}
\begin{proof}
It suffices using the standard long exact cohomology sequences, to show that
if $\sF$ is a Zariski sheaf on $X$ such that $\sF|_U = 0$, where 
$U \overset{j}{\inj} X$ is the complement of $Z$, then $H^i(X, \sF) = 0$ for 
$i \ge d+1$. We have the short exact sequence ({\sl cf.} \cite[Exercise~1.19,
Chapter II]{Hart})
\[
0 \to j_{!} j^* \sF \to \sF \to i_*(\sF|_Z) \to 0.
\]
Our assumption then implies that 
$H^i(X, \sF) \xrightarrow{\cong} H^i\left(X, i_*(\sF|_Z)\right)$ for all $i$.
But the last term is same as $H^i(Z, \sF|_Z)$ by \cite[Lemma~III.2.10]{Hart}.
The proof now follows from the Grothendieck vanishing theorem on $Z$.
\end{proof}  

\begin{lem}\label{lem:folklore2}
Let $X$ be a $k$-scheme of dimension $d$. Then for either of Zariski and
$cdh$-sites, the following hold for any subfield $F \subseteq k$ and any 
$j \ge 0$. \\
$(i)$ $H^d(X, \Omega^{j-1}_{X/F}) \to H^d(X, \Omega^j_{X/F}) \to
\H^{d+j}(X, \Omega^{\le j}_{X/F}) \to 0$ \\
is exact. \\
$(ii)$ $\H^{i}(X, \Omega^{\le j}_{X/F}) = 0$ for  $i\ge j+d+1$.
\end{lem}
\begin{proof}
This is a simple exercise in the hypercohomology using the fact that
the Zariski or the $cdh$ cohomological dimension of $X$ is $d$ ({\sl cf.}
\cite{SuslinV}). We give a sketch. For $j = 0$, this follows
from the above fact. So assume that the result holds for $0 \le j' \le j-1$ and
$j \ge 1$. The long exact cohomology sequence for the exact sequence
\[
0 \to \Omega^j_{X/F}[-j] \to \Omega^{\le j}_{X/F} \to \Omega^{\le j-1}_{X/F}
\to 0
\]
gives a diagram
\[
\xymatrix@C.4pc{
H^d(X, \Omega^{j-1}_{X/F}) \ar[d] \ar[dr] & & & \\ 
\H^{d+j-1}(X, \Omega^{\le j-1}_{X/F}) \ar[r] & 
H^d(X, \Omega^j_{X/F}) \ar[r] & \H^{d+j}(X, \Omega^{\le j}_{X/F}) \ar[r] &
\H^{d+j}(X, \Omega^{\le j-1}_{X/F}).}
\]
The vertical map is surjective and the last term on the bottom exact
sequence is zero by induction. This proves the first exact sequence of the
lemma. Finally, the exact sequence 
\[
H^{i-j}(X, \Omega^j_{X/F}) \to \H^{i}(X, \Omega^{\le j}_{X/F}) \to 
\H^{i}(X, \Omega^{\le j-1}_{X/F})
\]
for $i \ge d+j+1$ and the induction prove the second part of the lemma.
\end{proof} 

\section{Zariski and cdh cohomology on normal crossing schemes}
\label{section:SNC*}
In this section, we compare the Zariski and the $cdh$ cohomology of
sheaves of differential forms on normal crossing schemes. We begin with the
following local result.

For a finite morphism $f : Y \to X$ of $k$-schemes
and for a subfield $F \subseteq k$, let 
\[
{\Omega}^i_{{(X,Y)}/F} := {\rm Ker}\left({\Omega}^i_{X/F} \to
f_*\left({\Omega}^i_{Y/F}\right)\right) \ \ {\rm for} \ i \ge 0.
\]
If $X = {\rm Spec}(A)$ and $Y$ is the closed subscheme of $X$ defined by an
ideal $I$, we shall often write the above as ${\Omega}^i_{{(A,I)}/F}$.

\begin{prop}\label{prop:local}
Let $A$ be a reduced and essentially of finite type $k$-algebra. Assume that 
the normalization map $f : A \to B$ of $A$ is unramified. Let $I \subset A$
be a conducting ideal for the normalization. Then the 
natural map ${\Omega}^i_{{(A,I)}/F} \to {\Omega}^i_{{(B,I)}/F}$ is
surjective for any $i \ge 0$ and any subfield $F \subseteq k$.
\end{prop}  
\begin{proof} 
Let $A' = A/I$ and $B' = B/I$. We prove the proposition by the induction on 
$i \ge 0$. For $i = 0$, we have 
${\Omega}^0_{{(A,I)}/F} = I = {\Omega}^0_{{(B,I)}/F}$. For $i = 1$, we use the
commutative diagram of exact sequences
\[
\xymatrix@C.5pc{
HH^F_1(A,I) \ar[r] \ar[d] & \Omega^1_{A/F} \ar[r] \ar[d] & \Omega^1_{A'/F} 
\ar[d] \ar[r] & 0 \\
HH^F_1(B,I) \ar[r] \ar[d] & \Omega^1_{B/F} \ar[r] & \Omega^1_{A'/F}   
\ar[r] & 0 \\
HH^F_0(A,B,I), & & &}
\]
where $HH^F_i(A,B,I)$ is the double relative Hochschild homology
({\sl cf.} \cite{Loday}). Since $HH^F_0(A,B,I) \cong I {\otimes}_A
{\Omega}^1_{B/A}$ by \cite[Theorem~3.4]{CGW} and since $A \to B$ is 
unramified, we see that $HH^F_0(A,B,I) = 0$ and in particular, the top 
vertical map on the extreme left in the above diagram is surjective.
A diagram chase shows that ${\Omega}^1_{{(A,I)}/F} \surj
{\Omega}^1_{{(B,I)}/F}$.

Suppose now that $i \ge 2$ and the result holds for all $0 \le j \le i-1$.
Since $\Omega^1_{B/A} = 0$, the first fundamental exact sequence of
K{\"a}hler differentials implies that 
\begin{equation}\label{eqn:local1}
\Omega^1_{B/F} = B\Omega^1_{A/F}.
\end{equation}
Moreover, ${\Omega}^i_{{(A,I)}/F} \subseteq \Omega^i_{A/F}$ is generated
as a $A$-module by the exterior forms of the type $a_0 da_1 \wedge \cdots
\wedge da_i$, where $a_p \in A$ for all $p$ and $a_p \in I$ for some $p$.
This can be directly checked by the universal property of K{\"a}hler
differentials ({\sl cf.} \cite[Lemma~4.1]{Krishna0}). The same holds for
${\Omega}^i_{{(B,I)}/F}$. Thus
\begin{equation}\label{eqn:local2}
{\Omega}^i_{{(A,I)}/F} =   \stackrel{}{\underset {{\rm some} \ 
a_p \in I}{\sum}} A a_0 da_1 \wedge \cdots \wedge da_i.
\end{equation}
\[
{\Omega}^i_{{(B,I)}/F} =  \stackrel{}{\underset {{\rm some} \ 
b_p \in I}{\sum}} B b_0 db_1 \wedge \cdots \wedge db_i
\]
\[
\hspace*{2cm} = \stackrel{}{\underset {{\rm some} \ 
b_p \in I}{\sum}} A b_0 db_1 \wedge \cdots \wedge db_i.
\]
Hence, it suffices to show that 
\begin{equation}\label{eqn:local3}
\beta = b_0 db_1 \wedge \cdots \wedge db_i \in 
{\rm Image}\left({\Omega}^i_{{(A,I)}/F} \to {\Omega}^i_{{(B,I)}/F}\right).
\end{equation}
By permuting the orders of differentials (which only changes the sign), can
assume that $ b_p \in I$ for some $p \le i-1$. Then we have 
${\beta}' = b_0 db_1 \wedge \cdots \wedge db_{i-1} \in 
{\Omega}^{i-1}_{{(B,I)}/F}$ by \eqref{eqn:local2}, and $\beta = {\beta}' 
\wedge db_i$.

By induction, we see that ${\beta}' \in 
{\rm Image}\left({\Omega}^{i-1}_{{(A,I)}/F} \to 
{\Omega}^{i-1}_{{(B,I)}/F}\right)$.
This implies from \eqref{eqn:local2} again that 
\[
{\beta}' = \stackrel{}{\underset {{\rm some} \ 
a_p \in I}{\sum}} a_0 da_1 \wedge \cdots \wedge da_{i-1},
\]
with $a_p \in A$ for $0 \le p \le i-1$. So can assume that 
${\beta}' = a_0 da_1 \wedge \cdots \wedge da_{i-1}$. We then have 
\[
\beta = {\beta}' \wedge db_i = a_0 da_1 \wedge \cdots \wedge da_{i-1}
\wedge db_i.
\]
If $a_0 \in I$, then the case $i = 1$ implies that $a_0 db_i \in
{\Omega}^{1}_{{(A,I)}/F}$ and hence $\beta \in {\Omega}^{i}_{{(A,I)}/F}$.

So we suppose that $a_p \in I$ for some $1 \le p \le i-1$. We can again
assume that $a_1 \in I$. It follows from ~\eqref{eqn:local1} that 
$da_1 \wedge db_i$ is of the form $\stackrel{}{\underset p {\sum}} 
 b'_pda_1 \wedge d{\alpha}_p$, where $\alpha_p \in A$ and $b'_p \in B$ for 
all $p$. Hence we can assume that
\[
\beta = \left(bda_1 \wedge d{\alpha}_i\right) \wedge \left(a_0 da_2 \wedge
\cdots \wedge da_{i-1}\right),
\]
where $\alpha_i \in A$ and $b \in B$. In particular, we have
$bda_1 \in {\Omega}^{1}_{{(B,I)}/F}$ and hence in the image of
${\Omega}^{1}_{{(A,I)}/F}$ by the case $i =1$. This in turn implies that
up to sign
\[
\beta = bda_1 \wedge \left(da_2 \wedge \cdots \wedge da_{i-1}
\wedge d{\alpha}_i\right)
\]
\[
\hspace*{3cm} \in {\Omega}^{1}_{{(A,I)}/F} \wedge 
{\Omega}^{i-1}_{A/F} \subseteq {\Omega}^{i}_{{(A,I)}/F}.
\]
This proves ~\eqref{eqn:local3} and finishes the proof of the proposition.
\end{proof}

\begin{lem}\label{lem:SNC-curve}
Let $E$ be a reduced seminormal curve which is affine and essentially of finite
type over $k$. Then the following hold for any subfield $F \subseteq k$
and any $i \ge 0$.
\[
H^0_{Zar}(E, \sO_E) \xrightarrow{\cong} H^0_{cdh}(E, \sO_E)
\]
\[
H^0_{Zar}(X, \Omega^i_{E/F}) \surj H^0_{cdh}(X, \Omega^i_{E/F}) \ {\rm and}
\]
\[
H^j_{cdh}(E, \Omega^i_{E/F}) = 0 \ {\rm for} \ j > 0.
\]
\end{lem}
\begin{proof}
Let $\ov{E} \xrightarrow{f} E$ denote the normalization of $E$. Let 
$S = E_{\rm sing}$ be the reduced singular locus of $E$ and let $\ov{S} =
f^{-1}(S)$. Since $E$ is a seminormal curve, the map $f$ is unramified and 
hence the proof of Proposition~\ref{prop:local} implies that there is an 
exact sequence
\begin{equation}\label{eqn:SNC-curve2}
H^0_{Zar}(E, \Omega^i_{E/F}) \to H^0_{Zar}(\ov{E}, \Omega^i_{{\ov{E}}/F})
\oplus H^0_{Zar}(S, \Omega^i_{S/F}) \to 
H^0_{Zar}(\ov{S}, \Omega^i_{{\ov{S}}/F}) \to 0.
\end{equation}
for $0 \le i \le 1$ and the first arrow from the left is injective for $i =0$. 
We thus have the following commutative diagram of exact sequences for 
$0 \le i \le 1$.
\begin{equation}\label{eqn:SNC-curve1}
\xymatrix@C.3pc{
& H^0_{Zar}(E, \Omega^i_{E/F}) \ar[d] \ar[r] & 
{H^0_{Zar}(\ov{E}, \Omega^i_{{\ov{E}}/F})
\oplus H^0_{Zar}(S, \Omega^i_{S/F})} \ar[r] \ar[d] & 
H^0_{Zar}(\ov{S}, \Omega^i_{{\ov{S}}/F})
\ar[d] \ar[r] & 0 \\
0 \ar[r] & H^0_{cdh}(E, \Omega^i_{E/F}) \ar[r] & 
{H^0_{cdh}(\ov{E}, \Omega^i_{{\ov{E}}/F})
\oplus H^0_{cdh}(S, \Omega^i_{S/F})} \ar[r] & 
H^0_{cdh}(\ov{S}, \Omega^i_{{\ov{S}}/F}) \ar[r] &  
H^1_{cdh}(E, \Omega^i_{E/F}).}
\end{equation}
The top sequence is exact by ~\eqref{eqn:SNC-curve2}
and the bottom sequence is the Mayer-Vietoris exact sequence for 
the $cdh$ cohomology ({\sl cf.} \cite[Theorem~2.7]{CJams}). The smoothness of
$\ov{E}, S$ and $\ov{S}$ and \cite[Corollary~2.5]{CJams} imply that the
middle and the right vertical maps are isomorphisms. For the same reason,
the assertion of the lemma holds for these schemes. 
A diagram chase now proves the lemma for $0 \le i \le 1$.
Furthermore, it also follows that for $i \ge 2$,
\[
\Omega^i_{k/F} {\otimes}_k
H^0_{cdh}(\ov{E}, \sO_{{\ov{E}}/F}) \surj  \Omega^i_{k/F} {\otimes}_k
H^0_{cdh}(\ov{S}, \sO_{{\ov{S}}/F}) \xrightarrow{\cong}
H^0_{cdh}(\ov{S}, \Omega^i_{{\ov{S}}/F}).
\]
In particular, $H^j_{cdh}(E, \Omega^i_{E/F}) = 0$ for $i \ge 0$ and $j \ge 1$.

We now only need to show that 
\begin{equation}\label{eqn:higher}
H^0_{Zar}(E, \Omega^i_{E/F}) \surj
H^0_{cdh}(E, \Omega^i_{E/F}) \ {\rm for} \ i \ge 2.
\end{equation} 
For this, we first assume that $F = k$. Then $\Omega^1_{S/k}$ and 
$\Omega^1_{{\ov{S}}/k}$ vanish and hence the top exact sequence in the above
diagram shows that $\Omega^1_{E/k} \surj \Omega^1_{{\ov{E}}/k}$. This in turn
gives a commutative diagram
\[
\xymatrix@C.5pc{
\Omega^i_{E/k} \ar@{->>}[r] \ar[d] & \Omega^i_{{\ov{E}}/k} \ar[d]^{\cong} \\
H^0_{cdh}(E, \Omega^i_{E/k}) \ar[r]^{\cong} &
H^0_{cdh}(\ov{E}, \Omega^i_{{\ov{E}}/k}),}
\]
where the the isomorphism of the right vertical map and the bottom horizontal
map follows from the smoothness of $\ov{E}, S$ and $\ov{S}$ and the
Mayer-Vietoris exact sequence for the $cdh$ cohomology. This shows the
desired surjectivity of the left vertical map. For any $F \subseteq k$,
we note that $\Omega^i_{E/F}$ has a finite decreasing filtration
${\{F^j\Omega^i_{E/F}\}}_{0 \le j \le i}$ such that there is a surjection
\[
\Omega^j_{k/F} \otimes_k \Omega^{i-j}_{E/k} \surj \frac {F^j\Omega^i_{E/F}}
{F^{j+1}\Omega^i_{E/F}}
\]
which is an isomorphism on the smooth locus of $E$. Now ~\eqref{eqn:higher}
follows by an easy induction on $i$ and $j$.
\end{proof}
The above lemma is a special case of the following more general result.

\begin{prop}\label{prop:SNC-local}
Let $E \overset{i}{\inj} X$ be a strict normal crossing divisor on a smooth
$k$-scheme $X$ such that $X$ is affine and essentially of finite type 
over $k$ and its dimension is $d+1$. Then the 
following hold for any subfield $F \subseteq k$ and any $i \ge 0$.
\[
H^0_{Zar}(E, \sO_E) \xrightarrow{\cong} H^0_{cdh}(E, \sO_E)
\]
\[
H^0_{Zar}(X, \Omega^i_{E/F}) \surj H^0_{cdh}(X, \Omega^i_{E/F}) \ {\rm and}
\]
\[
H^j_{cdh}(E, \Omega^i_{E/F}) = 0 \ {\rm for} \ j > 0.
\]
\end{prop}
\begin{proof} Let $A$ be the coordinate ring of $E$ and let $A \xrightarrow{f}
B$ be the normalization of $A$. Let $I \subset A$ be the reduced conducting
ideal for the normalization. In particular, $I$ is the ideal of the closed
subscheme $E_{\rm sing}$. We prove the proposition by induction on $d$.
The case $d =1$ is shown in Lemma~\ref{lem:SNC-curve}. So we assume 
$d \ge 2$ and that the result holds for the normal crossing divisors of
dimension up to $d-1$.

We observe that since $E$ is a strict normal crossing divisor on $X$
which is smooth, the normalization is simply the disjoint union of the
irreducible components of $E$. Moreover, the map $f$ is unramified, as can 
be easily checked by local calculations. That is, 
\begin{equation}\label{eqn:SNC-local0} 
\Omega^1_{{\ov{E}}/E} = 0.
\end{equation} 
If $E$ is irreducible, then it is
smooth and the result is known ({\sl cf.} \cite[Corollary~2.5]{CJams}).
So we assume that $E$ is not irreducible and prove the assertion by induction 
on the number of components of $E$. Let $E = E_1 \cup \cdots \cup E_r$
be the union of all its irreducible components. Let $E' = E_2 \cup \cdots 
\cup E_r$ and $D = E' \cap E_1$. Then $E'$ is also a strict normal crossing
divisor on $X$ and $D$ is a strict normal crossing divisor on $E_1$.
We first claim that the sequence
\begin{equation}\label{eqn:SNC-local1}
\Omega^i_{E/F} \to \Omega^i_{E_1/F} \oplus \Omega^i_{E'/F} \to \Omega^i_{D/F}
\to 0
\end{equation}
is exact. To show this, it suffices to prove that the map
$\Omega^i_{(E, E')/F} \to \Omega^i_{(E_1, D)/F}$ is surjective.

Let $I'$ denote the ideal of $E'$ as a closed subscheme of $E$. Then the fact
that $E$ is a snc, implies that $J = I' \sO_{E_1}$ is the ideal of $D$ as
a closed subscheme of $E_1$. In other words, one has $I' \surj J$.
The desired surjectivity now follows from the presentations of
$\Omega^i_{(E, E')/F}$ and $\Omega^i_{(E_1, D)/F}$ in ~\eqref{eqn:local2}.
This proves the claim. 

It is easy to see from the definitions of $E'$ and $D$ that the diagram
\[
\xymatrix@C.7pc{
D^N \ar[r] \ar[d] & E'^N \ar[d] \\
D \ar[r] & E'}
\]
is Cartesian. We now claim that the map
\begin{equation}\label{eqn:SNC-local2}
\Omega^i_{(E', E'^N)/F} \to \Omega^i_{(D,D^N)/F} \ {\rm is \ surjective}.
\end{equation}

To prove this claim, we can work locally on $X$ at points of $E$. Thus 
we can let $A = R/(a)$, where $R = (R, \mathfrak{m}, k)$ is the regular local 
ring of a closed point on $X$ with maximal ideal $\mathfrak{m} = (x_1, \cdots ,
x_{d+1})$ and residue field $k$, and $a = x_{i_1}\cdots x_{i_r}$ with
$1 \le i_1 < \cdots < i_r \le d+1$. We prove the claim in the
case when $r = d+1$. The case $r < d+1$ is simpler and follows in the same 
way.      

In this case, $E_1$, $E'$ and $D$ can be described by the local rings
$A_1 = R/(x_1), A' = R/(b)$, and $S = {A_1}/(b)$ respectively, where 
$b = x_2 \cdots x_{d+1}$. In particular, we have $A'^N =
\stackrel{}{\underset {2 \le i \le d+1}{\prod}} R/{x_i}$ and
$S^N = \stackrel{}{\underset {2 \le i \le d+1}{\prod}} {A_1}/{x_i}$.
Since $\Omega^i_{(A', A'^N)/F} \to \Omega^i_{(S, S^N)/F}$ is $A'$-linear,
it is enough to show that the map
$\Omega^i_{(A',A'^N)/F} {\otimes}_{A'} \widehat {A'} \to
\Omega^i_{(S,S^N)/F} {\otimes}_{A'} \widehat {A'}$ is surjective. Hence we
can replace $R$, $A'$ and $S$ by their completions, and assume
that $A' = {k[[x_1, \cdots , x_{d+1}]]}/{(b)}$ and $S =
{k[[x_2, \cdots , x_{d+1}]]}/{(b)}$. In particular, we have 
$A' \cong S[[x_1]]$ and hence the maps $A' \surj S$ and $A'^N \surj S^N$
have sections and hence the map
$\Omega^i_{(A', A'^N)/F} \to \Omega^i_{(S, S^N)/F}$ is surjective. This 
proves \eqref{eqn:SNC-local2}.    

To prove the proposition for $E$, we consider the diagram
\begin{equation}\label{eqn:SNC-local3}
\xymatrix@C.5pc{
& \Omega^i_{E/F} \ar[r] \ar[d] & {\Omega^i_{E'/F} \oplus \Omega^i_{{E_1}/F}}
\ar[d]^{\alpha} \ar[r] & \Omega^i_{D/F} \ar[r] \ar[d]^{\beta} & 0 \\
0 \ar[r] & H^0_{cdh}(E, \Omega^i_{E/F}) \ar[r] &
{\left\{ \begin{array}{l} 
{H^0_{cdh}(E', \Omega^i_{E'/F})}  \\ \ \ \ \ \ \ \ \ \oplus \\
{H^0_{cdh}(E_1, \Omega^i_{{E_1}/F})} \end{array}\right\}}
\ar[r] & H^0_{cdh}(D, \Omega^i_{D/F}) \ar@{->>}[r] & 
H^1_{cdh}(E, \Omega^i_{E/F}),}
\end{equation}
where the top sequence is exact by ~\eqref{eqn:SNC-local1} and the bottom
sequence is the Mayer-Vietoris exact sequence for the $cdh$ cohomology.
The middle and the right vertical maps are surjective by induction on the
dimension and number of components. The last map on the bottom row is 
surjective as the next term in the long exact sequence is zero by induction
again. This shows in particular that $H^j_{cdh}(E, \Omega^i_{E/F}) = 0$
for $j \ge 1$. We now only need to show that ${\rm Ker}(\alpha) \to
{\rm Ker}(\beta)$ is surjective to finish the proof of the proposition.
However, since $E'^N$, $E_1$ and $D^N$ are smooth, ${\rm Ker}(\alpha)$ and 
${\rm Ker}(\beta)$ are same as $\Omega^i_{(E', E'^N)/F}$ and 
$\Omega^i_{(D, D^N)/F}$ respectively by \cite[Remark~5.6.1]{BassNK}.
The required surjectivity now follows from ~\eqref{eqn:SNC-local2}.
\end{proof}

\begin{cor}\label{cor:SNC-main1}
Let $E \overset{i}{\inj} X$ be a strict normal crossing divisor on a smooth
$k$-scheme $X$ of dimension $d+1$. Then the natural map
\[
H^j_{Zar}(E, \Omega^i_{E/F}) \to H^j_{cdh}(E, \Omega^i_{E/F})
\]
is surjective for $j = d-1$ and isomorphism for $j = d$ for any subfield 
$F \subseteq k$ and any $i \ge 0$.
\end{cor}
\begin{proof} For the morphism $a: E_{cdh} \to E_{Zar}$ of sites, it follows
from Proposition~\ref{prop:SNC-local} that $R^ja_* a^* \Omega^i_{E/F} =
0$ for $j \ge 1$ and $i \ge 0$. The Leray spectral sequence implies that
$Ra_*a^*\Omega^i_{E/F} \cong a_*a^*\Omega^i_{E/F}$.
On the other hand, it also follows from Lemma~\ref{lem:direct-image} and
Proposition~\ref{prop:SNC-local} that $\Omega^i_{E/F} \to a_*a^*
\Omega^i_{E/F}$ is a surjective map of quasi-coherent sheaves with kernel
supported on $E_{\rm sing}$. The corollary now follows from 
Lemma~\ref{lem:folklore1}.
\end{proof} 

\begin{cor}\label{cor:SNC-main1*}
Let $\wt{X} \xrightarrow{f} X$ be a resolution of singularities of a normal 
quasi-projective $k$-variety $X$ of dimension $d+1$ with the reduced 
strict normal crossing exceptional divisor $E = f^{-1}(X_{\rm sing})$.
Then the map $\H^{2d-1}_{Zar}\left(E, \Omega^{< d}_{E/F}\right) \to
\H^{2d-1}_{cdh}\left(E, \Omega^{< d}_{E/F}\right)$ is an isomorphism.
Furthermore, the map 
$\H^{2d-1}\left(\wt{X}, \Omega^{< d}_{{\wt{X}}/F}\right)
\to \H^{2d-1}\left(E, \Omega^{< d}_{E/F}\right)$ is surjective in the
Zariski and the $cdh$ topology if the map
$H^{d+1}_{cdh}\left(X, \Omega^{d-1}_{X/F}\right) \to
H^{d+1}_{cdh}\left(\wt{X}, \Omega^{d-1}_{{\wt{X}}/F}\right)$ is an 
isomorphism.
\end{cor}
\begin{proof}
The first isomorphism follows at once from Lemma~\ref{lem:folklore2}
and Corollary~\ref{cor:SNC-main1}. For the other assertion, we can use the
first assertion to reduce to the case of $cdh$ topology. Now, the
surjectivity $H^d_{cdh}\left(E, \Omega^{d-1}_{E/F}\right) \surj
\H^{2d-1}_{cdh}\left(E, \Omega^{< d}_{E/F}\right)$ ({\sl cf.} 
Lemma~\ref{lem:folklore2}) implies that we only have to show that the
map $H^{d}_{cdh}\left(\wt{X}, \Omega^{d-1}_{{\wt{X}}/F}\right) \to
H^d_{cdh}\left(E, \Omega^{d-1}_{E/F}\right)$ is surjective. But this follows 
directly from our assumption and the Mayer-Vietoris exact sequence for the 
$cdh$ cohomology of $\Omega^{d-1}_{/F}$ for the resolution map once we observe
that $H^d_{cdh}\left(X_{\rm sing}, \Omega^{d-1}_{{X_{\rm sing}}/F}\right)$
vanishes as $X$ is normal.
\end{proof} 

The following result about the Zariski and the $cdh$ cohomology of the
K{\"a}hler differentials on seminormal varieties is of independent interest.
This is a weaker form of the previous result.

\begin{prop}\label{prop:semi-normal}
Let $X$ be a reduced seminormal $k$-scheme of dimension $d$. Then the natural
map
\[
H^d_{Zar}\left(X, \Omega^i_{X/F}\right) \to 
H^d_{Zar}\left(X, a_*a^*\Omega^i_{X/F}\right)
\]
is an isomorphism for any subfield $F \subseteq k$ and any $i \ge 0$.
\end{prop} 
% This lemma is not needed. Remove it. It is already proven in 
%Proposition~\ref{prop:SNC-local} below.
\begin{proof}
It follows from Lemma~\ref{lem:direct-image} that the map
\[
\Omega^i_{X/F} \xrightarrow{\theta^i} a_*a^*\Omega^i_{X/F}
\]
is a morphism of quasi-coherent sheaves. In particular, the cokernel sheaf
${\sF}^i_{X/F}$ is also quasi-coherent. \\
{\bf Claim.} ${\sF}^i_{X/F}$ is supported on a closed subscheme $Z$ which
has codimension at least two in $X$. \\
{\sl Proof of the claim}. Let $\widetilde{X} \xrightarrow{\pi} X$ be the 
normalization
map. Let $S \overset{i}{\inj} X$ be the reduced conductor subscheme and let 
$\widetilde{S} = S {\times}_X \widetilde{X}$. Let $\widetilde{S} 
\overset{i'}{\inj} \widetilde{X}$
be the inclusion map. Let ${\sH}^i_{\left(X, \widetilde{X}, S\right)/F}$ 
denote the cokernel of the map $\Omega^i_{(X, S)/F} \to 
{\pi}_*\left(\Omega^i_{(\widetilde{X},
\widetilde{S})/F}\right)$. The seminormality of $X$ implies that $\pi$ is 
unramified in codimension one. Hence we conclude from 
Proposition~\ref{prop:local}
that ${\rm support}\left({\sH}^i_{\left(X, \widetilde{X}, S\right)/F}\right) =
Z_1$, where $Z_1$ is a closed subscheme of $X$ of codimension at least two.
In particular, the sequence
\begin{equation}\label{eqn:SN0}
\Omega^i_{X/F} \to \pi_* \Omega^i_{\widetilde{X}/F} \oplus \Omega^i_{S/F}
\to \pi_* \Omega^i_{\widetilde{S}/F} \to 0
\end{equation}
is exact on the open subscheme $U_1 = X - Z_1$. Thus we get a commutative
diagram
\begin{equation}\label{eqn:SN1}
\xymatrix@C.5pc{
& \Omega^i_{X/F} \ar[r] \ar[d]^{\theta^i} & 
{\Omega^i_{{\widetilde{X}}/F} \oplus
\Omega^i_{S/F}} \ar[r] \ar[d] & \Omega^i_{{\widetilde{S}}/F} \ar[r] \ar[d] &
0 \\
0 \ar[r] & a_*a^*\Omega^i_{X/F} \ar[r] & 
{a_*a^*\Omega^i_{{\widetilde{X}}/F} \oplus a_*a^*\Omega^i_{S/F}} \ar[r] & 
a_*a^*\Omega^i_{{\widetilde{S}/F}} \ar[r] & \cdots ,}
\end{equation}
where the bottom sequence is the Mayer-Vietoris exact sequence
for $cdh$-cohomology because $S$ and $\widetilde{S}$ are reduced.

Since $X$ is seminormal, it follows from \cite[Theorems~3.5, 3.8]{LeahyV}
that there is a dense open subscheme $U_2 \subset X$ such that for $Z_2 =
X - U_2$, we have \\
$(i)$ ${\rm Codim}_X(Z_2) \ge 2$ \ {\rm and} \\
$(ii)$ $\widetilde{X} \cap U_2$, $S \cap U_2$ and $\widetilde{S} \cap
\pi^{-1}(U_2)$ are all smooth. \\
In particular, the restriction of $R^ja_*a^*\Omega^i_{{\widetilde{S}/F}}$
on $U_2$ is zero. Moreover, the middle as well as the right vertical maps are 
isomorphisms
by \cite[Corollary~2.5]{CJams} on $U_2$. A diagram chase in ~\eqref{eqn:SN1}
shows that $\theta^i$ is surjective on $U = U_1 \cap U_2$ and the codimension
of $Z = X - U$ is at least two. This proves the claim.

It also follows from \cite[Corollary~2.5]{CJams} that the kernel of 
$\theta^i$ is a quasi-coherent sheaf supported on $X_{\rm sing}$.   The proof
of Lemma~\ref{lem:folklore1} now proves the desired isomorphism of the top
cohomology groups.
\end{proof}

\section{Zariski and cdh cohomology of some singular schemes}
\label{section:ZCSS}
Our aim in this section is to compare some Zariski and $cdh$ cohomology of
projective varieties with isolated singularities using the results of 
Section~\ref{section:SNC*} and the resolution of singularities. Let
$X$ be a normal projective $k$-variety of dimension $d+1 \ge 2$ with 
only isolated singularities. Let $S = X_{\rm sing}$ denote the singular locus
of $X$ with the reduced induced structure. Let $f : \widetilde{X} \to X$
be a resolution of singularities of $X$ such that the reduced exceptional 
divisor $E \overset{\wt{i}}{\inj} \wt{X}$ has smooth components with strict 
normal crossings. For $n \ge 1$, let $nE$ denote the $n$th infinitesimal
thickening of $E$ defined by the sheaf of ideals ${\sI}^n$ on $\wt{X}$,
where $\sI$ is the sheaf of ideals defining $E$. Let 
$\ov{E} \xrightarrow{g} E$ be the normalization map of $E$. Note that
$\ov{E}$ is simply the disjoint union of irreducible components
of $E$.  
\begin{lem}\label{lem:EXC1}
The natural map 
\[
H^{d}_{Zar}\left(E, \frac{\Omega^d_{E/k}}{\Omega^{d-1}_{E/k}}\right)
\to H^{d}_{Zar}\left(\ov{E}, \frac{\Omega^d_{{\ov{E}}/k}}
{\Omega^{d-1}_{{\ov{E}}/k}}\right)
\]
is an isomorphism.
\end{lem}
\begin{proof}
It suffices to show by Lemma~\ref{lem:folklore1} that the map 
\[
\frac{\Omega^d_{E/k}}{\Omega^{d-1}_{E/k}} \to 
\frac{\Omega^d_{{\ov{E}}/k}}{\Omega^{d-1}_{{\ov{E}}/k}}
\]
is surjective. For this, it is enough to show that the map
$\Omega^d_{E/k} \to \Omega^d_{{\ov{E}}/k}$ is surjective. One can now check 
by an easy local calculation ({\sl cf.}
proof of Lemma~\ref{lem:EXC-local}) that, in fact the composite map 
$\Omega^d_{{\wt{X}}/k} \to \Omega^d_{{\ov{E}}/k}$ is surjective. This
completes the proof.  
\end{proof} 

Let $\omega_{{\wt{X}}/k}$ denote the canonical line bundle on $\wt{X}$.
For $n \ge 0$, the differential map $\Omega^{d}_{(n+1)E/k} \to
\Omega^{d+1}_{(n+1)E/k}$ induces the map
\[
\frac{\Omega^d_{{\left((n+1)E, \ov{E}/k\right)}/k}}
{\Omega^{d-1}_{{\left((n+1)E, \ov{E}/k\right)}/k}} \to
\Omega^{d+1}_{(n+1)E/k}.
\]
Also, we have the following natural surjections. 
\[
\xymatrix@C.5pc{
\omega_{{\wt{X}}/k} \ar[r] \ar[d] & \Omega^{d+1}_{(n+1)E/k} \\
\omega_{{\wt{X}}/k} {\otimes}_{\sO_{\wt{X}}} \sO_{nE} & }
\]

\begin{lem}\label{lem:EXC-local}
For any $n \ge 1$, the above maps induce the surjective maps
\[
\frac{\Omega^d_{{\left((n+1)E, \ov{E}/k\right)}/k}}
{\Omega^{d-1}_{{\left((n+1)E, \ov{E}/k\right)}/k}} \overset{\partial}{\surj}
\Omega^{d+1}_{(n+1)E/k} \surj 
\omega_{{\wt{X}}/k} {\otimes}_{\sO_{\wt{X}}} \sO_{nE}  
\]
which are isomorphisms on the smooth locus of $E$.
\end{lem}
\begin{proof}
This is a local calculation and can be checked at the local ring of
closed points of $\wt{X}$. So as in the proof of 
Proposition~\ref{prop:SNC-local}, let
$R = (R, \mathfrak{m}, k)$ be the regular local 
ring of a closed point on $X$ with maximal ideal $\mathfrak{m} = (x_1, \cdots ,
x_{d+1})$ and residue field $k$. For $n \ge 1$, let $A_n$ denote
the local ring of $nE$ at the chosen point.
 
We first consider the case when $A = R/{x_1}$ is smooth. In that case, 
$A_{n+1}$ is in fact of the form $A[x]/{(x^{n+1})}$. On can then explicitly 
calculate that
\[
\Omega^1_{(A_{n+1}, A)/k} = \left(\Omega^1_{A/k} {\otimes}_A xA_{n+1}\right)
\oplus \left({A_{n+1}}/{x^{n}}\right)dx,
\]
\[
\Omega^i_{(A_{n+1}, A)/k} = \left(\Omega^{i}_{A/k}{\otimes}_A xA_{n+1}\right) 
\oplus \left(\Omega^{i-1}_{A/k}{\otimes}_A \left(A_{n+1}/{x^n}\right)dx\right)
\ {\rm for} \ {2 \le i \le d} \ {\rm and}
\]
\[
\Omega^{d+1}_{{A_{n+1}}/k} = \Omega^{d}_{A/k}{\otimes}_A  
\left({A_{n+1}}/{x^n}\right)dx.
\]
The desired isomorphism $\frac{\Omega^d_{(A_{n+1}, A)/k}}
{\Omega^{d-1}_{(A_{n+1}, A)/k}} \xrightarrow{\cong} \Omega^{d+1}_{{A_{n+1}}/k}
\xrightarrow{\cong} \frac{\Omega^{d+1}_{R/k}}{(x^n)}$
can be now directly checked. 

We now assume that $A$ is not smooth. Let $A_n = R/{(a^n)}$, where 
$a = x_{i_1}\cdots x_{i_r}$ with
$1 \le i_1 < \cdots < i_r \le d+1$. We prove the assertion in the
case when $r = d+1$. The case $r < d+1$ is simpler and follows in the same 
way. The normalization of $A$ in this case is the ring $A^N =
\stackrel{}{\underset {1 \le i \le d+1}{\prod}} R/{x_i}$.  
We describe the various K{\"a}hler differentials in terms of generators and
relations. We fix a few notations.

For $1 \le i \le d+1$, let 
\[
a_i = {\underset{j \neq i}{\prod}} x_j \ {\rm and} \
dX_i = dx_1 \wedge \cdots dx_{i-1} \wedge dx_{i+1} \wedge
\cdots \wedge dx_{d+1}.
\]
For $1 \le i < j \le d+1$, we let 
\[
dY_{ij} = dx_1 \wedge \cdots dx_{i-1} \wedge dx_{i+1} \wedge
\cdots \wedge dx_{j-1} \wedge dx_{j+1} \wedge \cdots \wedge dx_{d+1}.
\] 
Finally, we put $dw = dx_1 \wedge \cdots \wedge dx_{d+1}$.
We then have
\[
\Omega^1_{R/k} = \stackrel{}{\underset {1 \le i \le d+1}{\oplus}}Rdx_i
\]
\[
\Omega^1_{{A_{n+1}}/k} = \frac{\stackrel{}{\underset {1 \le i \le d+1}{\oplus}}
Rdx_i}
{\left(\stackrel{}{\underset {1 \le i \le d+1}{\oplus}}
a^{n+1}Rdx_i, 
\stackrel{}{\underset {1 \le i \le d+1}{\sum}}a^n a_i dx_i\right)}
\]
\[
\Omega^1_{A^N/k} = \stackrel{}{\underset {1 \le i \le d+1}{\bigoplus}}
\left( \stackrel{}{\underset {j \neq i}{\oplus}}
(R/{x_i}) dx_j\right).
\]
Taking the various exterior powers, we get
\[
\Omega^d_{{A_{n+1}}/k} = 
\frac{\stackrel{}{\underset {1 \le i \le d+1}{\oplus}} 
RdX_i}{\left(\stackrel{}{\underset {1 \le i \le d+1}{\oplus}}a^{n+1}RdX_i,  
\stackrel{}{\underset {1 \le i < j \le d+1}{\oplus}}
\stackrel{}{\underset {l \neq i, j}
{\stackrel{}{\underset {1 \le l \le d+1}{\sum}}}}
a^na_l dx_l \wedge dY_{ij}\right)} 
\]
\[
\Omega^d_{{A^N}/k} = \stackrel{}{\underset {1 \le i \le d+1}{\oplus}}
R/{x_i} dX_i \ {\rm and}
\]
\[
\Omega^{d+1}_{{A_{n+1}}/k} = 
\frac{Rdw}{\left(a^{n+1}Rdw, \stackrel{}{\underset {1 \le i \le d+1}{\oplus}}  
Ra^na_i dw\right)}.
\]
It can now be directly checked from the above description that the 
differential map
$\partial : \Omega^d_{(A_{n+1}, A^N)/k} \to
\Omega^{d+1}_{{A_{n+1}}/k}$ is surjective and there is a surjection
$\Omega^{d+1}_{{A_{n+1}}/k} \to \frac{Rdw}{a^{n}Rdw} = 
\frac{\Omega^{d+1}_{R/k}}{a^{n}\Omega^{d+1}_{R/k}}$. We omit more details.
This finishes the proof of the lemma.
\end{proof}
\begin{cor}\label{cor:CEXC-local}
For every $n \ge 1$, the natural map
\[
\frac{H^d_{Zar}\left(nE, \Omega^d_{{nE}/k}\right)}
{H^d_{Zar}\left(nE, {\Omega^{d-1}_{{nE}/k}}\right)}
\to H^d_{Zar}\left(E, \frac{\Omega^d_{E/k}}{\Omega^{d-1}_{E/k}}\right)
\]
is an isomorphism.
\end{cor}
\begin{proof}
Using the exact sequence
\[
H^d_{Zar}\left(nE, \Omega^{d-1}_{nE/k}\right) \to
H^d_{Zar}\left(nE, \Omega^d_{nE/k}\right) \to
H^d_{Zar}\left(nE, \frac{\Omega^d_{{nE}/k}}{\Omega^{d-1}_{{nE}/k}}\right)
\to 0,
\]
it suffices to show that the map
\[
H^d_{Zar}\left(nE, \frac{\Omega^d_{{nE}/k}}{\Omega^{d-1}_{{nE}/k}}\right)
\to H^d_{Zar}\left(E, \frac{\Omega^d_{E/k}}{\Omega^{d-1}_{E/k}}\right)
\]
is an isomorphism.

By Lemma~\ref{lem:EXC1}, we can replace $E$ by $\ov{E}$. It is easily seen
from the above calculations that the map $\Omega^d_{nE/k} \to
\Omega^d_{{\ov{E}}/k}$ is surjective. Thus we have an exact sequence
of sheaves
\[
\frac{\Omega^d_{(nE, \ov{E})/k}}{\Omega^{d-1}_{(nE, \ov{E})/k}}
\to \frac{\Omega^d_{nE/k}}{\Omega^{d-1}_{nE/k}} \to
\frac{\Omega^d_{{\ov{E}}/k}}{\Omega^{d-1}_{{\ov{E}}/k}} \to 0.
\]
Hence, we only need to show that 
$H^d_{Zar}\left(nE, \frac{\Omega^d_{(nE, \ov{E})/k}}
{\Omega^{d-1}_{(nE, \ov{E})/k}}\right) = 0$.
By Lemma~\ref{lem:EXC-local}, it suffices to show that
$H^d_{Zar}\left(\wt{X}, \omega_{{\wt{X}}/k} \otimes \sO_{nE}\right) = 0$
for all $n \ge 1$. \\
Since ${\left\{H^d_{Zar}\left(\wt{X}, \omega_{{\wt{X}}/k}
\otimes \sO_{nE}\right)\right\}}_{n \ge 1}$ is an inverse system of 
surjective maps of finite-dimensional $k$-vector spaces, the map
\[
\stackrel{}{\underset{n}\varprojlim} 
H^d_{Zar}\left(\wt{X}, \omega_{{\wt{X}}/k} \otimes \sO_{nE}
\right) \to H^d_{Zar}\left(\wt{X}, \omega_{{\wt{X}}/k} \otimes \sO_{nE}\right)
\]
is surjective for each $n \ge 1$. Hence, it suffices to show that the 
inverse limit vanishes. However, this inverse limit is isomorphic to 
$H^0_{Zar}\left(X, R^df_*\omega_{{\wt{X}}/k}\right)$ by the formal function
theorem. On the other hand, $R^df_*\omega_{{\wt{X}}/k}$ is zero by
the Grauert-Riemenschneider vanishing theorem 
({\sl cf.} \cite[p. 59]{EsnaultV}).
This completes the proof of the corollary.
\end{proof} 
\begin{lem}\label{lem:CEXC-main}
Let $f: \wt{X} \to X$ be a resolution of singularities of a normal 
projective $k$-variety of dimension $d+1 \ge 2$ with only isolated
singularities as above, with reduced exceptional divisor $E$. Then the natural 
map
\[
{\frac{H^0_{Zar}\left(X, R^df_*\Omega^d_{{\wt{X}}/k}\right)}
{H^0_{Zar}\left(X, R^df_*\Omega^{d-1}_{{\wt{X}}/k}\right)}} \to
{\frac{H^d_{Zar}\left(E, \Omega^d_{E/k}\right)}{
H^d_{Zar}\left(E, \Omega^{d-1}_{E/k}\right)}}
\]
is an isomorphism.
\end{lem}
\begin{proof}
By the formal function theorem, it suffices to show that
\[
\frac{{\underset{n}\varprojlim}
{H^d_{Zar}\left(\wt{X}, \Omega^d_{{\wt{X}}/k} \otimes \sO_{nE}\right)}}
{{\underset{n}\varprojlim}
{H^d_{Zar}\left(\wt{X}, \Omega^{d-1}_{{\wt{X}}/k} \otimes \sO_{nE}\right)}}
\to H^d_{Zar}\left(E, \frac{\Omega^d_{E/k}}{\Omega^{d-1}_{E/k}}\right)
\]
is an isomorphism.

By Corollary~\ref{cor:CEXC-local}, we only need to show that the map
\begin{equation}\label{eqn:C-main1}
\frac{{\underset{n}\varprojlim}
{H^d_{Zar}\left(\wt{X}, \Omega^d_{{\wt{X}}/k} \otimes \sO_{nE}\right)}}
{{\underset{n}\varprojlim}
{H^d_{Zar}\left(\wt{X}, \Omega^{d-1}_{{\wt{X}}/k} \otimes \sO_{nE}\right)}}
\to \
{\underset{n}\varprojlim}
\frac{H^d_{Zar}\left(nE, \Omega^d_{{nE}/k}\right)}
{H^d_{Zar}\left(nE, \Omega^{d-1}_{{nE}/k}\right)}
\end{equation}
is an isomorphism.
We consider the following diagram of exact sequences.
\[
\xymatrix@C.5pc{
{\left\{ \begin{array}{l} 
{H^d_{Zar}\left(\wt{X}, \frac{\sI_{nE}\Omega^{d-1}_{{\wt{X}}/k}}
{\sI_{2nE}}\right)} \\ \ \ \ \ \ \ \ \ \ \ \ \ \oplus \\  
{H^d_{Zar}\left(\wt{X}, \frac{\sI_{nE}}{\sI_{2nE}} \otimes 
\Omega^{d-2}_{nE/k}\right)}  \end{array}\right\}} \ar[r]  \ar[d] &
{H^d_{Zar}\left(\wt{X}, \frac{\Omega^{d-1}_{{\wt{X}}/k}}
{\sI_{{2nE}/k}}\right)} \ar[r] \ar[d] &
{H^d_{Zar}\left(nE, \Omega^{d-1}_{{nE}/k}\right)} \ar[r] \ar[d] & 0 \\
{H^d_{Zar}\left(\wt{X}, \frac{\sI_{nE}\Omega^{d-1}_{{nE}/k}}
{\sI_{2nE}}\right)} \ar[r] &
{H^d_{Zar}\left(\wt{X}, \frac{\Omega^{d}_{{\wt{X}}/k}}{\sI_{nE}}\right)} 
\ar[r] & {H^d_{Zar}\left(nE, \Omega^d_{nE/k}\right)} \ar[r] & 0.}
\]
The exactness of these sequences follows from the right exactness
of corresponding sequences of sheaves and the fact that
all the underlying sheaves are supported on $E$.
The left vertical map from the first factor of the direct sum is the
quotient map and hence surjective.
Taking the inverse limit over $n$, using the Mittag-Lefler property of
these cohomology groups, and using the isomorphism
${\underset{n}{\varprojlim}}
H^d_{Zar}\left(\wt{X}, \Omega^{d-1}_{{\wt{X}}/k} \otimes \sO_{2nE}\right)
\xrightarrow{\cong} 
{\underset{n}{\varprojlim}}
H^d_{Zar}\left(\wt{X}, \Omega^{d-1}_{{\wt{X}}/k} \otimes \sO_{nE}\right)$,
we obtain
\[
\begin{array}{lll}
{\frac{{\underset{n}\varprojlim}
{H^d_{Zar}\left(\wt{X}, \Omega^d_{{\wt{X}}/k} \otimes \sO_{nE}\right)}}
{{\underset{n}\varprojlim}
{H^d_{Zar}\left(\wt{X}, \Omega^{d-1}_{{\wt{X}}/k} \otimes \sO_{nE}\right)}}}
& \cong & 
{\frac{{\underset{n}\varprojlim}
{H^d_{Zar}\left(\wt{X}, \Omega^d_{{\wt{X}}/k} \otimes \sO_{nE}\right)}}
{{\underset{n}\varprojlim}
{H^d_{Zar}\left(\wt{X}, \Omega^{d-1}_{{\wt{X}}/k} \otimes \sO_{2nE}\right)}}}
\\
& \xrightarrow{\cong} &
{\underset{n}{\varprojlim}}
{\frac{{H^d_{Zar}\left(\wt{X}, \Omega^d_{{\wt{X}}/k} \otimes \sO_{nE}\right)}}
{H^d_{Zar}\left(\wt{X}, \Omega^{d-1}_{{\wt{X}}/k} \otimes \sO_{2nE}\right)}} \\
& \xrightarrow{\cong} &
{\underset{n}{\varprojlim}}
\frac{H^d_{Zar}\left(nE, \Omega^d_{nE/k}\right)}
{H^d_{Zar}\left(nE, \Omega^{d-1}_{nE/k}\right)}.
\end{array}
\]
This completes the proof.
\end{proof} 

\begin{lem}\label{lem:ZariskiC}
Let $f:\wt{X} \to X$ be as in Lemma~\ref{lem:CEXC-main}. Assume that
$H^{d+1}_{Zar}\left(X, \Omega^{d-1}_{X/k}\right) \xrightarrow{\cong}
H^{d+1}_{Zar}\left(\wt{X}, \Omega^{d-1}_{{\wt{X}}/k}\right)$. 
Then the natural map $H^{d+1}_{Zar}\left(X, \Omega^{d}_{X/k}\right)
\to H^{d+1}_{cdh}\left(X, \Omega^d_{X/k}\right)$ is an isomorphism.
\end{lem}
\begin{proof}
We first observe that the map
\begin{equation}\label{eqn:ZariskiC1}
H^{d+1}_{Zar}\left(X, \Omega^i_{X/F}\right) \to
H^{d+1}_{cdh}\left(X, \Omega^i_{X/F}\right)
\end{equation}
is surjective for all $i \ge 0$ and all subfields $F \subseteq k$ 
by \cite[Proposition~2.6]{CJams}.
In particular, $H^{d+1}_{cdh}\left(X, \Omega^{d-1}_{X/k}\right) 
\xrightarrow{\cong}
H^{d+1}_{cdh}\left(\wt{X}, \Omega^{d-1}_{{\wt{X}}/k}\right)$.
The Leray spectral sequence for $f$, applied to the sheaves 
$\Omega^i_{{\wt{X}}/k}$ for the Zariski site, gives the 
exact sequences
\[
H^{d}_{Zar}\left(\wt{X}, \Omega^i_{{\wt{X}}/k}\right) \to
H^0_{Zar}\left(X, R^df_*\Omega^i_{{\wt{X}}/k}\right) \to
H^{d+1}_{Zar}\left(X, \Omega^i_{X/k}\right) \to 
H^{d+1}_{Zar}\left(\wt{X}, \Omega^i_{{\wt{X}}/k}\right) \to 0
\]
for all $i \ge 0$.
Applying this for $i \ge d-1$ and using our assumption, we then 
get an exact sequence
\[
H^d_{Zar}\left(\wt{X}, \Omega^d_{{\wt{X}}/k}\right) 
\to 
\frac{H^0_{Zar}\left(X, R^df_*\Omega^d_{{\wt{X}}/k}\right)}
{H^0_{Zar}\left(X, R^df_*\Omega^{d-1}_{{\wt{X}}/k}\right)}
\to
H^{d+1}_{Zar}\left(X, \Omega^d_{X/k}\right) \to
H^{d+1}_{Zar}\left(\wt{X}, \Omega^d_{{\wt{X}}/k}\right) \to 0.
\]
We now compare this with the similar Mayer-Vietoris exact sequence for
the $cdh$ cohomology to get a commutative diagram of exact sequences
\[
\xymatrix@C.4pc{
{H^d_{Zar}\left(\wt{X}, \Omega^d_{{\wt{X}}/k}\right)}
\ar[r] \ar[d] & 
{\frac{H^0_{Zar}\left(X, R^df_*\Omega^d_{{\wt{X}}/k}\right)}
{H^0_{Zar}\left(X, R^df_*\Omega^{d-1}_{{\wt{X}}/k}\right)}}
\ar[r] \ar[d] & 
{H^{d+1}_{Zar}\left(X, \Omega^d_{X/k}\right)}
\ar@{->>}[r] \ar[d] &
{H^{d+1}_{Zar}\left(\wt{X}, \Omega^d_{{\wt{X}}/k}\right)} \ar[d] \\
{H^d_{cdh}\left(\wt{X}, \Omega^d_{{\wt{X}}/k}\right)} 
\ar[r] &
{\frac{H^d_{cdh}\left(E, \Omega^d_{E/k}\right)}{
H^d_{cdh}\left(E, \Omega^{d-1}_{E/k}\right)}} \ar[r] &
H^{d+1}_{cdh}\left(X, \Omega^d_{X/k}\right) \ar@{->>}[r] &
H^{d+1}_{cdh}\left(\wt{X}, \Omega^d_{{\wt{X}}/k}\right).}
\]
The left and the right vertical maps on both ends are isomorphisms by
\cite[Corollary~2.5]{CJams}. The second vertical map from the left is
an isomorphism by Lemma~\ref{lem:CEXC-main} and Corollary~\ref{cor:SNC-main1}.
Hence the remaining vertical map is also an isomorphism by $5$-lemma.
\end{proof}

For the remaining part of this section, recall our convention that
a K{\"a}hler differential (or Hochschild and cyclic homology) without
the mention of the coefficient field means that the underlying field
is taken to be $\Q$.
\begin{cor}\label{cor:ZariskiC*}
Let $X$ be as in Lemma~\ref{lem:CEXC-main} such that
$H^{d+1}_{Zar}\left(X, \Omega^{i}_{X/k}\right) = 0$ for $0 \le i \le d-1$.
Then $H^{d+1}\left(X, \Omega^i_{X}\right) = 0$ for the Zariski or the
$cdh$-site and for any $i \le d-1$. Moreover, the map
$H^{d+1}_{Zar}\left(X, \Omega^d_X\right) \to H^{d+1}_{cdh}\left(X,
\Omega^d_X\right)$ is an isomorphism.
\end{cor}
\begin{proof}
For the first assertion, we only need to show the vanishing of the
Zariski cohomology by ~\eqref{eqn:ZariskiC1}.
The case $i \ge 0$ is part of the assumption. For any $i \ge 1$, there
is a filtration ${\left\{F^j\Omega^i_{X}\right\}}_{0 \le j \le i}$
such that there is a surjection
\[
\Omega^j_{k/{\Q}} \otimes_k \Omega^{i-j}_{X/k} \surj
\frac{F^j\Omega^i_{X}}{F^{j+1}\Omega^i_{X}}
\]
which is an isomorphism on $X_{\rm smooth}$ and this latter set is
finite. Hence the first assertion follows from our assumption and an easy
induction on $1 \le i \le d-1$ and $0 \le j \le i$.
Furthermore, this also implies that the map
$H^{d+1}\left(X, \Omega^d_X\right) \to H^{d+1}\left(X, \Omega^d_{X/k}\right)$
is an isomorphism for both Zariski and $cdh$-sites. The second assertion of
the corollary now follows from Lemma~\ref{lem:ZariskiC}.
\end{proof}

The following is the main result of this section. 
 
\begin{prop}\label{prop:HC-1}
Let $X$ be a normal projective $k$-variety of dimension 
$d+1 \ge 2$ with only isolated singularities such that
$H^{d+1}_{Zar}\left(X, \Omega^i_{X/k}\right) = 0$ for $0 \le i \le d-1$.
Then the natural maps
\[
HC^{(d+1)}_0(X) \to \H^0_{cdh}\left(X, {\sH}{\sC}^{(d+1)}\right) \ 
{\rm and}
\]
\[
HC^{(d)}_{-1}(X) \to \H^1_{cdh}\left(X, {\sH}{\sC}^{(d)}\right)
\]
are respectively surjective and isomorphism.
\end{prop}
\begin{proof}
One knows that $HC^{(i)}_i(A) = {\Omega^i_A}/{\Omega^{i-1}_A}$ and
$HC^{(i)}_j(A) = 0$ for $j > i$ for any local ring of $X$ 
({\sl cf.} \cite[Theorems~4.6.7, 4.6.8]{Loday}).     
Hence the spectral sequence
\[
E^{p,q}_2 = H^p\left(X, {\sH}{\sC}_{X, q}\right) \Rightarrow
HC_{-p-q}(X)
\]
gives isomorphisms 
\begin{equation}\label{eqn:FOLK*}
HC^{(d)}_{-1}(X) \cong H^{d+1}_{Zar}\left(X, {\sH}{\sC}^{(d)}_d\right) \
{\rm and} \
HC^{(d+1)}_{0}(X) \cong 
H^{d+1}_{Zar}\left(X, {\sH}{\sC}^{(d+1)}_{d+1}\right).
\end{equation}
On the other hand, it follows from \cite[Theorem~2.2]{CJams} that
\begin{equation}\label{eqn:FOLK*1}
\H^0_{cdh}\left(X, {\sH}{\sC}^{(d+1)}\right) \cong
\H^{2d+2}_{cdh}\left(X, \Omega^{\le d+1}_X\right) \ {\rm and} \
\H^1_{cdh}\left(X, {\sH}{\sC}^{(d)}\right) \cong
\H^{2d+1}_{cdh}\left(X, \Omega^{\le d}_X\right).
\end{equation}
Lemma~\ref{lem:folklore2} now implies that there are commutative diagrams
and exact sequences
\[
\xymatrix@C.5pc{
{H^{d+1}_{Zar}\left(X, \Omega^{d+1}_X\right)} \ar@{->>}[r] \ar[d] &
{HC^{(d+1)}_{0}(X)} \ar[d] \\
{H^{d+1}_{cdh}\left(X, \Omega^{d+1}_X\right)} \ar@{->>}[r] &
{\H^0_{cdh}\left(X, {\sH}{\sC}^{(d+1)}\right)}}
\]
\[
\xymatrix@C.5pc{
{H^{d+1}_{Zar}\left(X, \Omega^{d-1}_X\right)} \ar[r] \ar[d] &
{H^{d+1}_{Zar}\left(X, \Omega^{d}_X\right)} \ar[r] \ar[d] &
{HC^{(d)}_{-1}(X)} \ar[d] \ar[r] & 0 \\
{H^{d+1}_{cdh}\left(X, \Omega^{d-1}_X\right)} \ar[r] &
{H^{d+1}_{cdh}\left(X, \Omega^{d}_X\right)} \ar[r] &
{\H^1_{cdh}\left(X, {\sH}{\sC}^{(d)}\right)} \ar[r] & 0.}
\]
The left vertical map in the top square is surjective by
~\eqref{eqn:ZariskiC1}. This proves the first assertion of the proposition.
Observe that we have not used the conditions of the proposition until now.
In particular, the first surjectivity always holds. 

Now assume the given conditions.
In this case, the terms on the left ends of both the rows in the bottom 
diagram are zero by Corollary~\ref{cor:ZariskiC*}.
The middle vertical map is an isomorphism again by 
Corollary~\ref{cor:ZariskiC*}.
Hence the right vertical map is also an isomorphism, proving the second
assertion.
\end{proof}  

\section{$K$-theory and $KH$-theory of some singular schemes}
\label{section:KKH}
Recall from Theorem~\ref{thm:fiber0} that the Chern character from the
algebraic $K$-theory to the negative cyclic homology induces a natural
weak equivalence $\wt{\sK}(X) \cong \Omega^{-1}\wt{{\sH}{\sC}}(X)$
for any $k$-scheme $X$. In particular, there is a homotopy fibration sequence
\[
\wt{\sK}(X) \to {\sH}{\sC}(X)[-1] \xrightarrow{{\phi}_{HC}} 
\H_{cdh}\left(X, {\sH}{\sC}\right)[-1].
\]
In particular, the homotopy groups of $\wt{\sK}(X)$ have 
$\lambda$-decomposition such that the above gives a long exact sequence
of homotopy groups which preserves this decomposition
({\sl cf.} \cite{CMAnnalen}). Thus we have $\wt{K}_n(X) :=
{\pi}_n {\wt{\sK}}(X) = {\underset{i}{\bigoplus}} \wt{K}^{(i)}_n(X)$.

\begin{lem}\label{lem:K-fiber**}
Let $X$ be a $k$-scheme of dimension $d+1$. Then there is an exact sequence
\[
0 \to \wt{K}^{(d+1)}_0(X) \to 
\frac{H^{d+1}_{Zar}\left(X, \Omega^{d}_X\right)}
{H^{d+1}_{Zar}\left(X, \Omega^{d-1}_X\right)} \to
\frac{H^{d+1}_{cdh}\left(X, \Omega^{d}_X\right)}
{H^{d+1}_{cdh}\left(X, \Omega^{d-1}_X\right)} \to 0.
\]
\end{lem}
\begin{proof}
We have the following long exact sequence coming from the above fibration
sequence.
\[
HC^{(d+1)}_0(X) \to \H^0_{cdh}\left(X, {\sH}{\sC}^{(d+1)}\right)
\to \wt{K}^{(d+1)}_0(X) \to HC^{(d)}_{-1}(X) \to 
\H^1_{cdh}\left(X, {\sH}{\sC}^{(d)}\right).
\]  
The lemma now follows from the identification of various terms in this 
exact sequence in ~\eqref{eqn:FOLK*} and ~\eqref{eqn:FOLK*1}, 
Lemma~\ref{lem:folklore2} and \cite[Theorem~2.6]{CJams}
({\sl cf.} proof of Proposition~\ref{prop:HC-1}).
\end{proof}

\begin{prop}\label{prop:K-fiber}
Let $X$ be a normal projective $k$-variety of dimension 
$d+1 \ge 2$ with only isolated singularities such that
$H^{d+1}_{Zar}\left(X, \Omega^i_{X/k}\right) = 0$ for $0 \le i \le d-1$.
Then $\wt{K}^{(d+1)}_0(X) = 0$.
\end{prop}
\begin{proof}
Follows directly from Lemma~\ref{lem:K-fiber**} and Proposition~\ref{prop:HC-1}.
\end{proof}

\subsection{Gamma filtration of $K_*(X)$}
Recall from \cite{Soule} that for any $k$-scheme $X$ of dimension $d$,
there are natural $\gamma$-operations $\gamma^j$ on $K_*(X)$ which
naturally define the Adams operations $\psi^j : K_n(X) \to K_n(X)$
for each $n$ and for $j \in \Z$. All these operations 
commute with the pull-back maps on $K$-groups of schemes.
These $\gamma$-operations define a natural decreasing filtration
\[
0 = F^{n+1+d}_{\gamma}K_n(X) \subseteq F^{n+d}_{\gamma}K_n(X) \subseteq 
\cdots \subseteq F^0_{\gamma}K_n(X) = K_n(X)
\]
such that $F^1_{\gamma}K_0(X)$ is the subgroup of $K_0(X)$ generated
by vector bundles of virtual rank zero. Our purpose here is to 
describe $F^d_{\gamma}K_0(X)$ in terms of algebraic cycles for certain
singular schemes. This description will be later used in this work
to study the Chow group of zero-cycles on such schemes.

We further recall that $K_{n, \Q}(X) := K_n(X)\otimes_{\Z} {\Q}$
has a canonical decomposition $K_{n, \Q}(X) = \
\stackrel{}{\underset{i}{\bigoplus}} K^{(i)}_{n, \Q}(X)$ 
in terms of the eigenspaces of the Adams operators $\psi^j$ (which
does not depend on $j$). This Adams decomposition is related to the
$\gamma$-filtration by the natural isomorphism
\[
\frac{F^i_{\gamma}K_{n, \Q}(X)}{F^{i+1}_{\gamma}K_{n, \Q}(X)}
\xrightarrow{\cong} K^{(i)}_{n, \Q}(X).
\]
In particular, one has 
\begin{equation}\label{eqn:filt}
F^{n+d}_{\gamma}K_{n, \Q}(X) \cong K^{(n+d)}_{n, \Q}(X).
\end{equation}
The following is the generalization of the Grothendieck
Adams-Riemann-Roch theorem for $K_0$ ({\sl cf.} \cite{SGA6})
to the higher $K$-theory of singular schemes.

\begin{thm}\label{thm:GRR}
Let $Y \xrightarrow{f} X$ be a regular embedding of quasi-projective
$k$-varieties of codimension $d \ge 1$. Let $N_{Y/X}$ denote the normal
bundle of $Y$ in $X$. Then for any $n \ge 0$ and $j \in \Z$, the diagram
\[
\xymatrix@C3.5pc{
K_n(Y) \ar[r]^{\theta^j(N_{Y/X}) \psi^j} \ar[d]_{f_*} & K_n(Y) \ar[d]^{f_*} \\
K_n(X) \ar[r]_{\psi^j} & K_n(X)}
\]
is commutative, where $\theta^j$'s are the cannibalistic operators
({\sl cf.} \cite[Section~4.5]{Soule}).
\end{thm}

Before we prove this theorem, we prove the following result on the
deformation to the normal cone of the regular embedding $Y \xrightarrow{f} X$.
Let $X' = {\P}(N_{Y/X} \oplus \sO_Y)$, 
$M = Bl_{Y \times \{\infty\}}(X \times \P^1)$
and consider the following deformation to the normal cone diagram.
\begin{equation}\label{eqn:GRR1}
\xymatrix{
Y \ar[rr]_{i_0} \ar[rd]^{u'} \ar[dd]^{f} & & {Y \times {\P}^1} 
\ar@/_/[ll]_{p_Y}
\ar[dd]^{F} & Y \ar[l]^{i_{\infty}} \ar[dd]^{f'} \\
& {Y \times {\A}^1} \ar[dd]_{F'} \ar[ru]^{j'} & & \\
X \ar[rr]_{h} \ar[dr]^{u} & & M \ar@/_/[ll]_{\phi} & X' \ar[l]_{i} \\
& {X \times {\A}^1} \ar[ru]^{j} & & }
\end{equation}   
In this diagram, all the vertical arrows are the closed regular
embeddings, $i_0$
and $i_{\infty}$ are the obvious inclusions of $Y$ in $Y \times {\P}^1$ along
the specified points, $i$ and $j$ are inclusions of the inverse
images of ${\A}^1$ and $\infty$ respectively under the map $\pi$, 
$u$ and $f'$ are are zero section embeddings and $p_Y$ is the projection
map. The map $\phi$ is the composite $M \to X \times \P^1 \to X$.
In particular, one has ${p_Y} \circ {i_0} = {p_Y} \circ {i_{\infty}}
= {id}_Y$ and $\phi \circ h = {id}_X$, .

\begin{lem}\label{lem:GRR2}
Consider the diagram ~\eqref{eqn:GRR1} and let $y \in K_n(Y)$. Then
there exists $z \in  K_n(M)$ such that $f_*(y) = h^*(z)$ and
${f'}_*(y) = i^*(z)$.
\end{lem}
\begin{proof} Put $\tilde{y} = {p_Y}^*(y)$ and $z = F_*(\tilde{y})$.
Then
\[
\begin{array}{lllll}
f_*(y) & = & f_*\left({\left(p_Y \circ {i_0}\right)}^*(x)\right) & & \\
& = & f_* \circ {i_0}^* \circ {p_Y}^* \left(y\right) & = &
f_* \circ {i_0}^*\left(\tilde{y}\right) \\
& = & f_* \circ {u'}^* \circ {j'}^*\left(\tilde{y}\right) & &  \\
& = & u^* \circ {F'}_*\left({j'}^*\left(\tilde{y}\right)\right) & &
\\
& = & u^* \circ j^* \circ F_*\left(\tilde{y}\right) & &  
\hspace*{2cm}\left({\rm{since}} \ j \ 
{\rm{is \ an \ open \ immersion}}\right) \\ 
& = & h^* \circ F_*\left(\tilde{y}\right) & = & h^*(z). 
\end{array}
\]
Similarly, 
\[
\begin{array}{lllll}
{f'}_*(y) & = & {f'}_*\left({\left(p_Y \circ 
{i_{\infty}}\right)}^*(x)\right) & & \\
& = & {f'}_* \circ i_{\infty}^* \circ {p_Y}^* \left(y\right) & = &
{f'}_* \circ i_{\infty}^*\left(\tilde{y}\right) \\
& = & i^* \circ {F}_*\left(\tilde{y}\right) & & \\
& = & i^*(z).
\end{array}
\]
Here, the fifth equality in the first array and the fourth equality in the
second follow from \cite[Proposition~2.11]{VistoliV}.
\end{proof}
{\bf{Proof of Theorem~\ref{thm:GRR}}.}
This theorem for $K_0$ was proven in \cite{SGA6} and was also proven in an 
axiomatic way in \cite[Theorem~6.3]{FultonL} using the method of the 
deformation to the normal cone. The proof for the higher $K$-theory can also 
be given using this method. We only give a brief sketch and leave the
details to the readers. If $f$ is the zero-section embedding for a vector
bundle $X \to Y$, then this is proven by Soul{\'e} 
({\sl cf.} \cite[Theorem~3]{Soule}). Actually, Soul{\'e} assumes $X$ and
$Y$ to be smooth, but the proof of this particular case of the 
zero-section embedding goes through even in the singular case.

Following the axiomatic approach in \cite[Chapter II, Theorem~1.2]{FultonL} 
for $K_0$, the general case is deduced from the above using the above 
deformation to the normal cone diagram. The reader can check in {\sl loc. cit.}
that the general case for the higher $K$-theory follows directly from \\
$(i)$ the general case for $K_0$, \\
$(ii)$ the zero-section embedding case for higher $K$-groups and \\
$(iii)$ Lemma~\ref{lem:GRR2}, \\
 once we have the projection formula for the push-forward
and pull-back maps on the higher $K$-theory for a local complete intersection
morphism. But this is already shown in \cite{VistoliV}.
We refer to \cite[Theorem~II.1.2, Lemmas~V.6.1, 6.2]{FultonL}
for more detail.
$\hfill \square$
\\

\begin{cor}\label{cor:GRR*}
Let $Y \xrightarrow{f} X$ be as in Theorem~\ref{thm:GRR}. Then for 
any $n, i \ge 0$, one has 
\[
f_*\left(F^i_{\gamma}K_{n, \Q}(Y)\right) \subseteq F^{i+d}_{\gamma} 
K_{n, \Q}(X).
\]
\end{cor}
\begin{proof} This follows directly from Theorem~\ref{thm:GRR} exactly
in the same way as is proven for $K_0$ in \cite[Proposition~V.6.4]{FultonL}.
\end{proof} 
  
For more applications, recall that ({\sl cf.} \cite{ThomasonT})
for a $k$-scheme $X$, there are Brown-Gersten spectral sequences 
\begin{equation}\label{eqn:SSZ}
E^{p,q}_2 = H^p_{Zar}\left(X, \sK_{q}\right) \Rightarrow K_{-p-q}(X).
\end{equation}
There are similar spectral sequences ({\sl cf.} \cite[Theorem~1]{Haes})  
\begin{equation}\label{eqn:SSC}
E^{p,q}_2 = H^p_{cdh}\left(X, \sK_{q}\right) \Rightarrow KH_{-p-q}(X)
\end{equation}
for the $KH$-theory, and there is a natural morphism from the first spectral
sequence to the second. This induces the associated {\sl Brown filtration}
(as called so in \cite{GilletS}) $F^{\bullet}_B$ on these $K$-theories.

Let $X$ be a quasi-projective $k$-variety of dimension $d$ and let 
$x \in X$ be a smooth closed point of $X$. Then $\{x\} \overset{i}{\inj} X$
is a regular embedding and hence there is a natural map 
$i_*: \Z = K_0(\{x\}) \to K_0(X)$. Let $F^dK_0(X)$ denote the subgroup
of $K_0(X)$ generated by the images of the classes of smooth points via these
maps.   
\begin{cor}\label{cor:compare*}
Let $X$ be a quasi-projective $k$-variety of dimension $d$. Then
\[
CH^d(X) \surj  F^dK_0(X) \inj  F^d_{\gamma}K_0(X) \inj 
F^d_BK_0(X).
\]
up to torsion.
\end{cor}
\begin{proof}
The first surjectivity is already known and can be easily proved from
the definitions of the terms. The second inclusion follows from
Corollary~\ref{cor:GRR*} and third inclusion follows from
\cite[(26), p.138]{GilletS}.
\end{proof}

\begin{cor}\label{cor:compare}
Let $X$ be a normal quasi-projective $k$-variety of dimension $d$ with only 
isolated singularities. Then there are isomorphisms
\[
H^d_{Zar}\left(X, \sK_{X,d}\right) \cong CH^d(X) \cong F^dK_0(X) \cong
F^d_BK_0(X) \cong F^d_{\gamma}K_0(X)
\]
up to torsion.
\end{cor}
\begin{proof} The first two isomorphisms are already known ({\sl cf.}
\cite{Viale} and \cite{Levine1}). The Brown-Gersten spectral sequence
implies that there is a surjection $H^d_{Zar}\left(X, \sK_{X, d}\right) 
\surj F^d_BK_0(X)$ and hence must be an isomorphism because of the first
two isomorphisms. The remaining isomorphisms now follow from 
Corollary~\ref{cor:compare*}. 
\end{proof} 

\begin{thm}\label{thm:main1}
Let $X$ be a normal projective $k$-variety of dimension $d$ with only
isolated singularities such that $H^d_{Zar}\left(X, \Omega^i_{X/k}\right) 
= 0$ for $i \le d-2$. Then the natural map $K_0(X) \to KH_0(X)$ induces
the inclusion
\[
CH^d(X) \xrightarrow{\cong} F^d_BK_0(X) \inj F^d_BKH_0(X).
\]
up to torsion.
\end{thm}
%\begin{remk}\label{remk:main1*}
%It will follow from the proof of Theorem~\ref{thm:main2} that this
%inclusion is actually an isomorphism.
%\end{remk}
\begin{proof} If $d = 1$, this is well known, so we assume $d \ge 2$.
The natural map from the first to the second spectral sequence above
gives the map $F^d_BK_0(X) \to F^d_BKH_0(X)$. To prove its inclusion,
we can replace $F^d_BK_0(X)$ by $F^d{\gamma}K_0(X)$ by
Corollary~\ref{cor:compare}. This group can in turn be replaced by
$K^{(d)}_{0, \Q}(X)$ by ~\eqref{eqn:filt}. We have seen earlier that
there are $\lambda$-decompositions on $K_0(X)$ and $\wt{K}_0(X)$, and
there is a fibration sequence of
spectra $\wt{K}(X) \to K(X) \to KH(X)$ by Corollary~\ref{cor:fiber1}.
This fibration sequence in particular yields an exact sequence
of eigenspaces
\[
\wt{K}^{(d)}_0(X) \to K^{(d)}_{0, \Q}(X) \to F^d_B KH_0(X).
\] 
Thus we only need to show that $\wt{K}^{(d)}_0(X) = 0$. But this is
proven in Proposition~\ref{prop:K-fiber}.  
\end{proof}

\section{$cdh$ cohomology and Hodge theory of singular schemes}
\label{section:HODGET}
We shall assume the ground field $k$ to be the field of complex numbers
$\C$ for the rest of this paper. For a $\C$-scheme $X$, we shall denote its
associated analytic space by $X_{\rm an}$. 
For simplicity of presentation, we shall assume all $\C$-schemes to be
quasi-projective for the rest of this work.
If $A$ is an abelian group,
then the analytic singular cohomology $H^*(X_{\rm an}, A)$ will be
simply written as $H^*(X, A)$. For a chain complex $\sF^{\bullet}$ of
presheaves of abelian groups on the analytic, Zariski or the 
$cdh$-site, we shall consider 
$\sF^{\bullet}$ also as a presheaf of Eilenberg-Mac Lane spectra and
write $R\Gamma\left(-, \sF^{\bullet}\right)$ as 
$\H \left(-, \sF^{\bullet}\right)$.

The Hodge theory of singular schemes
was invented by Deligne in \cite{Hodge2} and \cite{Hodge3}, where he showed
using Hironaka's resolution of singularities that for every $\C$-scheme $X$,
the analytic cohomology $H^*(X, \Z)$ has a mixed Hodge structure and hence
is equipped with a natural weight and Hodge filtration. This was achieved
by showing the descent property of the singular cohomology: if 
$X_{\bullet} \xrightarrow{\pi} X$ is a smooth and proper simplicial 
hypercovering, then the map $\Z_X \to R{\pi}_*\left({\Z}_{X_{\bullet}}\right)$
is a weak equivalence. Our purpose here is to interpret the Hodge theory
and the Hodge cohomology of $X$ in terms of the $cdh$ cohomology of
the algebraic K{\"a}hler differentials. We show in particular that for a
singular projective $\C$-scheme $X$, there
is a natural isomorphism $\H^*_{cdh}\left(X, \Omega^{\bullet}_{X/{\C}}\right)
\cong H^*(X, \C)$ such that the Hodge filtration corresponds to the 
{\sl Betti} filtration on the $cdh$ cohomology. We deduce several
consequences which are later used to study the Chow groups of
zero-cycles on such varieties in terms of $cdh$ cohomology of differential
forms.
   
\subsection{$cdh$-descent for Du Bois complex}\label{subsection:DBois}
Let $X$ be a $\C$-scheme and let $X_{\bullet} \xrightarrow{\pi} X$ be a 
smooth proper hypercovering of $X$. Let
\begin{equation}\label{eqn:Db0}
{\un{\Omega}}^{\bullet}_X : = 
R{\pi}_*\left(\Omega^{\bullet}_{{X_{\bullet}}/{\C}}\right).
\end{equation}
This complex was invented by Du Bois in \cite{DB}. He showed
that ${\un{\Omega}}^{\bullet}_X$ is a filtered complex of sheaves with
quasi-coherent cohomology sheaves such that 
\begin{equation}\label{eqn:Db00}
F^i{\un{\Omega}}^{\bullet}_X = 
R{\pi}_*\left(\Omega^{\ge i}_{{X_{\bullet}}/{\C}}[-i]\right) \ {\rm and}
\end{equation}
\[
Gr^i_F\left({\un{\Omega}}^{\bullet}_X\right) =
\frac{F^i{\un{\Omega}}^{\bullet}_X}{F^{i+1}{\un{\Omega}}^{\bullet}_X}
\cong R{\pi}_*\left(\Omega^i_{{X_{\bullet}}/{\C}}\right)[-i].
\]
In particular, ${\un{\Omega}}^{\bullet}_X \in D(qc/X)$, where
$D(qc/X)$ is the derived category of quasi-coherent sheaves on $X$.
If $X$ is projective, then ${\un{\Omega}}^{\bullet}_X$ has
coherent cohomology sheaves. 
The exact triangle 
\begin{equation}\label{eqn:Db01}
R{\pi}_*\left(\Omega^{\ge i}_{{X_{\bullet}}/{\C}}[-i]\right) \to
R{\pi}_*\left(\Omega^{\bullet}_{{X_{\bullet}}/{\C}}\right) \to
R{\pi}_*\left(\Omega^{<i}_{{X_{\bullet}}/{\C}}\right)
\end{equation}
now shows that 
\begin{equation}\label{eqn:Db02}
\frac{{\un{\Omega}}^{\bullet}_X}{F^i{\un{\Omega}}^{\bullet}_X}
\cong R{\pi}_*\left(\Omega^{<i}_{{X_{\bullet}}/{\C}}\right).
\end{equation}
We let
\[
{\un{\Omega}}^i_X := Gr^i_F\left({\un{\Omega}}^{\bullet}_X\right)[i]
=  R{\pi}_*\left(\Omega^i_{{X_{\bullet}}/{\C}}\right).
\]
\begin{lem}\label{lem:DB1}
Let $X$ be a $\C$-scheme of dimension $d$. Then 
\[
F^d{\un{\Omega}}^{\bullet}_X \cong 
Rf_*\Omega^{\bullet}_{{\wt{X}}/{\C}} \ {\rm and} \
F^i{\un{\Omega}}^{\bullet}_X = 0 \ {\rm in} \ D(qc/X) \ {\rm for} \ i \ge d+1,
\]
where $\wt{X} \xrightarrow{f} X$ is any resolution of singularities of 
$X_{\rm red}$.
\end{lem}
\begin{proof} {\sl cf.} \cite[Proposition~4.1]{DB}.
\end{proof} 

\begin{lem}\label{lem:DB2}
For any $\C$-scheme $X$ and $i \ge 0$, the natural map
\[
\H_{Zar}\left(X, F^i{\un{\Omega}}^{\bullet}_X\right) \to
\H_{cdh}\left(X, F^i{\un{\Omega}}^{\bullet}_X\right) 
\]
is a weak equivalence.
\end{lem}
\begin{proof}
We prove by a descending induction on $i \ge 0$. We know from 
Lemma~\ref{lem:DB1} that $F^{d+1}{\un{\Omega}}^{\bullet}_X = 0$,
where $d$ is the dimension of $X$. In particular, 
$F^{d}{\un{\Omega}}^{\bullet}_X \xrightarrow{\cong}
Gr^d_F{\un{\Omega}}^{\bullet}_X = {\un{\Omega}}^i_X [-d]$. On the other hand, 
we have for any $i \ge 0$,
\[
\begin{array}{lll}
\H_{cdh}\left(X, {\un{\Omega}}^i_X\right) & = &
\H_{cdh}\left(X, R{\pi}_*\Omega^i_{{X_{\bullet}}/{\C}}\right) \\
& = & \H_{cdh}\left(X_{\bullet}, \Omega^i_{{X_{\bullet}}/{\C}}\right) \\ 
& = & \H_{Zar}\left(X_{\bullet}, \Omega^i_{{X_{\bullet}}/{\C}}\right) \\
& = & \H_{Zar}\left(X, R{\pi}_*\Omega^i_{{X_{\bullet}}/{\C}}\right) \\
& = & \H_{Zar}\left(X, {\un{\Omega}}^i_X\right),
\end{array}
\]
where the third equality follows from \cite[Corollary~2.5]{CJams}
since $X_{\bullet}$ is smooth. This in particular proves the result for 
$F^{d}{\un{\Omega}}^{\bullet}_X$. Let us now assume that the lemma holds
for $F^{\ge i+1}{\un{\Omega}}^{\bullet}_X$ and consider the
commutative diagram of exact triangles 
\[
\xymatrix@C.5pc{
{\H_{Zar}\left(X, F^{i+1}{\un{\Omega}}^{\bullet}_X\right)} \ar[r] \ar[d] &
{\H_{Zar}\left(X, F^{i}{\un{\Omega}}^{\bullet}_X\right)} \ar[r] \ar[d] &
{\H_{Zar}\left(X, {\un{\Omega}}^i_X\right)[-i]} \ar[d] \\
{\H_{cdh}\left(X, F^{i+1}{\un{\Omega}}^{\bullet}_X\right)} \ar[r] &
{\H_{cdh}\left(X, F^{i}{\un{\Omega}}^{\bullet}_X\right)} \ar[r] & 
{\H_{cdh}\left(X, {\un{\Omega}}^i_X\right)[-i]}}
\]
The left vertical map is a weak equivalence by induction and we have just
shown that the right vertical map is also a weak equivalence. Hence so
is the middle vertical map.
\end{proof}  

For a $\C$-scheme $X$, let $\Omega^{\bullet}_{X/{\C}}$ denote the
filtered de Rham complex of $X$ with the Betti filtration
$F^i\Omega^{\bullet}_{X/{\C}} = \Omega^{\ge i}_{X/{\C}}[-i]$
and $Gr^i_F\Omega^{\bullet}_{X/{\C}} = \Omega^i_{X/{\C}}[-i]$.
Note that this is a finite filtration as $X$ is quasi-projective.
The morphism $X_{\bullet} \xrightarrow{\pi} X$ induces the natural
map of filtered complexes 
\begin{equation}\label{eqn:Db1} 
\left(\Omega^{\bullet}_{X/{\C}}, F\right) \to 
\left({\un{\Omega}}^{\bullet}_X, F\right),
\end{equation}
which gives rise to the map of filtered complexes 
\begin{equation}\label{eqn:Db2} 
\left(a^*\Omega^{\bullet}_{X/{\C}}, F\right) \xrightarrow{\theta_X} 
\left(a^*{\un{\Omega}}^{\bullet}_X, F\right).
\end{equation}
These maps are known ({\sl cf.} \cite{DB}) to be weak equivalences if $X$ is 
smooth. For singular schemes, they are related by the following. 
\begin{prop}{({\sl cf.} \cite[Lemma~7.1]{BassNK})}\label{prop:DB3}
The above map induces the weak equivalence
\[
{\H_{cdh}\left(X, F^i\Omega^{\bullet}_{X/{\C}}\right)} \xrightarrow{\cong}
{\H_{Zar}\left(X, F^{i}{\un{\Omega}}^{\bullet}_X\right)}  
\]
for every $i \ge 0$.
\end{prop}
\begin{proof}
By Lemma~\ref{lem:DB2}, we can replace the right hand side by the corresponding
$cdh$ hypercohomology. As in Lemma~\ref{lem:DB1}, we prove by
a descending induction on $i$. We first show the statement of the proposition
for the graded pieces. Since we are now working with the $cdh$ cohomology,
we can assume that $X$ is reduced. 
We use an induction on the dimension $d$ of $X$. If $d= 0$, then $X$ is
smooth and hence the statement is obvious. So suppose that $d \ge 1$.
Let $\wt{X} \xrightarrow{f} X$ be
a resolution of singularities of $X$. Let $S \inj X$ denote the singular locus
of $X$ and let $\wt{S} \inj \wt{X}$ be its inverse image. Then descent
property of the $cdh$ cohomology gives a commutative diagram of
exact triangles
\[
\xymatrix@C.5pc{
{\H_{cdh}\left(\wt{S}, \Omega^{i}_{{\wt{S}}/{\C}}\right)}[-1] \ar[r] \ar[d] &
{\H_{cdh}\left(X, \Omega^{i}_{X/{\C}}\right)} \ar[r] \ar[d] & 
{\H_{cdh}\left(\wt{X}, \Omega^{i}_{{\wt{X}}/{\C}}\right) \oplus
\H_{cdh}\left(S, \Omega^{i}_{S/{\C}}\right)} \ar[d] \\
{\H_{cdh}\left(\wt{S}, {\un{\Omega}}^i_{\wt{S}}\right)}[-1] \ar[r] &
{\H_{cdh}\left(X, {\un{\Omega}}^i_X\right)} \ar[r] &
{\H_{cdh}\left(\wt{X}, {\un{\Omega}}^i_{\wt{X}}\right) \oplus
\H_{cdh}\left(S, {\un{\Omega}}^i_S\right)},}
\]
where the left and the right vertical maps are weak equivalence by induction
on dimension and by the smoothness of $\wt{X}$. Hence the middle one is also
a weak equivalence.
Now suppose the proposition is proven for $F^{\ge i+1}$ and use the same 
argument as in the proof of Lemma~\ref{lem:DB2} to get the proof for $F^i$.
\end{proof}

We now collect some important consequences of the above comparison results.
\begin{cor}\label{cor:CDB1}
For any $\C$-scheme $X$ and any $n \ge 0$, there is a natural isomorphism
\[
{\H^n_{cdh}\left(X, \Omega^{\bullet}_{X/{\C}}\right)}
\xrightarrow{\cong} H^n\left(X, \C\right).
\]
If $X$ is projective, there is a spectral sequence 
\[
E^{p,q}_1 = H^q_{cdh}\left(X, \Omega^p_{X/{\C}}\right) \Rightarrow
H^{p+q}(X, \C).\]
Moreover, this spectral sequence degenerates and the induced
filtration on $H^*(X, \C)$ coincides with its Hodge filtration. 
\end{cor}
\begin{proof}
It follows directly from Proposition~\ref{prop:DB3} and
\cite[Theorem~4.5]{DB}.
\end{proof}
\begin{cor}\label{cor:CDB2}
For a projective $\C$-scheme $X$, the Hodge filtration on $H^n(X, \C)$
is given by
\[
H^n(X, \C) = F^0H^n(X, \C) \supseteq \cdots F^nH^n(X, \C) \supseteq
F^{n+1}H^n(X, \C) = 0, \ {\rm where}
\]
\[
\frac{F^{i}H^n(X, \C)}{F^{i+1}H^n(X, \C)} \cong 
H^{n-i}_{cdh}\left(X, \Omega^i_{X/{\C}}\right).
\] 
Moreover, there is a natural isomorphism
\[
{\H^n_{cdh}\left(X, \Omega^{\ge i}_{X/{\C}}[-i]\right)}
\xrightarrow{\cong} \H_{cdh}\left(X, F^i{\un{\Omega}}^{\bullet}_X\right)
\xrightarrow{\cong} F^iH^n\left(X, \C\right)
\]
for all $n, i \ge 0$.
\end{cor}
\begin{proof}
The first assertion follows at once from Corollary~\ref{cor:CDB1}.
For the second, we use the isomorphism $F^i{\un{\Omega}}^{\bullet}_X \cong
R{\pi}_*\left(\Omega^{\ge i}_{{X_{\bullet}}/{\C}}[-i]\right)$ from 
~\eqref{eqn:Db00}, which gives
\[
\begin{array}{lll}
\H^n_{cdh}\left(X, \Omega^{\ge i}_{X/{\C}}[-i]\right) & 
\cong & \H^{n}_{Zar}\left(X, F^i{\un{\Omega}}^{\bullet}_X\right) \\
& \cong & 
\H^{n}_{Zar}\left(X, R{\pi}_*\left(\Omega^{\ge i}_{{X_{\bullet}}/{\C}}[-i]
\right)\right) \\
& \cong & 
\H^{n}_{Zar}\left(X_{\bullet}, \Omega^{\ge i}_{{X_{\bullet}}/{\C}}[-i]\right)
\\
& \cong & F^iH^n\left(X, \C\right),
\end{array}
\]
where the first isomorphism follows from Proposition~\ref{prop:DB3}
and the last isomorphism follows from \cite[(1),(3)]{BPW}. This proves the
second assertion.
\end{proof}

Another consequence of Proposition~\ref{prop:DB3} is the following
simple criterion for the Du Bois singularities of a variety. 

\begin{cor}\label{cor:DBSing}
Let $X$ be a quasi-projective variety over a field $k$ of characteristic
zero. Then $X$ has {\sl Du Bois} singularities if and only if
\[
H^i_{Zar}\left(X, \sO_X\right) \xrightarrow{\cong}
H^i_{cdh}\left(X, \sO_X\right)
\]
for all $i \ge 0$.
\end{cor}
\begin{proof} Recall from \cite{Steenbrink} that $X$ is said to have
Du Bois singularity if the natural map $\sO_X \to {\un{\Omega}}^0_X$
is a quasi-isomorphism. By the Lefschetz pencil argument 
({\sl cf.} \cite{Schwede}), we can assume the base field to be $\C$.
But it is then an immediate consequence of Proposition~\ref{prop:DB3}.
\end{proof}
The following well known fact is a simple consequence of 
Proposition~\ref{prop:SNC-local} and Corollary~\ref{cor:DBSing}.
\begin{cor}\label{cor:DBSNC}
A normal crossing singularity is Du Bois.
\end{cor}

\begin{remk}\label{remk:canonical}
There have been a lot of work by various people on Du Bois singularities,
partially due to the complicated nature of the Du Bois complex 
${\un{\Omega}}^0_X$. One hopes that the above criterion of such a 
singularity in terms of the $cdh$ cohomology of the structure sheaf might
play a useful role in the study of Du Bois singularity. For example,
it was conjectured by Kollar (proven now by Kollar and Kov{\'a}cs \cite{KK})
that the log canonical singularity is Du Bois. One could ask if the above 
criterion would help in simplifying the proof of \cite{KK}.
\end{remk}      

\subsection{Deligne complexes}\label{subsection:D-complex}
Recall from \cite{EV} and \cite[Section~1]{BPW} that for a projective
$\C$-scheme $X$ and for the morphism of analytic sites $X_{{\bullet} an}
\xrightarrow{\pi} X_{an}$, the Deligne complex of $X$ is defined as the
complex $\Z_{\sD}(q) = R{\pi}_*\left(\Z_{{\sD}, X_{\bullet}}(q)\right)$,
where
\begin{equation}\label{eqn:DeligneC0}
\Z_{{\sD}, X_{\bullet}}(q) : =
\left(\Z(q) \xrightarrow{(2{\pi}i)^q} \sO_{X_{\bullet}} \to 
\Omega^1_{X_{\bullet}} \to \cdots \to \Omega^{q-1}_{X_{\bullet}}\right)
\end{equation} 
is the Deligne complex of the smooth simplicial scheme $X_{\bullet}$.
In particular, there is an exact triangle
\begin{equation}\label{eqn:DeligneC1} 
 R{\pi}_*\left(\Omega^{<q}_{X_{\bullet}}\right)[-1] \to \Z_{\sD}(q) \to
\Z(q).
\end{equation}
It follows at once (using GAGA) from ~\eqref{eqn:Db02} and 
Proposition~\ref{prop:DB3} that there is an exact triangle 
\begin{equation}\label{eqn:DeligneC2}
\left(Ra_* \Omega^{<q}_{X/{\C}}\right)[-1] \to 
Rb_*\Z_{\sD}(q) \to Rb_*\Z(q),
\end{equation}
where $X_{an} \xrightarrow{b} X_{Zar}$ is the usual morphism of sites.
The Deligne cohomology groups of $X$ are defined by 
$H^n_{\sD}\left(X, \Z(q)\right) := \H^n\left(X_{an}, \Z_{\sD}(q)\right)$.
 
For singular $\C$-schemes, Levine \cite{Levine2} had introduced a modified
version of the classical Deligne cohomology. It turns out that there are
Chern class maps from the algebraic $K$-theory to this 
modified Deligne cohomology which detect more nontrivial elements in
the $K$-groups of the singular schemes than the above classical Deligne
cohomology. We refer to \cite{Krishna1} for some applications of the
Chern classes into the modified Deligne cohomology. For a projective
$\C$-scheme $X$, the {\sl modified Deligne cohomology} groups 
$H^*_{\sD^*}\left(X, \Z(q)\right)$ are defined as the analytic hypercohomology
of the truncated complex
\begin{equation}\label{eqn:MDeligneC1}
\Z_{\sD^*}(q) : =
\left(\Z(q) \xrightarrow{(2{\pi}i)^q} \sO_{X} \to 
\Omega^1_{X/{\C}} \to \cdots \to \Omega^{q-1}_{X/{\C}}\right)
\end{equation} 
For a morphism $Y \xrightarrow{f} X$ of schemes, the relative modified
Deligne cohomology groups $\H^n_{\sD^*}\left((X,Y), \Z(q)\right)$
are defined as the hypercohomology of the
complex $Cone\left(\Z_{\sD^*, X}(q) \to Rf_*\Z_{\sD^*, Y}(q)\right)[-1]$. 

It is clear from the definition that the modified Deligne cohomology agrees
with the classical one for smooth projective varieties. In particular,
a smooth proper hypercovering $X_{\bullet} \xrightarrow{\pi} X$ defines
a natural map $\Z_{\sD^*}(q) \to \Z_{\sD}(q)$. Moreover, it follows from
~\eqref{eqn:DeligneC2} that there is a commutative diagram of exact
triangles
\begin{equation}\label{eqn:MDeligneC2}
\xymatrix@C.5pc{
\Omega^{<q}_{X/{\C}}[-1] \ar[r] \ar[d] & Rb_*\Z_{\sD^*}(q) \ar[r] \ar[d] &
Rb_*\Z(q) \ar@{=}[d] \\
\left(Ra_* \Omega^{<q}_{X/{\C}}\right)[-1] \ar[r] & 
Rb_*\Z_{\sD}(q) \ar[r] & Rb_*\Z(q).}
\end{equation}

\begin{lem}\label{lem:MDeligneC3}
There is a commutative diagram of long exact sequences 
\[
\xymatrix@C.4pc{
\cdots \ar[r] & \H^n_{Zar}\left(X, \Omega^{<q}_{X/{\C}}\right) 
\ar[d]^{\alpha^q_n} \ar[r] &
\H^{n+1}_{\sD^*}\left(X, \Z(q)\right) \ar[r] \ar[d] &
H^{n+1}\left(X, \Z\right) \ar[r] \ar[d] &
H^{n+1}_{Zar}\left(X, \Omega^{<q}_{X/{\C}}\right) \ar[d]^{\alpha^{q}_{n+1}}
& \cdots \\
\cdots \ar[r] & {\frac{H^{n}\left(X, \C\right)}{F^qH^{n}\left(X, \C\right)}}
\ar[r] & H^{n+1}_{\sD}\left(X, \Z(q)\right) \ar[r] &
H^{n+1}\left(X, \Z\right) \ar[r] &
{\frac{H^{n+1}\left(X, \C\right)}{F^qH^{n+1}\left(X, \C\right)}} \ar[r] &
\cdots}
\]
where the $\alpha^q_n$ are all surjective.
\end{lem}
\begin{proof} 
The exact sequences and commutative diagram follow from ~\eqref{eqn:MDeligneC2}
and Corollary~\ref{cor:CDB2}. To prove the required surjectivity, we consider
the commutative diagram
\[
\xymatrix@C.4pc{
& & \H^n_{Zar}\left(X, \Omega^{\bullet}_{X/{\C}}\right) \ar[r] \ar[d] &
\H^n_{Zar}\left(X, \Omega^{<q}_{X/{\C}}\right) \ar[d]^{\alpha^q_n} & \\
0 \ar[r] & F^qH^{n}\left(X, \C\right) \ar[r] & H^{n}\left(X, \C\right) \ar[r] &
{\frac{H^{n}\left(X, \C\right)}{F^qH^{n}\left(X, \C\right)}} \ar[r] & 0.}
\]
The bottom sequence is exact and 
${\frac{H^{n}\left(X, \C\right)}{F^qH^{n}\left(X, \C\right)}} \cong
H^{n}_{cdh}\left(X, \Omega^{<q}_{X/{\C}}\right)$
by Corollary~\ref{cor:CDB2}. 
On the other hand, the composite map ${\C}_{X_{an}} \to 
\Omega^{\bullet}_{X_{an}}
\to R{\pi}_*\Omega^{\bullet}_{X_{\bullet} an}$ is an isomorphism by the
cohomological descent of analytic cohomology and the de Rham theorem. In 
particular, the left
vertical map is split surjective. Hence the right vertical map is 
surjective too. 
\end{proof}

\subsection{Intermediate Jacobians and Abel-Jacobi maps}\label{subsection:IJ}
For a projective $\C$-scheme $X$ of dimension $d$, the classical $p$th 
{\sl intermediate Jacobian} is defined as 
\begin{equation}\label{eqn:CIJ}
J^p(X) = \frac{H^{2p-1}\left(X, \C(p)\right)}
{F^pH^{2p-1}\left(X, \C(p)\right) +
H^{2p-1}\left(X, \Z(p)\right)} \cong 
\frac{H^{2p-1}_{cdh}\left(X, \Omega^{<p}_{X/{\C}}\right)}
{H^{2p-1}\left(X, \Z(p)\right)},
\end{equation}
where the second isomorphism follows from Corollary~\ref{cor:CDB2}.
If $X$ is smooth, the intermediate Jacobian can also be written as
$J^p(X) = \frac{H^{2p-1}(X, \R)}{H^{2p-1}(X, \Z)}$. In particular,
this is a real torus. In fact, one knows that this has a complex structure
which makes it a complex torus. In general, $J^p(X)$ does not admit
a polarization, i.e., it is not an abelian variety. However,
in case $H^{i,j}(X) = 0$ for $|i-j| > 1$ and $i+j = 2p-1$, then, $J^p(X)$
is indeed an abelian variety ({\sl cf.} \cite[p.171, 172]{Lewis}).
We conclude in particular that $J^{d-1}(X)$ is an abelian variety if 
$H^d(X, \Omega^i_X) = 0$ for $0 \le i \le d-2$.    

For a smooth and projective $\C$-scheme $X$ of dimension $d$ and $p \ge 0$, let
$A^p(X) = {CH^p(X)}_{alg}$ be the subgroup of $CH^i(X)$ consisting of
algebraic cycles which are algebraically equivalent to zero.
Let ${CH^p(X)}_{hom}$ be the subgroup of homologically trivial cycles in
$CH^p(X)$, i.e., this is simply the kernel of the topological cycle class map
$CH^p(X) \to H^{2p}(X, \Z)$. One knows that $A^p(X) \subseteq 
{CH^p(X)}_{hom} \subseteq CH^p(X)$. There are Abel-Jacobi maps
\begin{equation}\label{eqn::Alg0*}
{CH^p(X)}_{hom} \xrightarrow{AJ^p} J^p(X).
\end{equation}

Let $J^p_a(X)$ denote the image of $A^p(X)$ under the Abel-Jacobi map.
Then $J^p_a(X)$ is an abelian subvariety of $J^p(X)$. 
Recall that the famous Hodge conjecture says that the topological
cycle class map ${CH^p(X)}_{\Q} \to H^{2p}(X, \Q) \cap H^p(X, \Omega^p_X)$
is surjective. This conjecture is known for $p \in \{0, 1, d-1, d\}$.  
We now recall the general Hodge conjecture. 

For any $l \ge 1$ and $p \ge 0$, let 
\[
F^p_aH^l(X, \Q) = 
{\underset{{{\rm codim}_X(Y) \ge \ p}}{\bigcup}}
{\rm Ker}\left(H^l(X, \Q) \to H^l(X-Y, \Q)\right).
\]
It is known that $F^p_aH^l(X, \Q)$ has a Hodge structure and 
\[
F^p_aH^l(X, \Q) \subseteq F^pH^l(X, \C) \cap H^l(X, \Q).
\]
The general Hodge conjecture says the following.
\begin{conj}{$\bf{(GHC(p,l,X))}$}\label{conj:GHC*} 
$F^p_aH^l(X, \Q)$ is the largest sub-Hodge structure of $F^pH^l(X, \C) \cap
H^l(X, \Q)$.
\end{conj}
There are few cases when this conjecture is known. On case which interests
us is when $X$ is a hypersurface in $\P^4$ of degree $\le 4$.
In this case, $GHC(1,3,X)$ is known ({\sl cf.} \cite[Chapter~7.IX]{Lewis}).
It is also known for any smooth projective variety $X$ that
$GHC(p-1, 2p-1, X)$ is equivalent to saying that $J^p_a(X)$ is the
largest abelian subvariety of $J^p(X)$. We conclude:
\begin{cor}\label{cor:GHC*1}
Let $X$ be a smooth projective variety of dimension $d$ such that
$H^d(X, \Omega^i_X) = 0$ for $0 \le i \le d-2$. Then
\[
GHC(d-2, 2d-3, X) \Leftrightarrow J^{d-1}_a(X) = J^{d-1}(X).
\]
\end{cor} 

Now we come back to the general case of singular projective schemes.
Suppose $X$ is a singular projective $\C$-scheme of dimension $d$.
If $\wt{X} \to X$ is a resolution of singularities of $X$, then the
isomorphism $H^{2d}\left(X, \Z(d)\right) \xrightarrow{\cong} 
H^{2d}(\wt{X}, \Z(d))$ and Lemma~\ref{lem:MDeligneC3}
yields an exact sequence 
\begin{equation}\label{eqn:InterJ}
0 \to J^d(X) \to H^{2d}_{\sD}\left(X, \Z(q)\right) \to 
H^{2d}\left(X, \Z(d)\right) \to 0.
\end{equation} 

We define the $p$th {\sl generalized intermediate Jacobian} as
\begin{equation}\label{eqn:GIJ}
J^p_*(X) = \frac{H^{2p-1}_{Zar}\left(X, \Omega^{<p}_{X/{\C}}\right)}
{H^{2p-1}\left(X, \Z(p)\right)}.
\end{equation} 
It follows from Lemma~\ref{lem:MDeligneC3} that the natural map
$J^p_*(X) \to J^p(X)$ is surjective. $J^d_*(X)$ and $J^d(X)$ are
also called the {\sl generalized albanese} and the {\sl albanese}
varieties of $X$. In fact, one knows from \cite[Theorem~2]{ESV} that 
$J^d_*(X)$ is a commutative algebraic group over $\C$ and $J^d(X)$ is 
its universal semi-abelian quotient variety. Recall from \cite{ESV}
that the Chow group of zero-cycles $CH^d(X)$ is defined as the 
free abelian group of smooth closed points of $X$ modulo the subgroup
generated by the cycles defined by the rational functions on all the 
{\sl Cartier} curves on $X$. We refer to {\sl loc. cit.} for the complete
definition. Let $A^d(X) = {CH^d(X)}_{{\rm deg} 0}$ denote the kernel
of the map $CH^d(X) \xrightarrow{deg} H^{2d}(X, \Z)$. We have seen in
~\eqref{eqn:SALBS} that there a regular map $A^d(X) \xrightarrow{c_{0,X}}  
J^d_*(X)$ which is a universal regular quotient.

\section{Chern Classes on $KH$-theory}\label{section:CH-KH}
Recall from \cite{Beilinson} (see also \cite{Gillet}, \cite{BPW}) that for a 
$\C$-scheme $X$, there are natural Chern classes
\begin{equation}\label{eqn:Chern-C}
c_{q,p} : K_{2q-p}(X) \to \H^p_{\sD}\left(X, \Z(q)\right),
\end{equation}
which have all the functorial properties with respect to the pull-back
maps on $K$-theory and Deligne cohomology. These functorial properties
also define such Chern class maps from the relative $K$-groups to the
relative Deligne cohomology. Gillet \cite{Gillet} has constructed
universal Chern classes into generalized cohomology theories of which the
above is a special case.  
It was shown in \cite{Levine2} that the
modified Deligne cohomology also satisfies Gillet's conditions for being
a generalized cohomology theory and hence there are functorially
defined Chern classes 
\begin{equation}\label{eqn:Chern-G}
c^*_{q,p} : K_{2q-p}(X) \to \H^p_{\sD^*}\left(X, \Z(q)\right)
\end{equation}
such that the diagram

\begin{equation}\label{eqn:Chern1}
\xymatrix{
K_{2q-p}(X) \ar[r]^{c^*_{q,p}} \ar@/^1cm/[rr]^{c_{q,p}} 
\ar@/_1cm/[rrr]_{c^{top}_{q,p}} & \H^p_{\sD^*}\left(X, \Z(q)\right) \ar[r] &
\H^p_{\sD}\left(X, \Z(q)\right) \ar[r] & H^p\left(X, \Z(q)\right)}
\end{equation}
commutes, where $c^{top}_{q,p}$ is the topological Chern class.

\begin{prop}\label{prop:CKH}
For any $\C$-scheme $X$, there are Chern class maps 
\[
c^h_{q,p}: KH_{2q-p}(X) \to \H^p_{\sD}\left(X, \Z(q)\right)
\]
such that the composite $K_{2q-p}(X) \to KH_{2q-p}(X) \to 
\H^p_{\sD^*}\left(X, \Z(q)\right)$ is the Chern class map of 
~\eqref{eqn:Chern-C}. Furthermore, these Chern classes are functorial
with the pull-back maps of (relative) $KH$-theory and the Deligne
cohomology.
\end{prop}
\begin{proof}
The Chern classes from the homotopy invariant $K$-theory are obtained 
by comparing it with the descent $K$-theory of \cite{PP}.
It was shown in \cite[Theorem~4.1]{PP} that the algebraic $K$-theory
functor $\sK : (Sm/{\C})^{op} \to HoSp$ from the category of smooth
schemes over $\C$ to the homotopy category of spectra is a functor to a
{\sl descent} category in the sense of Guill{\'e}n-Navarro \cite{GuillenN}.
Hence it uniquely
extends to a functor ${\sK}{\sD} : (Sch/{\C})^{op} \to HoSp$ which
satisfies the descent property in the sense that if $X_{\bullet} 
\xrightarrow{\pi} X$ is a Guill{\'e}n-Navarro smooth proper hypercubical 
resolution of $X$, then ${\sK}{\sD}(X)
\xrightarrow{\cong} \sK(X_{\bullet})$. One defines $KD_i(X) =
{\pi}_i\left({\sK}{\sD}(X)\right)$. 

Using the $cdh$-descent property
of $KH$-theory ({\sl cf.} \cite{Haes}) and the uniqueness of the
descent $K$-theory, it was shown in {\sl loc. cit.} that there is a 
weak equivalence $KH(X) \cong KD(X)$ of spectra. To see that the Chern
class maps of ~\eqref{eqn:Chern-C} descend to $KD(X)$, we choose a proper 
smooth hypercubical resolution $X_{\bullet} \xrightarrow{\pi} X$ as above.
This gives natural maps 
\[
\xymatrix@C.9pc{
KD_{2q-p}(X) \ar[r]^{\cong} \ar[d]_{c^h_{2q-p}} &
KD_{2q-p}(X_{\bullet}) \ar[r]^{\cong} &
K_{2q-p}(X_{\bullet}) \ar[d]^{c_{q,p}} \\
\H^p_{\sD}\left(X, \Z(q)\right) \ar[rr]_{\cong} & & 
\H^p_{\sD}\left(X_{\bullet}, \Z(q)\right)} 
\]
giving the desired factorization.
The functoriality of $c^h$ now easily follows from the above
using similar properties the usual Chern classes of smooth schemes. 
\end{proof}  

Recall that the descent spectral sequence ~\eqref{eqn:SSC} induces
a functorial Brown filtration $F^{\bullet}_B$ on the $KH$-theory.
Since we shall be considering only this filtration on $KH_*(X)$, we
shall drop the subscript and simply write $F^{\bullet}KH_*(X)$.

\subsection{Milnor and Quillen $K$-theory}
Recall that for a $\C$-algebra $A$, the Milnor $K$-theory $K^M_*(A)$ is the
quotient of the tensor algebra $T(A^*)$ of units in $A$ over $\Z$ by
the two-sided ideal generated by homogeneous elements $\{a \otimes (1-a)|
a, 1-a \in A^*\}$.
For any variety $X$ over $\C$, let ${\sK}^M_{m, X}$
denote the sheaf of Milnor $K$-groups on $X$. This is the sheaf
associated to the presheaf which on affine open subsets of $X$ is given by 
above. In particular,
its stalk at any point $x$ of $X$ is the Milnor 
$K$-group of the local ring ${\sO}_{X, x}$. For any closed embedding 
$i : Y \inj X$, let ${\sK}^M_{m, (X, Y)}$ 
be the sheaf of relative Milnor $K$-groups defined so that the sequence
of sheaves
\begin{equation}\label{eqn:M1} 
0 \to {\sK}^M_{m, (X, Y)} \to {\sK}^M_{m, X} \to
i_*({\sK}^M_{m, Y}) \to 0
\end{equation}
is exact. Note that the map
${\sK}^M_{m, X} \to i_*({\sK}^M_{m, Y})$ is always surjective.

There is a natural map of $K$-theory sheaves $\sK^M_{*, X} \to
\sK_{*,X}$ for any variety $X$, and it is known ({\sl cf.} \cite{Soule})
that this map is injective up to torsion and the Milnor $K$-sheaves
are the smallest piece of the gamma filtration on the corresponding
Quillen $K$-sheaves.   
It is well known ({\sl loc. cit.}) that the Chow groups of algebraic cycles on 
smooth varieties can also be described as the cohomology of Milnor 
$\sK^M$-sheaves.
The results of \cite{Krishna3} and \cite{Krishna2} suggest that
similar identifications should be valid for singular varieties as well.
In fact, we strongly suspect the following which seems to be believed
by experts.

\begin{conj}\label{conj:MQ}
For a quasi-projective scheme $X$ of dimension $d$ over an algebraically 
closed field $k$ of characteristic zero, there is an isomorphism
\[
CH^d(X) \cong H^d_{Zar}\left(X, \sK^M_d\right).
\]
\end{conj}
One also believes that $H^d_{Zar}\left(X, \sK^M_d\right)$ is in general
smaller than $H^d_{Zar}\left(X, \sK_d\right)$. However, there is one case  
where they coincide. We refer to \cite[Corollary~4.2]{Krishna3} for a 
proof.

\begin{prop}\label{prop:MQC}
Let $X$ be an affine or projective variety of dimension $d$ over $\C$ with
only isolated singularities. Then 
\[
CH^d(X) \cong H^d_{Zar}\left(X, \sK^M_d\right) \cong 
H^d_{Zar}\left(X, \sK_d\right).
\]
\end{prop}

We shall also need the following comparison result for the cohomology
of Milnor and Quillen $K$-sheaves on smooth schemes.

\begin{lem}\label{lem:MQC-top}
Let $X$ be a smooth $\C$-scheme of dimension $d$. Then for any $i \ge 0$,
\[
H^j_{Zar}\left(X, \sK^M_{i}\right) \xrightarrow{\cong}
H^j_{Zar}\left(X, \sK_{i}\right) \ {\rm for} \ j \ge i-1.
\]
\end{lem}
\begin{proof}
This follows directly by comparing the Gersten resolutions for appropriate
$\sK^M$ and $\sK$-sheaves and using the fact that $K^M_i(R) \cong K_i(R)$
for any local ring $R$ over $\C$ and for $0 \le i \le 2$. The
Gersten resolution for the Milnor $K$-sheaves is proven in \cite{Kerz}.
\end{proof}   

For a quasi-projective scheme $X$ of dimension $d$ over $\C$ and $i \ge 0$, let
\begin{equation}\label{eqn:MQC*}
F^dK^M_i(X) := {\rm Image}\left(H^d_{Zar}\left(X, \sK^M_{d+i}\right) 
\to K_i(X)\right), \ {\rm and}
\end{equation}
\[
F^dKH^M_i(X) := {\rm Image}\left(H^d_{cdh}\left(X, \sK^M_{d+i}\right) 
\to KH_i(X)\right).
\]
Note that these maps are induced by the natural maps from Milnor 
sheaves to Quillen sheaves followed by the maps induced by the
Brown-Gersten spectral sequences.

For the rest of this paper, we choose and fix the following resolution
of singularities diagram for a projective $\C$-scheme $X$ of dimension $d$.
\begin{equation}\label{eqn:RSD}
\xymatrix@C.9pc{
E \ar[r]^{\wt{i}} \ar[d]_{\wt{f}} & \wt{X} \ar[d]^{f} \\
S \ar[r]_{i} & X,}
\end{equation}
where $S = X_{\rm sing}$ and $E = f^{-1}(S)$ is the reduced exceptional
divisor which is assumed to be strict normal crossing. We recall the
following Mayer-Vietoris property of the Deligne cohomology from
\cite[Variant~3.2]{BPW}.
\begin{lem}\label{lem:MVD}{({\cite[Variant~3.2]{BPW}})} The above resolution 
diagram induces the following long exact sequence of Deligne cohomology.
\[
\cdots \to \H^i_{\sD}\left(X, \Z(q)\right) \to  
\H^i_{\sD}(\wt{X}, \Z(q)) \oplus  \H^i_{\sD}\left(S, \Z(q)\right)  
\to \H^i_{\sD}\left(E, \Z(q)\right) \to
\H^{i+1}_{\sD}\left(X, \Z(q)\right).
\]
\end{lem}
Such a Mayer-Vietoris exact sequence also holds for the singular cohomology.
The functoriality property of the Chern classes with the pull-back maps of
(relative) $K$-theory and Deligne cohomology gives such maps
between the Mayer-Vietoris sequences of $KH$-theory and Deligne cohomology. 

\section{Chern classes for normal crossing schemes}
\label{section:KDC} 
In this section, we prove some results about the $cdh$
cohomology of $\sK$-sheaves and the Chern classes from them into the
Deligne cohomology of normal crossing schemes.

\begin{lem}\label{lem:SNCD}
Let $E$ be a strict normal crossing divisor of dimension $d$ on a smooth
scheme. Then the cup product map 
\[
\H^{2d}_{\sD}\left(E, \Z(d)\right) \otimes \H^{1}_{\sD}\left(E, \Z(1)\right)
\to \H^{2d+1}_{\sD}\left(E, \Z(d+1)\right)
\]
is surjective.
\end{lem}
\begin{proof}
We can assume that $E$ is connected. 
It is easy to see that $\H^{1}_{\sD}\left(E, \Z(1)\right) \cong 
H^0(E, \C^*) \cong \C^*$ ({\sl cf.} \cite[Section~1]{BPW}).
%Easy to prove this using induction on number of components of $E$ and
%by comparing the mayer-vietoris exact sequences cohomology of $\C^*$
%and $\Z_{\sD}(1)$. 
It follows from Lemma~\ref{lem:MDeligneC3} that 
\[
\H^{2d+1}_{\sD}\left(E, \Z(d+1)\right) = \frac{{H^{2d}(E, \C)}/{H^{2d}(E, \Z)}}
{F^{d+1}H^{2d}(E, \C)} \cong \frac{H^{2d}(E, \C^*)}{F^{d+1}H^{2d}(E, \C)}\]
\[
\hspace*{8cm} \cong 
\frac{H^{2d}(E, \Z) \otimes \C^*}{F^{d+1}H^{2d}(E, \C)}.
\]
On the other hand, if $E^N \to E$ is the smooth normalization of $E$,
then the map $H^{2d}(E, \Z) \to H^{2d}(E^N, \Z)$ is an
isomorphism of Hodge structures. Moreover, $F^{d+1}H^{2d}(E^N, \C) = 0$
by Hodge theory. In particular, we get $F^{d+1}H^{2d}(E, \C) = 0$.
We conclude that
\begin{equation}\label{eqn:strictN}
H^{2d}(E, \Z(d)) \otimes \C^* \xrightarrow{\cong}
\H^{2d+1}_{\sD}\left(E, \Z(d+1)\right).
\end{equation}   
In particular, we get 
\[
\xymatrix{
\H^{2d}_{\sD}\left(E, \Z(d)\right) \otimes \H^{1}_{\sD}\left(E, \Z(1)\right)
\ar[r]^{\ \ \ \ \cong} \ar[d] & \H^{2d}_{\sD}\left(E, \Z(d)\right) \otimes \C^*
\ar@{->>}[d] \\
\H^{2d+1}_{\sD}\left(E, \Z(d+1)\right) &
H^{2d}(E, \Z) \otimes \C^* \ar[l]^{\cong},}
\] 
which shows that the left vertical arrow is surjective.
\end{proof}

We next recall the following result which signifies the importance of
using Milnor $K$-sheaves in place of Quillen $K$-sheaves to study 
algebraic cycles on singular varieties. We refer to 
\cite[Proposition~8.2]{Krishna2} for a proof. It is not clear if such a 
result is true for the cohomology of Quillen $K$-sheaves.

\begin{prop}\label{prop:Milnor}
Let $E$ be as in Lemma~\ref{lem:SNCD}. Then the natural cup product map
\[H^d(E, {\sK}^M_{i, E}) \otimes \C^* \to H^d(E, {\sK}^M_{{i+1}, E})\]
is surjective for all $i \ge d$ in either of Zariski and $cdh$ topology. 
\end{prop}
%Using the equivalence of $K \cong KH$, it is easy to show that the Zariski
%and $cdh$ cohomology of $K$-sheaves on smooth varieties agree. Now, one
%can use \cite[Proposition~8.3]{Krishna2} to prove the same for the Milnor
%$K$-sheaves rationally. But we only need these results rationally.
We remark here that the proof of the above proposition in 
\cite[Proposition~8.2]{Krishna2} is given for the Zariski cohomology,
but the same (in fact easier) proof also works for the $cdh$ cohomology.

\begin{cor}\label{cor:KtheoryCup}
Let $E$ be as in Lemma~\ref{lem:SNCD}. Then the cup product maps in $K$-theory
induce the following diagram
\begin{equation}\label{eqn:KtheoryCup1}
\xymatrix@C.6pc{
H^d_{cdh}\left(E, \sK^M_d\right) \otimes \C^* \ar[r] \ar[d] &
H^d_{cdh}\left(E, \sK^M_{d+1}\right) \ar[d] \\
F^dKH^M_0(E) \otimes \C^* \ar[r] & F^dKH^M_1(E),}   
\end{equation}
where all the arrows are surjective.
\end{cor}
\begin{proof}
Follows immediately from Proposition~\ref{prop:Milnor} and ~\eqref{eqn:MQC*}.
\end{proof}

\begin{lem}\label{lem:SNCD1}
Let $E$ be as in Lemma~\ref{lem:SNCD}. Then the Chern class maps
\[
H^d_{cdh}\left(E, {\sK}^M_{d}\right) \to \H^{2d}_{\sD}\left(E, \Z(d)\right) 
\]
\[
H^d_{cdh}\left(E, {\sK}^M_{d+1}\right) \to 
\H^{2d+1}_{\sD}\left(E, \Z(d+1)\right) 
\]
are surjective.
\end{lem}
\begin{proof}
We prove the first assertion by induction on dimension of $E$. 
If $E$ is smooth, then we can replace $\sK^M$ by $\sK$ using 
Lemma~\ref{lem:MQC-top}.
The assertion is then standard. If $E$ has dimension zero, then it is smooth. 
So assume that
the result holds for all strict normal crossing divisors of dimension
less than $d$ which is at least one. 
Let $E^N \xrightarrow{f} E$ be the normalization map and 
let $\wt{S} = f^{-1}(S= E_{\rm sing})$. Then we have the commutative diagram 
of exact sequences
\begin{equation}\label{eqn:KH*12}
\xymatrix@C.6pc{
H^{d-1}_{cdh}\left(\wt{S}, {\sK}^M_{d}\right) \ar[d] \ar[r] &  
H^{d}_{cdh}\left(E, {\sK}^M_{d}\right) \ar[r] \ar[d] &
H^{d}_{cdh}\left(E^N, {\sK}^M_{d}\right) \ar[r] \ar[d] & 0 \\
\H^{2d-1}_{\sD}\left(\wt{S}, \Z(d)\right) \ar[r] &  
\H^{2d}_{\sD}\left(E, \Z(d)\right) \ar[r] &
\H^{2d}_{\sD}\left(E^N, \Z(d)\right) \ar[r] & 0.}
\end{equation}
Since $\wt{S}$ is a strict normal crossing divisor on $E^N$ (which is
smooth), the map $H^{d-1}_{cdh}\left(\wt{S}, {\sK}^M_{d-1}\right) \to
\H^{2d-2}_{\sD}\left(\wt{S}, \Z(d-1)\right)$ is surjective by induction.
Now the commutative diagram
\begin{equation}\label{eqn:KH*13}
\xymatrix@C.6pc{
H^{d-1}_{cdh}\left(\wt{S}, {\sK}^M_{d-1}\right) \otimes \C^*
\ar[r] \ar@{->>}[d] &
H^{d-1}_{cdh}\left(\wt{S}, {\sK}^M_{d}\right) \ar[d] \\
\H^{2d-2}_{\sD}\left(\wt{S}, \Z(d-1)\right) \otimes \C^* \ar[r] &
\H^{2d-1}_{\sD}\left(\wt{S}, \Z(d)\right)}
\end{equation}
and Lemma~\ref{lem:SNCD} imply that the left vertical map 
in ~\eqref{eqn:KH*12} is surjective. Hence so is the middle vertical map.
This proves the first surjectivity of the lemma. The second surjectivity
now follows from the first, Lemma~\ref{lem:SNCD} and the commutative
diagram ~\eqref{eqn:KH*13} for $E$.
\end{proof}

It follows from \cite[Theorem~2]{ESV} and \cite[Lemma~2.2]{Krishna1}
that the Chern class map in ~\ref{eqn:Chern-G} from the algebraic 
$K$-theory to the modified Deligne cohomology gives rise to a
commutative diagram of exact sequences
\begin{equation}\label{eqn:AlbKA1}
\xymatrix@C.6pc{
0 \ar[r] & F^dK_0(X)_{{\rm deg} 0} \ar[r] \ar@{->>}[d]^{Alb_{X}} & F^dK_0(X) 
\ar[r]^{c^{top}_0} \ar@{->>}[d]^{c_0} & H^{2d}(X, \Z) \ar@{=}[d] \ar[r] & 0 \\
0 \ar[r] & J^d_*(X) \ar[r] & H^{2d}_{\sD^*}\left(X, \Z(d)\right) \ar[r] &
H^{2d}(X, \Z) \ar[r] & 0.}
\end{equation}
In fact, one knows from \cite[Theorem~2]{ESV} that $J^d_*(X)$ is a
smooth and commutative algebraic group which is a universal regular
quotient of $A^d(X) = F^dK_0(X)_{{\rm deg} 0}$ in the category of
smooth commutative group algebraic groups over $\C$. The following is the
$cdh$ analogue of this albanese map.

For a smooth and projective $\C$-scheme $X$ of dimension $d$ and $i \ge 0$, let
\begin{equation}\label{eqn:Alg0}
{H^i_{cdh}\left(X, \sK^M_i\right)}_{hom}
= {\rm Ker}\left(H^i_{cdh}\left(X, \sK^M_i\right) \to 
H^{2i}(X, \Z(i))\right).
\end{equation}
Note that this group is same as the subgroup of ${CH^i(X)}_{hom}$ by
Lemma~\ref{lem:MQC-top}.
We shall also write ${H^d_{cdh}\left(X, \sK^M_d\right)}_{hom}$
as ${H^d_{cdh}\left(X, \sK^M_d\right)}_{{\rm deg} 0}$. This makes sense even
if $X$ is singular.

\begin{prop}\label{prop:AlbKH}
For a projective $\C$-scheme $X$ of dimension $d$, the Chern class maps
induce the following commutative diagram with exact rows.
\begin{equation}\label{eqn:AlbKH1}
\xymatrix@C.6pc{
0 \ar[r] & {H^d_{cdh}\left(X, \sK^M_d\right)}_{{\rm deg} 0} \ar[r] 
\ar@{->>}[d] &
H^d_{cdh}\left(X, \sK^M_d\right) \ar@{->>}[d] \ar[r] &
 H^{2d}(X, \Z) \ar@{=}[d] \ar[r] & 0 \\
0 \ar[r] & F^dKH^M_0(X)_{{\rm deg} 0} \ar[r] \ar@{->>}[d]^{Alb^H_X} & 
F^dKH^M_0(X) 
\ar[r]^{c^{top}_0} \ar@{->>}[d]^{c_{0,X}} & H^{2d}(X, \Z) \ar@{=}[d] \ar[r] 
& 0 \\
0 \ar[r] & J^d(X) \ar[r] & H^{2d}_{\sD}\left(X, \Z(d)\right) \ar[r] &
H^{2d}(X, \Z) \ar[r] & 0}
\end{equation}
\end{prop}

\begin{remk}\label{remk:AbKH}
The reader should be warned that this proposition can not be deduced
from the diagram ~\eqref{eqn:AlbKA1} because there is {\'a} priori no
map from $F^dK_0(X)$ (or from $CH^d(X)$) to $F^d_BKH_0(X)$ which factors the 
classical albanese map $F^dK_0(X) \to  H^{2d}_{\sD}\left(X, \Z(d)\right)$. 
This is because of the lack of the current knowledge on the question whether
$F^dK_0(X)$ is isomorphic to $F^d_BK_0(X)$. 
\end{remk}

\begin{proof} 
The bottom row is exact by ~\eqref{eqn:InterJ}. To show that $c^{top}_0$ is
surjective, it suffices to show that the composite map
$H^d_{cdh}\left(X, {\sK}^M_{d}\right) \to F^dKH_0(X) 
\to H^{2d}(X, \Z)$ is surjective.
But this follows from the commutative diagram
\begin{equation}\label{eqn:KH*10}
\xymatrix@C.6pc{
H^d_{cdh}\left(X, {\sK}^M_{d}\right) \ar[d] \ar[r] &
H^d_{cdh}\left(\wt{X}, {\sK}^M_{d}\right) \ar[d] \\
H^{2d}(X, \Z) \ar[r]^{\cong} & H^{2d}(\wt{X}, \Z),}
\end{equation}
where the top horizontal arrow is surjective by Mayer-Vietoris, the right
vertical arrow is surjective by the smoothness of $\wt{X}$ plus
Lemma~\ref{lem:MQC-top}, and the
bottom horizontal arrow is an isomorphism, as can be checked by the
Mayer-Vietoris sequence for the singular cohomology. 
${H^d_{cdh}\left(X, \sK^M_d\right)}_{{\rm deg} 0}$ and 
${F^dKH^M_0(X)}_{{\rm deg} 0}$ are defined to make the top and the middle
rows exact. This makes the above diagram commutative.
We only need to show that the middle vertical arrow is surjective to
complete the proof.

We first observe that the functorial property of the 
descent spectral sequence ~\eqref{eqn:SSC}
and the Mayer-Vietoris property of the $cdh$
cohomology gives the map of spectral sequences 
$E^{p,q}_2 (E) = H^p_{cdh}\left(E, \sK_q\right)
\to H^{p+1}_{cdh}\left(X, \sK_q\right) = E^{p+1,q}_2$. In particular,
this induces the natural map $F^{d-1}KH_1(E) = E^{d-1, d}_{\infty}(E)
\to E^{d,d}_{\infty}(X) = F^dKH_0(X)$. This restricts to the map
$F^{d-1}KH^M_1(E) \to F^dKH^M_0(X)$.
We get a commutative diagram    
\begin{equation}\label{eqn:KH*11}
\xymatrix@C.6pc{
H^{d-1}_{cdh}\left(E, {\sK}^M_{d}\right) \ar[r] \ar@{->>}[d] & 
H^d_{cdh}\left(X, {\sK}^M_{d}\right) \ar@{->>}[d] \ar[r] &
H^d_{cdh}\left(\wt{X}, {\sK}^M_{d}\right) \ar[d]^{\cong} \ar[r] & 0 \\
F^{d-1}KH^M_1(E) \ar[r] \ar[d]^{c_{1,E}} & 
F^dKH^M_0(X) \ar[r] \ar[d]^{c_{0,X}} & 
F^dKH^M_0(\wt{X}) \ar[r] \ar[d]^{c_{0, \wt{X}}} & 0 \\
\H^{2d-1}_{\sD}\left(E, \Z(d)\right) \ar[r] &
\H^{2d}_{\sD}\left(X, \Z(d)\right) \ar[r] & 
\H^{2d}_{\sD}\left(\wt{X}, \Z(q)\right) \ar[r] & 0.}
\end{equation}
The top and the bottom rows are exact. The left and the middle vertical maps
on the top are surjective by ~\eqref{eqn:MQC*}.
The top right vertical map is isomorphism by the
Bloch's formula for smooth varieties and Lemma~\ref{lem:MQC-top}. This shows 
in particular that the middle
row is also exact. The right vertical map on the bottom is known to be
surjective as $\wt{X}$ is smooth. We now only have to show that the composite 
vertical map on the left is surjective to finish the proof. But this is
shown in Lemma~\ref{lem:SNCD1}.
\end{proof}

\begin{cor}\label{cor:AlbKH*}
Let $E$ be as in Lemma~\ref{lem:SNCD}. Then the Chern class 
$c_1 = c_{d+1, 2d+1} : KH_1(E) \to \H^{2d+1}_{\sD}\left(E, \Z(d+1)\right)$
induces the exact sequences
\begin{equation}\label{eqn:ONE}
{F^dKH^M_0(X)_{{\rm deg} 0}} \otimes \C^* \to F^dKH^M_1(E) \xrightarrow{c_1}  
\H^{2d+1}_{\sD}\left(E, \Z(d+1)\right) \to 0.
\end{equation}
\begin{equation}\label{eqn:TWO}
{H^d_{cdh}\left(E, \sK^M_d\right)}_{{\rm deg} 0} \otimes \C^* \to
H^d_{cdh}\left(E, \sK^M_{d+1}\right) \xrightarrow{c_1}
\H^{2d+1}_{\sD}\left(E, \Z(d+1)\right)
\to 0.
\end{equation}
\end{cor}
\begin{proof}
Tensoring the diagram ~\eqref{eqn:AlbKH1} of Proposition~\ref{prop:AlbKH}
(for $X = E$) with $\C^*$ and using the isomorphism in ~\eqref{eqn:strictN},
we get a commutative diagram 
\begin{equation}\label{eqn:AlbKH*1}
\xymatrix@C.6pc{
0 \ar[r] & {F^dKH^M_0(X)_{{\rm deg} 0}} \otimes \C^* \ar[r] \ar[d] &
F^dKH^M_0(X) \otimes \C^* \ar[r] \ar[d] & 
\H^{2d+1}_{\sD}\left(E, \Z(d+1)\right) \ar[r] \ar@{=}[d] & 0 \\
0 \ar[r] & {\rm Ker}(c_1) \ar[r] & F^dKH_1(E) \ar[r]^{c_1} &
\H^{2d+1}_{\sD}\left(E, \Z(d+1)\right) \ar[r] & 0,}
\end{equation}
where the first map on the top row is injective because $H^{2d}(E, \Z)$
is a free abelian group of finite rank. The corollary now follows from 
Corollary~\ref{cor:KtheoryCup}.
The second exact sequence follows exactly in the same way using
Corollary~\ref{cor:KtheoryCup} again.
\end{proof} 

\begin{defn}\label{defn:FDC}
Let $X$ be a projective $\C$-scheme of dimension $d$ over $\C$. We say that
$F^dKH^M_0(X)$ is {\sl finite-dimensional}, if the map 
$F^dKH^M_0(X) \xrightarrow{c_{0,X}} H^{2d}_{\sD}\left(X, \Z(d)\right)$ is
an isomorphism. 
\end{defn}

One could now formulate the following $cdh$ version of
the Roitman torsion theorem and the finite-dimensionality problem for the 
Chow group of zero-cycles on singular projective schemes.

\begin{conj}\label{conj:FDCCR}
For a projective $\C$-scheme $X$ of dimension $d$, the map
\[
F^dKH^M_0(X) \xrightarrow{c_{0,X}} H^{2d}_{\sD}\left(X, \Z(d)\right)
\]
is isomorphism on torsion subgroups.
\end{conj} 

\begin{conj}\label{conj:FDCC}
Let $X$ be projective $\C$-scheme of dimension $d$.
Then $F^dKH^M_0(X)$ is finite-dimensional if and only if 
$H^d_{cdh}\left(X, \Omega^i_{X/{\C}}\right) = 0$ for $0 \le i \le d-2$.
\end{conj}

We shall discuss the cases of these two conjectures for surfaces in
Section~\ref{section:Surf}.

\section{$cdh$ cohomology of lower $\sK$-sheaves on curves}
\label{section:K-curves}

In Section~\ref{section:KDC}, we proved some results
regarding the Chern class maps on the $KH$-theory of normal crossing schemes.
For schemes which have worse singularities than normal crossings, we 
prove the following results for curves in this section. 

\begin{lem}\label{lem:curvePic}
Let $E$ be a (possibly nonreduced) curve over $\C$. Then 
the map $H^1_{Zar}\left(E, \sO^*_E\right) \to 
H^1_{cdh}\left(E, \sO^*_E\right)$ is surjective. This map is an
isomorphism if $E$ is seminormal. 
\end{lem}
\begin{proof}
Since $H^1_{Zar}\left(E, \sO^*_E\right) \surj
H^1_{Zar}\left(E_{\rm red}, \sO^*_E\right)$
and is an isomorphism in the $cdh$ topology, we can assume that $E$ is reduced.
This result is well known when $E$ is smooth. In general,
let $E^N \xrightarrow{f} E$ be the normalization of $E$ and let
$S \inj E$ be a conducting subscheme for the normalization.
Put $\wt{S} = S {\times}_E E^N$. 
The smoothness of $E^N$ and the Leray spectral sequence
imply that $H^*_{cdh}\left(E^N, \sO^*\right) \cong
H^*_{cdh}\left(E, \pi_*\sO^*\right)$. Now, the exact sequence
\[
0 \to \sO^*_E \to \sO^*_{E^N} \to \frac{\sO^*_{\wt{S}}}{\sO^*_S} \to 0
\]
gives the commutative diagram of exact sequences
\[
\xymatrix@C.5pc{
\sO^*(E^N) \oplus \sO^*(S) \ar[r] \ar[d] & \sO^*(\wt{S}) \ar[r] \ar[d] &
H^1_{Zar}\left(E, \sO^*\right) \ar[r] \ar[d] &
H^1_{Zar}\left(E^N, \sO^*\right) \ar[r] \ar[d]^{\cong} & 0 \\
\sO^*_{cdh}(E^N) \oplus \sO^*_{cdh}(S_{\rm red}) \ar[r] & 
\sO^*_{cdh}({\wt{S}}_{\rm red}) 
\ar[r] & H^1_{cdh}\left(E, \sO^*\right) \ar[r] &
H^1_{cdh}\left(E^N, \sO^*\right) \ar[r] & 0.} 
\]
Now we observe that $S_{\rm red}$ and ${\wt{S}}_{\rm red}$ are 
zero-dimensional and hence smooth. In particular,
$\sO^*(S) \surj  \sO^*_{cdh}(S)$ and is an isomorphism when $S$ is
reduced, which is the case when $E$ is seminormal. The same holds for
$\wt{S}$. We always have $\sO^*(E^N) \cong \sO^*_{cdh}(E^N)$.
A diagram chase now proves the result.
\end{proof} 

\begin{lem}\label{lem:curveSK1}
For any curve $E$ over $\C$, there is an exact sequence
\[
{\rm Pic}^0(E) \otimes \C^* \to
H^1_{cdh}\left(E, \sK_{2}\right) \xrightarrow{c_1}
\H^{3}_{\sD}\left(E, \Z(2)\right)
\to 0.
\]
\end{lem}
\begin{proof}
By Lemma~\ref{lem:curvePic}, it suffices to prove ~\eqref{eqn:TWO} for $E$.
We can assume that $E$ is reduced.
As in the proof of Lemma~\ref{lem:curvePic}, we consider the commutative
diagram of exact sequences
\[
\xymatrix@C.5pc{
\sO^*_{cdh}({\wt{S}}_{\rm red}) \otimes \C^* 
\ar[r] \ar[d] & {H^1_{cdh}\left(E, \sO^*\right)}_{{\rm deg} 0} 
\otimes \C^* \ar[r] \ar[d] &
{H^1_{cdh}\left(E^N, \sO^*\right)}_{{\rm deg} 0} \otimes \C^* 
\ar[r] \ar[d] & 0 \\
\sK_{2, cdh}({\wt{S}}_{\rm red})  
\ar[r] & H^1_{cdh}\left(E, \sK_2\right) \ar[r] &
H^1_{cdh}\left(E^N, \sK_2\right) \ar[r] & 0.}
\]
Since $\wt{S}$ is zero-dimensional, the left vertical map is surjective 
using the isomorphism $K^M_2(\C) \cong K_2(\C)$. The lemma now follows
easily by a diagram chase, Corollary~\ref{cor:AlbKH*} for $E^N$ and
the isomorphism $\H^{3}_{\sD}\left(E, \Z(2)\right) \xrightarrow{\cong}
\H^{3}_{\sD}\left(E^N, \Z(2)\right)$.
\end{proof}

\begin{prop}\label{prop:SK1-curve*}
Let $E$ be a curve over $\C$. Then the map $H^1_{Zar}\left(E, \sK_2\right) \to
H^1_{cdh}\left(E, \sK_2\right)$ is surjective.
This map is an isomorphism if $E$ is seminormal.
\end{prop}
\begin{proof} We can assume $E$ to be reduced. For smooth $E$, there is 
nothing to prove. We now assume that $E$ is seminormal.
Let $E^N \xrightarrow{f} E$ be the normalization of $E$ as
above. 
%We claim that there is a conducting subscheme $S \inj E$ such that

%\begin{equation}\label{eqn:SSK1}
%\frac{f_*\left(\sK_{2, E^N}\right)}{\sK_{2, E}} \xrightarrow{\cong}
%\frac{f_*\left(\sK_{2, \wt{S}}\right)}{\sK_{2, S}}.
%\end{equation}

We consider the commutative diagram of exact sequences
\[
\xymatrix@C.5pc{
\sK_{{2, (E, S)}} \ar[r] \ar[d] & \sK_{2, E} \ar[r] \ar[d] &
\sK_{2, S} \ar[r] \ar[d] & 0 \\
f_*\left(\sK_{{2, (E^N, \wt{S})}}\right) \ar[r] & 
f_*\left(\sK_{2, E^N}\right) \ar[r] & 
f_*\left(\sK_{2, \wt{S}}\right) \ar[r] & 0.}
\]
The double relative $K$-theory exact sequence tells us that the cokernel
of the left vertical map is contained in  $\sK_{1, (E, E^N, S)}$, which
in turn is isomorphic to $\sI_{S} \otimes \Omega^1_{E^N/E}$ by the main 
result of Geller-Weibel \cite{GellerW}.   
But this last term is zero because $E$ is seminormal.
In particular, the left vertical map in the above diagram is surjective. 
Hence we get exact sequence
\begin{equation}\label{SK1-curve*1} 
\sK_{2, E} \to f_*\left(\sK_{2, E^N}\right) \oplus 
\sK_{2, S} \to f_*\left(\sK_{2, \wt{S}}\right) \to 0.
\end{equation}
Taking the associated long exact sequences of the Zariski and $cdh$
cohomology, we get the following commutative diagram.
\[
\xymatrix@C.5pc{
{\left\{ \begin{array}{l} 
H^0_{Zar}\left(E^N, \sK_2\right) \\ \ \ \ \ \ \ \ \oplus \\
H^0_{Zar}\left(S, \sK_2\right)\end{array}\right\}} \ar[r] \ar[d] &
H^0_{Zar}\left(\wt{S}, \sK_2\right) \ar[r] \ar[d] &
H^1_{Zar}\left(E, \sK_2\right) \ar[r] \ar[d] &
H^1_{Zar}\left(E^N, \sK_2\right) \ar[r] \ar[d] & 0 \\
{\left\{ \begin{array}{l} 
H^0_{cdh}\left(E^N, \sK_2\right) \\ \ \ \ \ \ \ \ \ \oplus \\
H^0_{cdh}\left(S, \sK_2\right)\end{array}\right\}} \ar[r] &
H^0_{cdh}\left(\wt{S}, \sK_2\right) \ar[r] &
H^1_{cdh}\left(E, \sK_2\right) \ar[r] &
H^1_{cdh}\left(E^N, \sK_2\right) \ar[r] & 0.} 
\]
The smoothness of $E, S$ and $\wt{S}$ implies that the first two
vertical maps from the left and the last vertical map on the right are
isomorphisms. Hence the remaining vertical map is also an isomorphism.
This proves the case of seminormal curves.
For a general reduced curve, we compare the commutative diagram
\[
\xymatrix@C.5pc{
{\rm Pic}^0(E) \otimes \C^* \ar[r] \ar@{=}[d] & 
H^1_{Zar}\left(E, \sK_{2}\right) \ar[r]^{c_1} \ar[d] &
\H^{3}_{\sD}\left(E, \Z(2)\right) \ar[r] \ar@{=}[d] & 0 \\
{\rm Pic}^0(E) \otimes \C^* \ar[r] & 
H^1_{cdh}\left(E, \sK_{2}\right) \ar[r]^{c_1} &
\H^{3}_{\sD}\left(E, \Z(2)\right) \ar[r] & 0,}
\]
where the top row is exact by \cite[Lemma~3.1]{Krishna1}
and the bottom row is exact by Lemma~\ref{lem:curveSK1}.
This proves the result.
\end{proof}

\begin{cor}\label{cor:SK1-curve*1}
Let $E$ be a curve over $\C$. Then \\
$(i)$ The map ${H^1_{cdh}\left(E, \sK_{2}\right)}_{tors}
\to {\H^{3}_{\sD}\left(E, \Z(2)\right)}_{tors}$ is split surjective  
and is an isomorphism if $E$ is seminormal. \\
$(ii)$ $H^1_{cdh}\left(E, \sK_{2}\right) \otimes {\Q}/{\Z} = 0$.
\end{cor}
\begin{proof}
The first part follows directly from Proposition~\ref{prop:SK1-curve*} and
\cite[Theorem~5.3]{BPW}. The second part follows from 
Proposition~\ref{prop:SK1-curve*} and \cite[Proposition~8.4]{BPW}.
\end{proof}

\section{Chern classes for smooth schemes}\label{section:CHERNS}
Let $X$ be a smooth projective $\C$-scheme of dimension $d$. In this section,
we prove some results about the Chern classes from the $K$-theory of $X$ 
into its Deligne cohomology. We shall assume all abelian groups in this
section to be tensored with $\Q$, i. e., the abelian group $A$ will actually
mean $A \otimes_{\Z} \Q$. Let
$\sH^i_{\sD}(q)$ denote the sheaf on $X_{Zar}$ associated to the presheaf
$U \mapsto \H^i_{\sD}(U_{an}, \Q(q))$, where 
\begin{equation}\label{AffineD}
\Q_{\sD, U}(q) : = \left(\Q(q) \to \sO_U \to \cdots \to \Omega^{q-1}_U\right).
\end{equation} 
\begin{lem}\label{lem:HDvanish}
Let $X$ be a smooth projective $\C$-scheme of dimension $d$. Then 
$H^d_{Zar}\left(X, \sH^{d-1}_{\sD}(d)\right) = 0$.
\end{lem}
\begin{proof} 
Let $U$ be an affine neighbourood of a closed point on $X$. Then the
map $\C_U \to \left(\sO_{U_{an}} \to \cdots \to \Omega^d_{U_{an}}\right)$
is a quasi-isomorphism of complexes by the Poincar{\'e} lemma. In other
words, there is exact sequence of complexes on $U_{an}$:
\[
0 \to \Omega^d_{U_{an}}[-d-1] \to {\C}/{\Q}(d)[-1] \to \Q(d) \to 0,
\]
which gives the isomorphism $H^{d-2}\left(U_{an}, {\C}/{\Q}(d)\right) 
\cong \H^{d-1}_{\sD}\left(U, \Q(d)\right)$. Since the map
$H^*(U_{an}, \Q) \to H^*(U_{an}, \C)$ is injective, we get exact sequence 
of Zariski sheaves 
\[
0 \to \sH^{d-2}_{\Q , X} \to \sH^{d-2}_{\C , X} \to  \sH^{d-1}_{\sD}(d) \to 0.
\]
This in particular gives a surjection $
H^d_{Zar}\left(X, \sH^{d-2}_{\C , X}\right) \surj 
H^d_{Zar}\left(X, \sH^{d-1}_{\sD}(d)\right)$. But the Bloch-Ogus-Gersten
sequence ({\sl cf.} \cite[(0.3)]{BlochO}) implies that 
$H^p_{Zar}\left(X, \sH^{q}_{\C , X}\right) = 0 $ for $p >q$.
\end{proof} 

\begin{cor}\label{cor:SmmothCH}
Let $X$ be as in Lemma~\ref{lem:HDvanish}. Then the Chern class map 
$K_1(X) \xrightarrow{c_1} 
\H^{2d-1}_{\sD}\left(X, \Q(d)\right)$ gives rise to the Chern class map
\[
H^{d-1}_{\Zar}\left(X, \sK^M_d\right) \surj \frac{F^{d-1}K_1(X)}{F^dK_1(X)}
\to \H^{2d-1}_{\sD}\left(X, \Q(d)\right).
\]
\end{cor}
\begin{proof}
We can replace $H^{d-1}_{\Zar}\left(X, \sK^M_d\right)$ with
$H^{d-1}_{\Zar}\left(X, \sK_d\right)$ by Lemma~\ref{lem:MQC-top}.
The first surjection then follows at once from the Brown-Gersten spectral 
sequence ~\eqref{eqn:SSZ}. This spectral sequence also implies that
there is a surjection \\
$H^{d}_{\Zar}\left(X, \sK_{d+1}\right) \surj
F^dK_1(X)$. Thus we only need to show that the composite
$H^{d}_{\Zar}\left(X, \sK_{d+1}\right) \surj F^dK_1(X) \to
\H^{2d-1}_{\sD}\left(X, \Q(d)\right)$ is zero.
Now, the functoriality of the Chern classes gives a commutative diagram
\[
\xymatrix@C.6pc{
H^{d}_{\Zar}\left(X, \sK_{d+1}\right) \ar@{->>}[r] \ar[d]_{c^*_{d, d-1}} & 
F^dK_1(X) \ar[d]^{c_1} \\
H^d_{Zar}\left(X, \sH^{d-1}_{\sD}(d)\right) \ar[r] & 
\H^{2d-1}_{\sD}\left(X, \Q(d)\right),}
\]
where the bottom horizontal arrow is the edge map in the spectral
sequence
\[
E^{i,j}_2 = H^i_{Zar}\left(X, \sH^{j}_{\sD}(q)\right) \Rightarrow
\H^{i+j}_{\sD}\left(X, \Q(d)\right).
\]
The corollary now follows immediately from Lemma~\ref{lem:HDvanish}.
\end{proof}
 
\begin{lem}\label{lem:C-class}
Let $X$ be as in Lemma~\ref{lem:HDvanish}. Then
$H^{2d-1}\left(X, \Q\right) \cap F^dH^{2d-1}\left(X, \C\right) = 0$ 
and hence 
\[
H^{2d-1}\left(X, \Q\right) \inj \frac{H^{2d-1}\left(X, \C\right)}
{F^dH^{2d-1}\left(X, \C\right)}.
\]
\end{lem}
\begin{proof}
We observe that $H^{2d-1}\left(X, \Q\right)$ is invariant under the
complex conjugation on $H^{2d-1}\left(X, \C\right)$. On the other hand,
\[
F^dH^{2d-1}\left(X, \C\right) \cap \ov{F^dH^{2d-1}\left(X, \C\right)}
= 0
\]
by the Hodge theory. The lemma now follows.
\end{proof}

Let $X$ be a smooth and projective $\C$-scheme of dimension $d$. 
For any $i \ge 0$, let 
${H^{2i}(X, \Q)}_{alg} = H^{i}_{Zar}(X, \Omega^{i}_{X/{\C}}) 
\cap H^{2i}(X, \Q)$. 
Note that 
\begin{equation}\label{eqn:HDG**1}
H^{2i}(X, \Q) \inj H^{2i}(X, \Q) {\otimes}_{\Q} \C \cong
H^{2i}(X, \C)
\end{equation}
and the Hodge decomposition shows that 
${H^{2i}(X, \Q)}_{alg} {\otimes}_{\Q} \C \cong  
H^{i}_{Zar}(X, \Omega^{i}_{X/{\C}})$ under this isomorphism.

Denote the 
image of the map $\frac{H^{d-1}_{Zar}(X, \Omega^{d-1}_{X/{\C}})}
{{H^{2d-2}(X, {\Q})}_{alg}} \to
\H^{2d-1}_{\sD}\left(X, \Q(d)\right)$ by the subgroup
${\H^{2d-1}_{\sD}\left(X, \Q(d)\right)}_{alg}$. 
%Let 
%\[
%{\H^{2d-1}_{\sD}\left(X, \Q(d)\right)}_{tr} =
%\frac{\H^{2d-1}_{\sD}\left(X, \Q(d)\right)}
%{{\H^{2d-1}_{\sD}\left(X, \Q(d)\right)}_{alg}}.
%\]

\begin{prop}\label{prop:Smooth-main}
Let $X$ be as in Lemma~\ref{lem:HDvanish}.
Then 
\[
{\H^{2d-1}_{\sD}\left(X, \Q(d)\right)}_{alg} \subseteq
{\rm Image}\left(H^{d-1}_{\Zar}\left(X, \sK^M_{d}\right) \xrightarrow{c_1}
\H^{2d-1}_{\sD}\left(X, \Q(d)\right)\right).
\]
\end{prop} 

\begin{proof}
We can replace $H^{d-1}_{\Zar}\left(X, \sK^M_{d}\right)$
with $H^{d-1}_{\Zar}\left(X, \sK_{d}\right)$ by Lemma~\ref{lem:MQC-top}. 
The injectivity of maps $H^*(X, \Q) \to H^*(X, \C)$
and ~\ref{eqn:HDG**1} together imply that

\begin{equation}\label{eqn:S-main2}
{H^{2d-2}(X, \Q)}_{alg} \otimes_{\Q} \C^*_{\Q} 
%\cong H^{2d-2}\left(X, \C^*_{\Q}\right)
\cong \frac{{H^{2d-2}(X, {\Q})}_{alg} \otimes_{\Q} \C}
{{H^{2d-2}(X, {\Q})}_{alg}} \cong
\frac{H^{d-1}_{Zar}(X, \Omega^{d-1}_{X/{\C}})}{{H^{2d-2}(X, {\Q})}_{alg}}.
\end{equation}

%\to \H^{2d-1}_{\sD}\left(X, \Q(d)\right).
%\end{equation}
%Note that $H^{2d-1}(X, \Q) \otimes_{\Q} \C^*_{\Q} 
%\cong H^{2d-1}\left(X, \C^*_{\Q}\right)$.

The validity of Hodge conjecture for codimension $(d-1)$ cycles on $X$
({\sl cf.} \cite[p. 91]{Lewis}) implies that 
$H^{d-1}_{\Zar}\left(X, \sK_{d-1}\right) \surj {H^{2d-2}(X, \Q)}_{alg}$
under the topological Chern class maps
\begin{equation}\label{eqn:S-main3}
\xymatrix@C.8pc{
H^{d-1}_{\Zar}\left(X, \sK_{d-1}\right) \ar[r]^{\cong} &
CH^{d-1}(X) \ar[r]  \ar[dr] \ar[d] 
& \H^{2d-2}_{\sD}\left(X, \Q(d-1)\right) \ar[d] \\
& H^{2d-2}(X, \C) & H^{2d-2}(X, \Q) \ar@{^{(}->}[l] \cdot }
\end{equation}
The proposition now follows by using the commutative
diagram
\begin{equation}\label{eqn:S-main4}
\xymatrix@C.6pc{
& 
H^{d-1}_{\Zar}\left(X, \sK_{d-1}\right) \otimes \C^* \ar@{->>}[d] \ar[r] &
H^{d-1}_{\Zar}\left(X, \sK_{d}\right)  \ar[d]^{c_1} \\
\frac{H^{d-1}_{Zar}(X, \Omega^{d-1}_{X/{\C}})}{{H^{2d-2}(X, {\Q})}_{alg}}
\ar[r]^{\cong \ \ \ \ \ \ \ } &
{H^{2d-2}(X, \Q)}_{alg} \otimes_{\Q} \C^* \ar[r] &
\H^{2d-1}_{\sD}\left(X, \Q(d)\right),}
\end{equation}
where the first map on the bottom row is the isomorphism of 
~\eqref{eqn:S-main2}. 
\end{proof}  

\section{Zero-cycles on singular varieties and their resolutions}
In this section, we prove our main result about the comparison
between the $cdh$ analogue of the Chow group of zero-cycles on a
normal projective variety $X$ and that of a resolution of singularities
of $X$. We need the following intermediate result.

\begin{lem}\label{lem:Extra-Chern}
Consider the resolution diagram ~\eqref{eqn:RSD} for a projective
(not necessarily normal) $\C$-scheme $X$ of dimension $d$.
Then the map 
\[
H^{d-1}_{\Zar}\left(\wt{X}, \sK_{d}\right) \xrightarrow{c_1}
\frac{\H^{2d-1}_{\sD}\left(\wt{X}, \Q(d)\right)}
{\H^{2d-1}_{\sD}\left(X, \Q(d)\right)}
\]
is surjective.
\end{lem}

\begin{proof}
It follows from Lemmas~\ref{lem:MDeligneC3} and ~\ref{lem:C-class} that
\begin{equation}\label{eqn:EXCh-1}
\frac{\H^{2d-2}_{Zar}\left(\wt{X}, \Omega^{<d}_{{\wt{X}}/{\C}}\right)}
{H^{2d-2}(\wt{X}, \Q)} \xrightarrow{\cong}
H^{2d-1}_{\sD}\left(\wt{X}, \Q(d)\right).
\end{equation}

Using the exact sequence
\[
0 \to \frac{H^{d-1}_{Zar}(X, \Omega^{d-1}_{X/{\C}})}{{H^{2d-2}(X, {\Q})}_{alg}}
\to \frac{\H^{2d-2}_{Zar}\left(\wt{X}, \Omega^{<d}_{{\wt{X}}/{\C}}\right)}
{H^{2d-2}(\wt{X}, \Q)} \to
\frac{\H^{2d-2}_{Zar}\left(\wt{X}, \Omega^{<d-1}_{{\wt{X}}/{\C}}\right)}
{H^{2d-2}(\wt{X}, \Q)} \to 0
\]
and Proposition~\ref{prop:Smooth-main}, it suffices to show that
\begin{equation}\label{eqn:EXCh-2}
\H^{2d-2}_{cdh}\left(X, \Omega^{<d-1}_{X/{\C}}\right)
\surj \H^{2d-2}_{cdh}\left(\wt{X}, \Omega^{<d-1}_{{\wt{X}}/{\C}}\right).
\end{equation}

But this follows from the Mayer-Vietoris exact sequence
\[
\H^{2d-2}_{cdh}\left(X, \Omega^{<d-1}_{X/{\C}}\right) \to
\H^{2d-2}_{cdh}\left(\wt{X}, \Omega^{<d-1}_{{\wt{X}}/{\C}}\right)
\oplus \H^{2d-2}_{cdh}\left(S, \Omega^{<d-1}_{S/{\C}}\right) \to
\H^{2d-2}_{cdh}\left(E, \Omega^{<d-1}_{E/{\C}}\right)
\]
together with the fact that 
$\H^{2d-2}_{cdh}\left(S, \Omega^{<d-1}_{S/{\C}}\right) = 0
= \H^{2d-2}_{cdh}\left(E, \Omega^{<d-1}_{E/{\C}}\right)$,
which follows from Lemma~\ref{lem:folklore2}.
\end{proof} 

The following is our main result about the $cdh$ version of 
the Chow group of zero-cycles on singular schemes.
\begin{thm}\label{thm:Main22}
Let $X$ be a normal and projective $\C$-scheme of dimension 
$d \ge 2$. Suppose that for a resolution diagram ~\eqref{eqn:RSD} for $X$,
\[
{H^{d-1}_{cdh}\left(\wt{X}, \sK^M_{d-1}\right)}_{hom} \otimes \C^*  
\surj {H^{d-1}_{cdh}\left(E, \sK^M_{d-1}\right)}_{hom} \otimes \C^*.
\]
Then the map 
\[
F^{d}KH^M_0(X) \to F^{d}KH^M_0(\wt{X}) \cong 
CH^{d}(\wt{X})
\]
is an isomorphism.
\end{thm}
\begin{proof}
Since $H^{d}_{cdh}\left(\wt{X}, \sK^M_{d}\right) \cong
F^{d}KH^M_0(\wt{X}) \cong F^{d}K_0(\wt{X}) \cong CH^{d}(\wt{X})$ and \\ 
$H^{d}_{cdh}\left(X, \sK^M_{d}\right) \surj F^{d}KH^M_0(X)$,
it suffices to show the stronger assertion that
\begin{equation}\label{eqn:thm:Main22*H} 
H^{d}_{cdh}\left(X, \sK^M_{d}\right) \xrightarrow{\cong}
H^{d}_{cdh}\left(\wt{X}, \sK^M_{d}\right).
\end{equation}
Let us write $\frac{\H^{2d-1}_{\sD}\left(\wt{X}, \Q(d)\right)}
{\H^{2d-1}_{\sD}\left(X, \Q(d)\right)}$ as
$\ov{\H^{2d-1}_{\sD}\left(\wt{X}, \Q(d)\right)}$
and consider the following commutative diagram of Mayer-Vietoris
exact sequences.
\[
\xymatrix@C.3pc{
{H^{d-1}_{cdh}\left(\wt{X}, \sK^M_{d-1}\right)}_{hom} \otimes \C^* \ar[r] 
\ar[d] &
{H^{d-1}_{cdh}\left(E, \sK^M_{d-1}\right)}_{hom} \otimes \C^* \ar[d]^{\alpha} 
& & \\
H^{d-1}_{cdh}\left(\wt{X}, \sK^M_{d}\right) \ar[r]^{\beta} \ar[d] &
H^{d-1}_{cdh}\left(E, \sK^M_{d}\right) \ar[r] \ar[d] &
H^{d}_{cdh}\left(X, \sK^M_{d}\right) \ar@{->>}[r] \ar[d] &
H^{d}_{cdh}\left(\wt{X}, \sK^M_{d}\right) \ar[d] \\
{\ov{\H^{2d-1}_{\sD}\left(\wt{X}, \Q(d)\right)}} \ar[r]^{\gamma} &
\H^{2d-1}_{\sD}\left(E, \Q(d)\right) \ar[r] &
\H^{2d}_{\sD}\left(X, \Q(d)\right) \ar@{->>}[r] &
\H^{2d}_{\sD}\left(\wt{X}, \Q(d)\right).}
\]
The bottom row is exact by Lemma~\ref{lem:MVD} since the normality of $X$ 
implies that ${\rm dim}(X_{\rm sing}) \le d-2$ and hence
$\H^{i}_{\sD}\left(X_{\rm sing}, \Q(d)\right) = 0$ for $i \ge 2d-1$.  
The middle row is exact also because of the normality of $X$.
Furthermore, the universal property of the albanese variety of $X$ 
({\sl cf.} \cite{ESV}) shows that $J^d(X) \xrightarrow{\cong} J^d(\wt{X})$ and 
hence it follows from ~\eqref{eqn:InterJ} that 
$\H^{2d}_{\sD}\left(X, \Q(d)\right) \xrightarrow{\cong}
\H^{2d}_{\sD}\left(\wt{X}, \Q(d)\right)$. In particular, the
map $\gamma$ is surjective.

The left lower vertical map is surjective by Lemma~\ref{lem:Extra-Chern}.
The second lower vertical map from the left is surjective by 
Lemma~\ref{lem:SNCD1}.
%The last two vertical maps are surjective by Proposition~\ref{prop:AlbKH}.
%The first and the third assumptions imply that the vertical map on the right 
%end is an isomorphism. 
The second column from the left is exact by
Corollary~\ref{cor:AlbKH*}. A diagram chase shows that we only need to show
that ${\rm Image}(\alpha) \subseteq {\rm Image}(\beta)$. But this follows from
the assumption of the theorem.
\end{proof}

\section{Chow groups of zero-cycles on surfaces}\label{section:Surf}
In this section, we deduce some consequences of Theorem~\ref{thm:Main22}
for the Chow group of zero-cycles on surfaces. In particular, we prove
the $cdh$ version of the Roitman torsion theorem and compare the Chow 
group of a surface with arbitrary singularity with the Chow group
of a resolution of singularities.

\subsection{Roitman torsion for surfaces}\label{subsection:RTTS}

The following is a version of the Roitman torsion theorem for singular
surfaces in the $cdh$ topology. It proves Conjecture~\ref{conj:FDCCR}
for surfaces.

\begin{thm}\label{thm:RTC}
Let $X$ be a projective surface over $\C$. Then the Chern class map
$H^2_{cdh}\left(X, \sK_2\right) \xrightarrow{c_{0,X}}
\H^{4}_{\sD}\left(X, \Z(2)\right)$ is isomorphism on torsion subgroups.
\end{thm}
\begin{proof}
This follows by imitating the proof of the Zariski version of the
Roitman torsion theorem in \cite{BPW}. 
We consider the resolution diagram ~\eqref{eqn:RSD} for $X$ and
use the following diagram.
\[
\xymatrix@C.3pc{
{\left\{ \begin{array}{l} 
{H^1_{cdh}\left(\wt{X}, \sK_2\right)}_{tors} \\ \ \ \ \ \ \ \ \ \oplus \\
{H^1_{cdh}\left(S, \sK_2\right)}_{tors}\end{array}\right\}} \ar[r] \ar[d] &
{H^1_{cdh}\left(E, \sK_2\right)}_{tors} \ar[r] \ar[d] &
{H^2_{cdh}\left(X, \sK_2\right)}_{tors} \ar[r] \ar[d]^{c_{0,X}} &
{H^2_{cdh}\left(\wt{X}, \sK_2\right)}_{tors} \ar[r] \ar[d]^{c_{0,\wt{X}}} & 0 
\\
{\left\{ \begin{array}{l} 
{\H^{3}_{\sD}\left(\wt{X}\right)}_{tors} \\ \ \ \ \ \ \ \ \ \oplus \\
{\H^{3}_{\sD}\left(S\right)}_{tors}\end{array}\right\}}
\ar[r] & {\H^{3}_{\sD}\left(E, \Z(2)\right)}_{tors} \ar[r] &
{\H^{4}_{\sD}\left(X, \Z(2)\right)}_{tors} \ar[r] &
{\H^{4}_{\sD}\left(\wt{X}, \Z(2)\right)}_{tors} \ar[r] & 0.}
\]
The exactness of the bottom row is already shown in {\sl loc. cit.}.
The top row is exact without taking the torsion part by the Mayer-Vietoris
property of the $cdh$ cohomology. Corollary~\ref{cor:SK1-curve*1} now shows
that the top row is exact except at ${H^1_{cdh}\left(E, \sK_2\right)}_{tors}$.
The first vertical map from left is surjective by  
Corollary~\ref{cor:SK1-curve*1} and [{\sl loc. cit.}, Proposition~8.4].
The second vertical map from the left is an isomorphism by
Corollary~\ref{cor:SK1-curve*1} since $E$ is seminormal. The last vertical
map on the right is an isomorphism by the Roitman torsion for smooth
surfaces. A diagram chase shows that $c_{0,X}$ is an isomorphism on the
torsion subgroups. 
\end{proof}

\subsection{Chow group of singular surfaces}\label{subsection:FDSU}
The following is our main result 
for the $cdh$ analogue of the Chow group zero-cycles on singular surfaces.

\begin{thm}\label{thm:Surface*}
Let $X$ be a projective surface over $\C$ and let $\wt{X} \to
X$ be a resolution of singularities of $X$ as in ~\eqref{eqn:RSD}.
Then \\
$(1)$ $H^2_{cdh}\left(X, \sO_X\right) \xrightarrow{\cong} 
H^2_{cdh}(\wt{X}, \sO_{\wt{X}}) \Rightarrow
{\rm ker}(c_{0,X}) \xrightarrow{\cong} {\rm ker}(c_{0,\wt{X}})$. \\
$(2)$ $H^2_{cdh}\left(X, \Omega^i_{X/{\C}}\right) \xrightarrow{\cong} 
H^2_{cdh}\left(\wt{X}, \Omega^i_{{\wt{X}}/{\C}}\right) \ {\rm for} \
i \le 1 \Rightarrow F^2KH_0(X) \xrightarrow{\cong} CH^2(\wt{X})$. 
\end{thm}
\begin{proof}
Since $\sK^M_{2, X} \cong \sK_{2, X}$, we can use $F^2KH^M_0(X)$ and
$F^2KH_0(X)$ interchangeably.

Assume first that $H^2_{cdh}\left(X, \sO_X\right) \xrightarrow{\cong} 
H^2_{cdh}(\wt{X}, \sO_{\wt{X}})$. 
The exact sequence
\[
H^1_{cdh}(\wt{X}, \sO) \oplus H^1_{cdh}\left(S, \sO\right)
\to H^1_{cdh}\left(E, \sO\right) \to H^2_{cdh}\left(X, \sO\right) 
\to H^2_{cdh}(\wt{X}, \sO) \to 0,
\]
Corollary~\ref{cor:SNC-main1} for $E$, \cite[Corollary~2.5, 
Proposition~2.6]{CJams} for $\wt{X}$ and $S$, and our assumption 
together imply that 

\begin{equation}\label{eqn:Surface*2}
H^1_{Zar}(\wt{X}, \sO) \oplus H^1_{Zar}\left(S, \sO\right)
\surj H^1_{Zar}\left(E, \sO\right).
\end{equation}
Since these groups are the Lie algebras of the associated Picard varieties,
we conclude that 
\begin{equation}\label{eqn:Surface*21}
{\rm Pic}^0(\wt{X}) \oplus {\rm Pic}^0(S) \surj {\rm Pic}^0(E).
\end{equation}
We now consider the following commutative diagram of 
exact sequences
\begin{equation}\label{eqn:Surface*213}
\xymatrix@C.5pc{
{{H^1_{cdh}\left(\wt{X}, \sK_1\right)}_{hom}} \otimes \C^*
\ar[r] \ar[d] &
{\frac{{H^1_{cdh}\left(E, \sK_1\right)}_{hom}}
{{H^1_{cdh}\left(S, \sK_1\right)}_{hom}}} \otimes \C^* \ar[d]^{\alpha} &  &  \\
{H^1_{cdh}\left(\wt{X}, \sK_2\right)} \ar[r]^{\beta} \ar[d] &
{\frac{H^1_{cdh}\left(E, \sK_2\right)}{H^1_{cdh}\left(S, \sK_2\right)}}
\ar[r] \ar@{->>}[d] &
{H^2_{cdh}\left(X, \sK_2\right)} \ar@{->>}[r] \ar@{->>}[d]^{c_{0,X}} &
{H^2_{cdh}\left(\wt{X}, \sK_2\right)} \ar@{->>}[d]^{c_{0, \wt{X}}} \\
{\ov{\H^{3}_{\sD}\left(\wt{X}, \Z(2)\right)}} \ar[r]^{\gamma} &
{\frac{\H^{3}_{\sD}\left(E, \Z(2)\right)}{\H^{3}_{\sD}\left(S, \Z(2)\right)}}
\ar[r] & 
\H^{4}_{\sD}\left(X, \Z(2)\right) \ar@{->>}[r] &
\H^{4}_{\sD}\left(\wt{X}, \Z(2)\right).}
\end{equation}
In this diagram, the second column from the left is exact as it is the quotient
of the exact sequences given by Lemma~\ref{lem:curveSK1}. This already
implies that ${\rm ker}(c_{0,X}) \surj {\rm ker}(c_{0,\wt{X}})$.
To show the injectivity, it suffices now to show the same with rational
coefficients because of Theorem~\ref{thm:RTC}. In this case,
the left lower vertical map is surjective by Lemma~\ref{lem:Extra-Chern}.
The injectivity now follows from a simple diagram chase and
~\eqref{eqn:Surface*21}. This proves the first assertion.

To prove the second part of the theorem, it suffices to show that
$c_{0,X}$ is an isomorphism.
Under the given assumption, it follows from Lemma~\ref{lem:folklore2} that 
$\H^3_{cdh}\left(X, \Omega^{<2}_{X/{\C}}\right)$ \\ 
$\xrightarrow{\cong} 
\H^3_{cdh}\left(\wt{X}, \Omega^{<2}_{{\wt{X}}/{\C}}\right)$.
The Mayer-Vietoris exact sequence for the singular cohomology
implies that $H^3(X, \Z) \surj H^3(\wt{X}, \Z)$ and
$H^4(X, \Z) \surj H^4(\wt{X}, \Z)$. It follows now from ~\eqref{eqn:CIJ}
that $\H^{4}_{\sD}\left(X, \Z(2)\right) \xrightarrow{\cong}
\H^{4}_{\sD}\left(\wt{X}, \Z(2)\right)$. In particular,
the map $\gamma$ is surjective. This also shows using Theorem~\ref{thm:RTC}
that we only have to show the isomorphism with rational coefficients.
Now, using a simple diagram chase and arguing as in the proof of 
Theorem~\ref{thm:Main22}, we only have to show that the top horizontal
map is surjective, which follows directly from ~\eqref{eqn:Surface*21}. 
\end{proof}

The following recovers \cite[Theorem~1.3]{KrishnaS} 
and \cite[Theorem~1.3]{Krishna1} by a different
and more conceptual approach.

\begin{cor}\label{cor:N-surface}
Let $X$ be a projective surface over $\C$ with a 
resolution of singularities $\wt{X}$ such that
$H^2_{Zar}\left(X, \sO_X\right) = 0$. Then the finite-dimensionality
of $CH^2(\wt{X})$ implies the same for $CH^2(X)$.
\end{cor}
\begin{proof}
By the Roitman torsion theorem ({\sl cf.} \cite{BPW}, \cite{BiswasS}),
it suffices to prove the result rationally.
So we assume all the groups to be tensored with $\Q$.
The surjectivity $H^2_{Zar}\left(X, \sO_X\right) \surj
H^2_{Zar}(\wt{X}, \sO_{\wt{X}})$ and our assumption imply from
Theorem~\ref{thm:Surface*} that $F^2KH_0(X) \xrightarrow{\cong}
\H^{4}_{\sD}\left(X, \Z(2)\right)$. Thus we need to show that
\begin{equation}\label{eqn:N-surface12}
{\rm ker}\left(CH^2(X) \to F^2KH_0(X)\right) \xrightarrow{\cong}
{\rm ker}\left(\H^{4}_{\sD^*}\left(X, \Z(2)\right) \to
\H^{4}_{\sD}\left(X, \Z(2)\right)\right)
\end{equation}
\[
\hspace*{7cm}
 \cong 
{\rm ker}\left(H^2_{Zar}(X, \Omega^1_{X/{\C}})
\to H^2_{cdh}(X, \Omega^1_{X/{\C}})\right),
\]
where the second isomorphism follows from our assumption, ~\eqref{eqn:CIJ},
~\eqref{eqn:InterJ} and ~\eqref{eqn:GIJ}.

However, the Bloch's formula $H^2_{Zar}(X, \sK_2) \cong F^2K_0(X) \cong
CH^2(X)$ of \cite{Levine3}, Corollary~\ref{cor:GRR*} and ~\eqref{eqn:filt}
imply that ${\rm ker}\left(CH^2(X) \to F^2KH_0(X)\right) \subseteq
\wt{K}^{(2)}_0(X)$. 

The corollary now follows from the exact sequence
\[
\Omega^1_{\C} \otimes H^2(X, \sO_X) \to H^2(X, \Omega^1_X) \to
H^2(X, \Omega^1_{X/{\C}}) \to 0,
\]
in the Zariski and the $cdh$ topology, the surjection
$H^2_{Zar}(X, \sO_X) \surj H^2_{cdh}(X, \sO_X)$ and 
Lemma~\ref{lem:K-fiber**}. 
\end{proof}

\begin{cor}\label{cor:BAH}
Let $X$ be a strict normal crossing divisor on a smooth projective
threefold such that
$H^2_{cdh}\left(X, \sO_X\right) = 0$. Then the finite-dimensionality
of $CH^2(\wt{X})$ implies that the maps 
\begin{equation}\label{eqn:BAH1}
\xymatrix@C.8pc{
H^2_{Zar}\left(X, \sK_2\right) \ar[r] \ar[d] & CH^2(X) \ar[d] \ar[r] &
{\H^{4}_{\sD^*}\left(X, \Z(2)\right)} \ar[d] \\ 
H^2_{cdh}\left(X, \sK_2\right) \ar[r] & F^dKH_0(X) \ar[r] &
{\H^{4}_{\sD}\left(X, \Z(2)\right)}}
\end{equation}
are all isomorphisms.
\end{cor}  
\begin{proof}
The first horizontal map on the top row is an isomorphism by the main result 
of \cite{Levine3}. Since $X$ is a strict normal crossing divisor, it
follows from our assumption and Corollary~\ref{cor:SNC-main1} that
$H^2_{Zar}\left(X, \sO_X\right) = H^2_{cdh}\left(X, \sO_X\right) = 0$. 
Hence, the second horizontal map on the top row is an
isomorphism by \cite[Theorem~1.3]{Krishna1}.
It follows from Theorem~\ref{thm:Surface*}
that both the horizontal maps on the bottom row are isomorphisms.
Since $X$ is seminormal, the right vertical map is an isomorphism
by \cite[Corollary~6.2]{Krishna1}. Hence all other vertical maps are 
also isomorphisms.
\end{proof}  

\section{Finite-dimensionality for normal threefolds}\label{section:3fold}
In this section, we use Theorem~\ref{thm:Main22} to deduce a
conditional result on the finite-dimensionality of the Chow group of
zero-cycles on normal threefolds. So let $X$ be a normal and projective  
threefold over $\C$ and consider a resolution of singularities diagram
for $X$ as in ~\eqref{eqn:RSD}.
Since $E$ is a surface, there is an isomorphism
$H^2_{Zar}\left(E, \sK_{2, X}\right) \cong CH^2(X)$ by the main result of
\cite{Levine3}. Since such an isomorphism also holds for $\wt{X}$,
there is a natural map $CH^2(\wt{X}) \to CH^2(E)$ which in turn induces the
restriction map 
\begin{equation}\label{eqn:3fold*}
A^2(\wt{X}) = CH^2(\wt{X})_{alg} 
\inj {CH^2(\wt{X})}_{hom} \to A^2(E) = {CH^2(E)}_{\rm deg 0}.
\end{equation}    
\begin{lem}\label{lem:GHC0}
Assume that $GHC(1,3, \wt{X})$ and $GBC(E^N)$
hold. Suppose that $H^3_{cdh}\left(X, \Omega^i_{X/{\C}}\right) = 0$
for $i \le 1$. Then the map $A^2(\wt{X}) \to A^2(E)$ is surjective.
\end{lem}
\begin{proof}
As in the proof of Corollary~\ref{cor:SNC-main1*}, it follows from our
assumption and Corollary~\ref{cor:SNC-main1} that the map 
$H^2_{Zar}(\wt{X}, \sO_{\wt{X}}) \surj 
H^2_{Zar}\left(E, \sO_{E}\right)$. 
On the other hand, the surjectivity
$H^3_{Zar}\left(X, \Omega^1_{X/{\C}}\right) \surj
H^3_{Zar}\left(\wt{X}, \Omega^1_{{\wt{X}}/{\C}}\right)$ 
and the Hodge theory imply that 
$H^2_{Zar}(\wt{X}, \sO_{\wt{X}}) = 0$.
We conclude that $H^2_{Zar}\left(E, \sO_{E}\right) = 0$. Now, the
Bloch's conjecture $GBC(E^N)$ and \cite[Theorem~1.3]{Krishna1} imply
that the albanese map $A^2(E) \to J^2_*(E)$ is an isomorphism. 
Since $E$ is a strict normal crossing divisor, it follows from 
\cite[Corollary~6.5]{Krishna1} that $J^2_*(E)$ is a semi-abelian variety
and hence $J^2_*(E) \cong J^2(E)$. In particular, we get $A^2(E) 
\xrightarrow{\cong} J^2(E)$.

It also follows from ~\eqref{eqn:CIJ} and Corollary~\ref{cor:SNC-main1*}
that the morphism of complex algebraic groups $J^2(\wt{X}) \to J^2(E)$ is 
surjective. The lemma now follows from Corollary~\ref{cor:GHC*1}.
\end{proof}   

\begin{thm}\label{thm:3fold-main}
Let $X$ be a normal and projective threefold over $\C$ such that
$H^3_{cdh}\left(X, \Omega^i_{X/{\C}}\right) = 0$ for $0 \le i \le 1$.
Assume $GBC(2)$ and $GHC(1,3, \wt{X})$ for some resolution
of singularities $\wt{X}$ as in ~\eqref{eqn:RSD}. Then
$F^3KH^M_0(X)_{\Q} \xrightarrow{\cong} CH^3(\wt{X})_{\Q}$.
\end{thm}
\begin{proof}
In this proof, we assume all groups as tensored with $\Q$ without mentioning it
explicitly.
Since $H^3_{cdh}\left(X, \sK^M_3\right) \surj F^3KH^M_0(X)$, we need to show 
that
\begin{equation}\label{eqn:3fold-main1}
H^3_{cdh}\left(X, \sK^M_3\right) \xrightarrow{\cong}
H^3_{cdh}\left(\wt{X}, \sK^M_3\right).
\end{equation}
We only need to show that the hypothesis of Theorem~\ref{thm:Main22} is
satisfied.
Our assumption and the Mayer-Vietoris exact sequence show that
$H^2_{cdh}(\wt{X}, \sO_{\wt{X}}) \surj 
H^2_{cdh}\left(E, \sO_E\right)$. On the other hand,
our assumption $H^3_{cdh}\left(X, \Omega^1_{X/{\C}}\right) = 0$
implies that the same holds for $\wt{X}$. The Hodge theory now
implies that $H^2_{cdh}(\wt{X}, \sO_{\wt{X}}) = 0$.
We conclude that $H^2_{cdh}\left(E, \sO_E\right) = 0$.

Since $E$ is a strict normal crossing divisor, and since $CH^2(E^N)$ is
finite-dimensional by our assumption,
we can now apply
Corollary~\ref{cor:BAH} to reduce to proving that
${H^2_{Zar}\left(\wt{X}, \sK^M_2\right)}_{hom} \otimes \C^*
\surj {H^2_{Zar}\left(E, \sK^M_2\right)}_{hom} \otimes \C^*$.
For this, it suffices to show that $A^2(\wt{X}) \surj A^2(E)$, using
the isomorphism $H^2_{Zar}\left(\wt{X}, \sK^M_2\right) \cong CH^2(\wt{X})$.
But this is shown in Lemma~\ref{lem:GHC0}.
\end{proof}  

\begin{cor}\label{cor:3fold-main2}
Let $X$ be a normal and projective threefold over $\C$ with only isolated
singularities such that
$H^3_{Zar}\left(X, \Omega^i_{X/{\C}}\right) = 0$ for $0 \le i \le 1$.
Assume $GBC(2)$ and $GHC(1,3, \wt{X})$ for some resolution
of singularities $\wt{X}$ as in ~\eqref{eqn:RSD}. Then
$CH^3(X) \xrightarrow{\cong} CH^3(\wt{X})$.
\end{cor}
\begin{proof} As in the proof of Corollary~\ref{cor:N-surface}, it suffices 
to prove the result with the rational coefficients.
It follows from our assumption and \cite[Proposition~2.6]{CAnnals} that
$H^3_{cdh}\left(X, \Omega^i_{X/{\C}}\right) = 0$ for $0 \le i \le 1$.
In particular, Theorem~\ref{thm:3fold-main} applies.
The corollary now follows from  Theorems~\ref{thm:main1} and 
~\ref{thm:3fold-main}. 
\end{proof}

\section{Zero-cycles in arbitrary dimension}
In this last section, we give a weaker form of Theorem~\ref{thm:3fold-main}
in arbitrary dimension using Theorem~\ref{thm:Main22}. This situation 
particularly applies
in case of projective cones over smooth projective varieties.

\begin{thm}\label{thm:GENERAL}
Let $X$ be a normal and projective variety of dimension $d$ over $\C$ such that
$H^j_{cdh}\left(X, \Omega^i_{X/{\C}}\right) = 0$ for $0 \le i \le d-2$ and
$j \ge d-1$. Let $\wt{X} \to X$ be a resolution of singularities such that
the reduced exceptional divisor is smooth. 
Assume that $GBC(d-1)$ and $GHC(d-2,2d-3, \wt{X})$ hold. 
Then $F^dKH^M_0(X)_{\Q} \xrightarrow{\cong} CH^d(\wt{X})_{\Q}$.
\end{thm}

\begin{proof}
Following the proof of Theorem~\ref{thm:3fold-main}, we can use 
Theorem~\ref{thm:Main22} and Lemma~\ref{lem:MQC-top} 
to reduce to showing that the map
\begin{equation}\label{eqn:GEN*}
CH^{d-1}(\wt{X})_{hom} \otimes \C^*  
\to CH^{d-1}(E)_{hom} \otimes \C^*
\end{equation}
is surjective, where $E$ is the reduced exceptional divisor on $\wt{X}$.
The $cdh$ cohomology exact sequence (which uses the normality of $X$)
\[
H^{d-1}_{cdh}\left(\wt{X}, \Omega^i_{{\wt{X}}/{\C}}\right) \to
H^{d-1}_{cdh}\left(E, \Omega^i_{E/{\C}}\right) \to 
H^{d}_{cdh}\left(X, \Omega^i_{X/{\C}}\right)
\]
shows that $H^{d-1}_{cdh}\left(E, \Omega^i_{E/{\C}}\right) = 0$ for
$0 \le i \le d-3$. Together with our assumption, this implies that the
map
\begin{equation}\label{eqn:GEN*1}
CH^{d-1}(E)_{hom} \xrightarrow{\cong} J^{d-1}(E).
\end{equation}
On the other hand, the vanishing assumption on the cohomology of $X$ and its
normality imply that
$\H^{2d-2}_{cdh}\left(X, \Omega^{<d-1}_{X/{\C}}\right) = 0 =
\H^{2d-3}_{cdh}\left(X_{\rm sing}, \Omega^{<d-1}_{{X_{\rm sing}}/{\C}}\right)$.
This in turn implies from ~\eqref{eqn:CIJ} that the map
$J^{d-1}(\wt{X}) \to J^{d-1}(E)$ is surjective. The required surjectivity 
in ~\eqref{eqn:GEN*} now follows from Corollary~\ref{cor:GHC*1}.
\end{proof}

\begin{remk}\label{remk:extra*}
A conscious reader would observe that one does not need the
assumption $H^{d-1}_{cdh}\left(X, \sO_X\right) = 0$. This condition is
needed only for $\wt{X}$, which follows from the vanishing assumption
on the top cohomology groups and Hodge theory.
\end{remk}
  
\begin{cor}\label{cor:GENERAL*}
Let $X$ be a normal and projective variety of dimension $d$ over $\C$ 
with only isolated singularities such that
$H^d_{Zar}\left(X, \Omega^i_{X/{\C}}\right) = 0$ for $0 \le i \le d-2$.
Let $\wt{X} \to X$ be a resolution of singularities such that
the reduced exceptional divisor is smooth and
$H^{d-1}_{Zar}\left(\wt{X}, \Omega^i_{{\wt{X}}/{\C}}\right) = 0$ for 
$1 \le i \le d-2$. 
Assume that $GBC(d-1)$ and $GHC(d-2,2d-3, \wt{X})$ hold. 
Then $CH^d(X) \xrightarrow{\cong} CH^d(\wt{X})$.
\end{cor}
\begin{proof}
Follows from Theorems~\ref{thm:GENERAL}, ~\ref{thm:main1}
and Remark~\ref{remk:extra*} as in the proof of 
Corollary~\ref{cor:3fold-main2}.
\end{proof} 

The following improves \cite[Theorem~1.5]{Krishna2} for the 
Chow group of zero-cycles on projective cones.
\begin{cor}\label{cor:CONE}
Let $Y \inj {\P}^N_{\C}$ be a smooth projective variety of dimension $d$ and
let $X$ be the projective cone over $Y$. Assume that 
$H^{d+1}_{Zar}\left(X, \Omega^i_{X/{\C}}\right) = 0$ for $0 \le i \le  d-1$. 
Then $CH^{d+1}(X)$ is finite-dimensional if $CH^{d}(Y)$ is so.
\end{cor}

\begin{proof}
Let $\wt{X} \xrightarrow{f} X$ be the blow-up of $X$ at the vertex. Then 
$\wt{X} \xrightarrow{p} Y$ is a $\P^1$-bundle, and $f$ is a resolution of
singularities of $X$. Moreover, $E \inj \wt{X}$ maps isomorphically to
$Y$ under the map $p$. It is also easy to see that $CH^{d}(Y) \cong
CH^{d+1}(\wt{X})$ and $J^d(Y) \cong J^{d+1}(\wt{X})$.
Hence we only need to show that
$CH^{d+1}(X) \xrightarrow{\cong} CH^{d+1}(\wt{X})$. 
We only have to show the injectivity as the surjectivity is already known.
But this follows from the exact sequence 
\[
H^d_{cdh}\left(\wt{X}, \sK_{d+1}\right) \to
H^d_{cdh}\left(E, \sK_{d+1}\right) \to 
H^{d+1}_{cdh}\left(X, \sK_{d+1}\right)  
\to H^{d+1}_{cdh}\left(\wt{X}, \sK_{d+1}\right) \to 0,
\] 
the fact that $E \inj \wt{X}$ has a section, Corollary~\ref{cor:compare}
and Theorem~\ref{thm:main1}.
\end{proof}
% Notice that our proof does not require normality of $X$. We only need to
% know the singular locus is 0-dimensional.

\noindent\emph{Acknowledgments.} 
The final parts of this work were done
while the author was visiting Universit{\'e} Paris-Sud 11 at Orsay
in February 2010. The author would like to thank the department
for the invitation and financial support. The author would also like
to thank C. Soul{\'e} for useful discussion about the gamma filtrations
on higher $K$-theory.


\begin{thebibliography}{99}
\bibitem{Grothendieck}  M. Artin, A. Grothendieck, J. Verdier,  
{\sl Th{\'e}ori{\'e} des topos et cohomologie {\'e}tale\/},
SGA4, Lecture Notes in Mathematics, {\bf 270}, Springer-Verlag, (1972).
\bibitem{Viale}  L. Barbieri-Viale,  {\sl Zero-cycles on singular varieties:
torsion and Bloch's formula\/}, J. Pure Appl. Alg., {\bf 78}, (1992),
1-13.
\bibitem{BPW}  L. Barbieri-Viale, C. Pedrini, C. Weibel,
{\sl Roitman's Theorem for singular Varieties\/},
Duke Math. J., {\bf 84}, (1996), 155-190.
\bibitem{Beilinson} A. Beilinson, {\sl Higher regulators and values of 
$L$-functions\/}, Current problems in mathematics, {\bf 24},  181-238, 
\bibitem{BiswasS}  J. Biswas, V. Srinivas, {\sl Roitman's theorem for singular 
projective varieties\/}, Compositio Math.,  {\bf 119},
(1999),  no. 2, 213-237.
\bibitem{Bloch} S. Bloch, {\sl Algebraic cycles and higher $K$-theory\/},
Adv. in Math.,  {\bf 61},  (1986),  no. 3, 267-304.
\bibitem{Bloch*} S. Bloch, {\sl $K_2$ of artinian $\Q$-algebras with
applications to zero cycles\/}, Comm. Algebra, {\bf 3}, (1975), 405-428.
\bibitem{BlochO} S. Bloch, A. Ogus, {\sl Gersten's conjecture and the 
homology of schemes\/},  Ann. Sci. {\'e}cole Norm. Sup., (4),  {\bf 7},
(1974), 181-201.
\bibitem{BKL} S. Bloch, A. Kas, D.Lieberman, {\sl Zero cycles on surfaces
with $p_g = 0$\/}, Compositio Math., {\bf 33}, (1976), 135-145.
\bibitem{CGW}  G. Cortinas, S. Geller, C. Weibel,  {\sl 
Artinian Berger's Conjecture\/}, Math. Zeitschrift, {\bf 228}, (1998),
569-588.
\bibitem{CAnnals}  G. Cortin±as, C. Haesemeyer, M. Schlichting,
C. Weibel, {\sl Cyclic homology, cdh-cohomology and negative $K$-theory\/},
Ann. of Math., {\bf 167}, (2008),  no. 2, 549-573. 
\bibitem{CJams}  G. Cortin±as, C. Haesemeyer,
C. Weibel, {\sl $K$-regularity, $cdh$-fibrant Hochschild homology, and a 
conjecture of Vorst\/},  J. Amer. Math. Soc.,  {\bf 21},
(2008),  no. 2, 547-561. 
\bibitem{CMAnnalen}  G. Cortinas, C. Haesemeyer,
C. Weibel, {\sl Infinitesimal cohomology and the Chern character to negative 
cyclic homology\/},  Math. Ann.,  {\bf 344},  (2009),  no. 4, 891-922. 
\bibitem{BassNK}  G. Cortinas, C. Haesemeyer, M. Walker,
C. Weibel, {\sl Bass' $NK$ groups and $cdh$-fibrant Hochschild homology\/},
Available at arxiv:0802-1928.
\bibitem{C-cone}  G. Cortinas, C. Haesemeyer, M. Walker,
C. Weibel, {\sl $K$-theory of cones of smooth varieties\/},
Available at arXiv:0905.4642.
\bibitem{Hodge2} P. Deligne, {\sl Th{\'e}ori{\'e} de Hodge. II\/},
Inst. Hautes {\'E}tudes Sci. Publ. Math., {\bf 40}, (1971), 5-57.
\bibitem{Hodge3} P. Deligne, {\sl Th{\'e}ori{\'e} de Hodge. III\/},
Inst. Hautes {\'E}tudes Sci. Publ. Math., {\bf 44}, (1974), 5-77.
\bibitem{DB} Ph. Du Bois, {\sl Complexe de de Rham filtr{\'e} d'une
vari{\'e}t{\'e} singuli{\'e}re\/},
Bull. SMF, {\bf 109}, (1981), 41-81.
\bibitem{EV} H. Esnault, E. Viehweg, {\sl Deligne-Beilinson cohomology\/},
BeÄ­linson's conjectures on special values of $L$-functions,  43-91, 
Perspect. Math., {\bf 4}, Academic Press, Boston, MA, 1988. 
\bibitem{EsnaultV}  H. Esnault,  E. Viehweg,   {\sl Lectures on Vanishing
Theorems\/}, DMV Seminar, {\bf 20}, Birkhauser Verlag, Basel, 1992.
\bibitem{ESV}  H. Esnault, V. Srinivas, E. Viehweg,  
{\sl The Universal regular quotient of Chow group of points on projective 
varieties\/}, Invent. Math., {\bf 135}, (1999), 595-664.
\bibitem{FultonL} W. Fulton, S. Lang, {\sl Riemann-Roch algebra\/},
Grundlehren der Mathematischen Wissenschaften 
{\bf 277}, Springer-Verlag, New York, 1985. 
\bibitem{GellerW} S. Geller, C. Weibel, {\sl $K_{1}(A,\,B,\,I)$\/},
J. Reine Angew. Math.,  {\bf 342},  (1983), 12-34. 
\bibitem{Gillet} H. Gillet, {\sl Riemann- Roch theorems for higher algebraic 
$K$-theory\/}, Adv. in Math.,  {\bf 40},  (1981), no. 3, 203-289.
\bibitem{GilletS} H. Gillet, C. Soul{\'e}, {\sl Filtrations on higher 
algebraic $K$-theory\/},
Algebraic $K$-theory (Seattle, WA, 1997),  89-148, 
Proc. Sympos. Pure Math., {\bf 67}, Amer. Math. Soc., Providence, RI, 1999.
\bibitem{SGA6} A. Grothendieck, {\sl Th{\'e}ori{\'e} des intersections et
th{\'e}oreme de Riemann-Roch\/},
SGA 6, Lecture notes in Mathematics, {\bf 225}, Springer-Verlag. 
\bibitem{GuillenN} F. Guill{\'e}n,  V. Navarro Aznar, {\sl Un critère 
d'extension des foncteurs d{\'e}finis sur les sch{\'e}mas lisses\/},
Publ. Math. Inst. Hautes {\'e}tudes Sci.,  {\bf 95},  (2002), 1-91. 
\bibitem{Haes} C. Haesemeyer, {\sl Descent properties of homotopy 
$K$-theory\/},  Duke Math. J.,  {\bf 125}, (2004),  no. 3, 589-620. 
\bibitem{Hart} R. Hartshorne, {\sl Algebraic Geometry\/},
Springer-Verlag, 1977.
\bibitem{Kerz} M. Kerz, {\sl The Gersten conjecture for Milnor $K$-theory\/},
Invent. Math., {\bf 175}, (2009), 1-33.
\bibitem{KK} J. Kollar, S. Kov{\'a}cs,  {\sl Log canonical singularities are
Du Bois\/}, Available at arxiv:0902.0648.
\bibitem{Krishna0} A. Krishna, {\sl On $K_2$ of one-dimensional local rings\/},
$K$-Theory,  {\bf 35}, (2005),  no. 1-2, 139-158.
\bibitem{Krishna3} A. Krishna, {\sl Zero-cycles on a threefold with isolated 
singularities\/},  J. Reine Angew. Math.,  594  (2006), 93-115. 
\bibitem{Krishna1} A. Krishna, {\sl Zero-cycles on singular surfaces\/},
J. K-Theory, {\bf 4}, (2009), 101-143.
\bibitem{Krishna2} A. Krishna, {\sl An Artin-Rees theorem and applications to
zero cycles\/}, J. Algebraic Geometry, DOI:S 1056-3911(09)00521-9. 
\bibitem{Krishna4} A. Krishna, {\sl Zero-cycles and $cdh$ cohomology II\/},
In preparation, 2010. 
\bibitem{KrishnaS} A. Krishna, V. Srinivas, {\sl Zero-Cycles and $K$-theory
on Normal Surfaces\/},  Ann. of Math., {\bf 156}, (2002), 155-195.
\bibitem{LeahyV} V. Leahy, A. Vitulli, {\sl Weakly normal varieties: 
the multicross singularity and some vanishing theorems on local cohomology\/},
Nagoya Math. J.,  {\bf 83},  (1981), 137-152. 
\bibitem{Levine3} M. Levine, {\sl Bloch's formula for singular surfaces\/},
Topology,  {\bf 24},  (1985),  no. 2, 165-174. 
\bibitem{Levine1}  M. Levine, {\sl Zero-Cycles and $K$-theory on Singular
Varieties\/}, Proc. Symp. Pure Math., AMS Providence, {\bf 46}, (1987), 
451-462. 
\bibitem{Levine2} M. Levine, {\sl Deligne-Beilinson cohomology for singular 
varieties\/},  
Algebraic $K$-theory, commutative algebra, and algebraic geometry 
(Santa Margherita Ligure, 1989),  113-146, Contemp. Math., {\bf 126}, 
Amer. Math. Soc., Providence,
\bibitem{LevineW}  M. Levine, C. Weibel, {\sl Zero-cycles and complete
intersections on singular varieties\/}, {\bf 359}, (1985), 106-120.  
\bibitem{Lewis} J.Lewis, {\sl A survey of the Hodge conjecture\/}
Second edition, CRM Monograph Series, {\bf 10}. 
American Mathematical Society, Providence, RI, 1999. 
 \bibitem{Loday}  J-L Loday,  {\sl Cyclic Homology\/}, Grund. der math.
Wissen. Series, {\bf 301}, Springer Verlag, 1998.
\bibitem{PP}  P. Pascual, L. Pons,  {\sl Algebraic $K$-theory and 
cubical descent\/},  Homology, Homotopy Appl.,  {\bf 11},
 (2009),  no. 2, 5-25. 
\bibitem{VistoliV} D. Quillen, {\sl Algebraic $K$-theory\/},
Algebraic $K$-Theory I, Seatle, 1972, Lecture Notes in Mathematics, {\bf 341},
Springer-Verlag.
 \bibitem{Schwede} K. Schwede, {\sl A simple characterization of Du Bois 
singularities\/},  Compos. Math.  {\bf 143}  (2007),  no. 4, 813-828.
\bibitem{Srinivas2} V. Srinivas, {\sl Zero-cycles on a singular surface. II\/},
J. Reine Angew. Math., {\bf 362}, (1985), 4-27.
\bibitem{Srinivas1} V. Srinivas, {\sl Zero-cycles on singular varieties\/},
The arithmetic and geometry of algebraic cycles (Banff, AB, 1998),  
347-382, NATO Sci. Ser. C Math. Phys. Sci., {\bf 548}, 
Kluwer Acad. Publ., Dordrecht, 2000. 
\bibitem{Soule} C. Soul{\'e}, {\sl Op{\'e}rations en $K$-th{\'e}ori{\'e}
algebrique\/}, Can. J. Math., {\bf 37}, (1985), 488-550.  
\bibitem{Steenbrink} J. Steenbrink, {\sl Du Bois invariants of isolated 
complete intersection singularities\/},
Ann. Inst. Fourier (Grenoble), {\bf 47}, (1997), no. 5, 1367-1377.
\bibitem{SuslinV} A. Suslin, V. Voevodsky, {\sl Bloch-Kato conjecture and 
motivic cohomology with finite coefficients\/},
The arithmetic and geometry of algebraic cycles (Banff, AB, 1998),  
117-189, NATO Sci. Ser. C Math. Phys. Sci., {\bf 548}, 
Kluwer Acad. Publ., Dordrecht, 2000. 
\bibitem{ThomasonT} R. W. Thomason, T. Trobaugh, {\sl Higher Algebraic
K-Theory Of Schemes And Of Derived Categories\/},
The Grothendieck Festschrift III, Progress in Math. 88, Birkhauser.
\bibitem{Voisin} C. Voisin, {\sl Remarks on filtrations on Chow groups
and the Bloch conjecture\/}, Ann. Mat. Pura Appl., (4), {\bf 183},
(2004), 421-438. 
\bibitem{WeibelH} C. Weibel, {\sl Homotopy algebraic $K$-theory\/},
Algebraic $K$-theory and algebraic number theory (Honolulu, HI, 1987),  
461-488, Contemp. Math., {\bf 83}, Amer. Math. Soc., Providence, RI, 1989.




\end{thebibliography}
\end{document}